\documentclass[a4paper,12pt]{report}

\usepackage{amsmath,amsfonts,amssymb}
\usepackage{amsthm}
\usepackage{color,multicol}
\usepackage{graphicx}

\newtheorem{lemma}{Lemma}[section]
\newtheorem{theorem}[lemma]{Theorem}
\newtheorem{corollary}[lemma]{Corollary}
\newtheorem{proposition}[lemma]{Proposition}
\newtheorem{definition}[lemma]{Definition}
\newtheorem{example}[lemma]{Example}

\newtheorem{remark}[lemma]{Remark}

\begin{document}

\title{Estimates for the number of eigenvalues of two dimensional Schr\"odinger operators lying below the essential spectrum}
\vspace{2cm}
\vspace {1cm}
\author{\\Martin Karuhanga \vspace{2cm}\\
 A thesis submitted for the degree of Doctor of Philosophy \\of King's College London \vspace{5.1cm}\\}

\date{September, 2016}
\maketitle
%\pagestyle{myheadings}
% \pagenumbering{arabic}
%\thispagestyle{empty}
\cleardoublepage
 \thispagestyle{empty}
 \setcounter{page}{2}
\chapter*{Acknowledgments}
My PhD studies have been supported by the Commonwealth Scholarship Commission in the UK funded by the UK government. I am exceptionally appreciative of this support. \\\\I am extremely grateful to my supervisor, Eugene Shargorodsky for the direction and guidance he has provided during my studies. Suffice  to say, this thesis would not have been produced without his superb guidance.\\\\
Finally, I would like to extend my heartfelt thanks to my family for their constant encouragement throughout this endeavour.
\chapter*{Declaration}
I hereby declare that the work
in this thesis is my own, where I have used materials from other sources,
references have been made and that this work has never been
submitted to any University or College for any award.\\\\\\\\\\
Signature\\\\
...................................\\\\
Date ..............................\\\\\\

 \chapter*{Abstract}
The celebrated Cwikel-Lieb-Rozenblum inequality gives an upper estimate for the number of negative eigenvalues of Schr\"odinger operators in  dimension three and higher. The situation is much more difficult in the two dimensional case. There has been significant progress in obtaining  upper estimates for the number of negative eigenvalues of two dimensional Schr\"odinger operators on the whole plane. In this thesis, we present upper estimates of the Cwikel-Lieb-Rozenblum type for the number of eigenvalues (counted with multiplicities) of two dimensional Schr\"odinger operators lying below the essential spectrum in terms of the norms of the potential. The problem is considered on the whole plane with different supports of the potential (in particular, sets of dimension $\alpha \in (0, 2]$ and on a strip with  various boundary conditions. In both cases, the estimates involve weighted $L^1$ norms and Orlicz norms of the potential.
\thispagestyle{plain} \markboth{Contents}{Contents}
\maketitle
\tableofcontents
 %\pagenumbering{arabic}

\chapter{Introduction}\markboth{Chapter \ref{Introduction}.
Introduction}{}\label{Introduction}
\section{Background}\label{background}
\markright{\ref{Introsub}. Introduction}
Given a non-negative $L^1_{\textrm{loc}}$ function $V$ on $\mathbb{R}^n$, consider the Schr\"odinger type operator on $L^2(\mathbb{R}^n)$
\begin{equation}\label{1}
H_V := -\Delta - V, \;\;\;\;\;\;\;\; V \geq 0,
\end{equation}where $\Delta := \sum^n_{k = 1}\frac{\partial^2}{\partial x^2_k}$. Precisely, one can define $H_V$ via the quadratic form
\begin{eqnarray*}
\mathcal{E}_{V, \mathbb{R}^n}[u] &=& \int_{\mathbb{R}^n}|\nabla u(x)|^2\,dx - \int_{\mathbb{R}^n}V(x)|u(x)|^2\,dx ,\\ \textrm{Dom}(\mathcal{E}_{V, \mathbb{R}^n}) &=& \left\{u\in W^1_2(\mathbb{R}^n)\cap L^2(\mathbb{R}^n, V(x)dx)\right\}.
\end{eqnarray*}
Under certain assumptions about $V$, $H_V$ is well defined and a self-adjoint operator on $L^2(\mathbb{R}^n)$ and its essential spectrum  is $[0, \infty)$ (see e.g.,  \cite[$\S$ 6.4]{Teschl}). The negative spectrum of $H_V$ consists of eigenvalues of finite multiplicity with zero as the only possible accumulation point. A general problem arising from Physics is to estimate the number of these negative eigenvalues (counted with multiplicities) in terms of the norms of $V$ which is the subject of this thesis for $n = 2$. Denote by $N_-(\mathcal{E}_{V, \mathbb{R}^n})$ the number of negative eigenvalues of $H_V$ counted with multiplicity.\\\\  For $n\geq 3$, according to the celebrated Cwikel (1977)-Lieb (1976)-Rozenblum (1972) inequality, $N_-(\mathcal{E}_{V, \mathbb{R}^n})$ is estimated above by
\begin{equation}\label{CLR}
N_-(\mathcal{E}_{V, \mathbb{R}^n})\le C_n\int_{\mathbb{R}^n}V(x)^{n/2}\,dx
\end{equation}(see, e.g., \cite{BE}, \cite{BEL}, \cite{Roz} and the references therein).
In this case ($n\geq 3$), $N_-(\mathcal{E}_{V, \mathbb{R}^n})$ is zero provided the integral in the right hand side of \eqref{CLR} is small enough.

Often one inserts a parameter $\alpha > 0$ called the coupling constant and studies the behaviour of $N_-(\mathcal{E}_{\alpha V, \mathbb{R}^n})$ as $\alpha \longrightarrow +\infty$. For nice potentials, e.g $V\in C_0^{\infty}({\mathbb{R}^n})$, the Weyl-asymptotic formula
\begin{equation}\label{Weyl1}
\underset{\alpha \to +\infty}\lim\alpha^{-n/2} N_-(\mathcal{E}_{\alpha V, \mathbb{R}^n}) = \frac{\textrm{Vol}\{x\in\mathbb{R}^n : |x| \le 1\}}{(2\pi)^n}\int_{\mathbb{R}^n}V(x)^{n/2}dx
\end{equation}holds (see, e.g., \cite[Theorem 5.1]{BirLap}). We say that an estimate is semi-classical, if it yields
\begin{equation}\label{order}
N_-(\mathcal{E}_{\alpha V, \mathbb{R}^n}) = O\left(\alpha^{n/2}\right) \;\;\textrm{as} \;\; \alpha \to +\infty.
\end{equation} The conditions guaranteeing \eqref{Weyl1} and \eqref{order} depend on the dimension.
If $V\in L^{n/2}(\mathbb{R}^n), \; n\geq 3,$ \eqref{CLR} is optimal, that is, $V\in L^{n/2}(\mathbb{R}^n)$ is a necessary and sufficient condition for \eqref{order} and \eqref{Weyl1} to hold (see, e.g., \cite{Roz}). This shows that $N_-(\mathcal{E}_{\alpha V, \mathbb{R}^n})$ is estimated through its own asymptotics.\\\\
For $n = 1$, there is no analogue of \eqref{CLR}. It does only exist for potentials that are monotone on $\mathbb{R}_+$ and $\mathbb{R}_-$. This was obtained simultaneously and independently in 1965 by F. Calogero \cite{Calo} and J.H.E. Cohn \cite{Cohn}. Finitiness of the right-hand side of \eqref{Weyl1} for $n= 1$ is only sufficient for $N_-(\mathcal{E}_{\alpha V, \mathbb{R}^1}) = O\left(\sqrt{\alpha}\right)$. The necessary and sufficient condition for the latter is given in terms of the ``weak $l_1$ -space'' (see, e.g., \cite{Sol2}). For any non-trivial potential $N_-(\mathcal{E}_{ V, \mathbb{R}^1}) \geq 1$.\\\\

Consider the following integral
\begin{equation}\label{int}
\int_{\mathbb{R}^n} V(x)|u(x)|^2\,dx\,,\,\,\,\,\,u \in W^1_2(\mathbb{R}^n).
\end{equation}
It follows from the Sobolev embedding theorem (see, e.g., \cite[Theorem 5.4]{Ad}) that $W^1_2(\mathbb{R}^n)\hookrightarrow L^q(\mathbb{R}^n)$ if $ 2 \le q < \infty$ and $1 - \frac{n}{2} + \frac{n}{q} \ge 0$. Let $$1 - \frac{n}{2} + \frac{n}{q} = 0. \,\,\mbox{ Then } q = \frac{2n}{n -2}\,,\,\,\,n \ge 3.$$ Thus $u \in W^1_2(\mathbb{R}^n)$ implies $u \in L^{\frac{2n}{n-2}}(\mathbb{R}^n)$ and $|u|^2 \in L^{\frac{n}{n -2}}(\mathbb{R}^n)$. Let $p = \frac{n}{n - 2}$\,. Then the H\"older inequality implies that \eqref{int} is finite if $V \in L^{p'}(\mathbb{R}^n)$, where $p'$ is the conjugate exponent  of $p$ i.e., $\frac{1}{p} + \frac{1}{p'} = 1$. We have $p' = \frac{n}{2}$ implying that \eqref{int} is finite if $V \in L^{\frac{n}{2}}(\mathbb{R}^n)$, which explains why the right hand side of \eqref{CLR} might work.\\\\

 Putting $n = 2$ in the above argument formally, one gets $q = \infty, \, p = \infty $, \\$p' = 1$ and $V \in L^1(\mathbb{R}^2)$. Unfortunately, $W^1_2(\mathbb{R}^2)$ is not embedded in $L^{\infty}(\mathbb{R}^2)$, $V \in L^1(\mathbb{R}^2)$ does not guarantee that \eqref{int} is finite, and \eqref{CLR} fails. However, $W^1_2(\mathbb{R}^n)\hookrightarrow L^q(\mathbb{R}^n),\, \forall q \in [2, +\infty)$ and there are estimates for $N_-(\mathcal{E}_{ V, \mathbb{R}^2})$ involving $\int_{\mathbb{R}^2} |V(x)|^r\,dx ,\,\forall r > 1$ (see e.g.,\cite{BirLap}, \cite{BS}, \cite{Grig}, \cite{LapSolo}). More precisely, $W^1_2(\mathbb{R}^2)$ is embedded in a space of exponentially integrable functions which sits between $L^1(\mathbb{R}^2)$ and $L^{p}(\mathbb{R}^2),\, p > 1$. This gives rise to estimates for $N_-(\mathcal{E}_{ V, \mathbb{R}^2})$ involving a norm of $V$ weaker than $\|V\|_{L^r}\,,r > 1$, namely, the Orlciz $L\log L$ norm (see $\S$ \ref{Orliczspaces}).

In addition, $L^1$-integrability doesn't provide enough decay of $V$ at infinity and thus one needs stronger weighted $L^1$ norms of $V$ with weights growing as $|x|\longrightarrow \infty$. Logarithmic growth of the weights is enough in $\mathbb{R}^2$, and $|x_1|$ is sufficient in case of the strip. Most known upper estimates for $N_-(\mathcal{E}_{V, \mathbb{R}^2})$ have terms of these two types i.e., weighted $L^1$ and $L^r\,, r > 1$ norms of $V$, see for example the estimates obtained by A. Laptev  and M. Solomyak \cite{LapSolo} and M. Solomyak \cite{Sol}. In this thesis, we present estimates of a similar structure.\\\\
For any non-zero potential $V \ge 0$ and $n =2$, \eqref{1} has at least one negative eigenvalue  as in the one-dimensional case and therefore \eqref{CLR} cannot hold. More importantly, no estimate of the type

$$
N_- (\mathcal{E}_{V,\mathbb{R}^2}) \le \mbox{const} + \int_{\mathbb{R}^2} V(x)  W(x)\, dx
$$
can hold, provided the weight function $W$ is bounded in a neighborhood of at least
one point (see, e.g., \cite[Proposition 2.1]{Grig}). On the other hand,
$$
N_- (\mathcal{E}_{V, \mathbb{R}^2}) \ge  \mbox{const}\, \int_{\mathbb{R}^2} V(x)\, dx.
$$This result is due to A. Grigor'yan, Yu. Netrusov  and S.-T. Yau
\cite{GNY} in 2004.  It is well known that the lowest possible (semi-classical) rate of growth of $N_- (\mathcal{E}_{\alpha V,\mathbb{R}^2})$ is
\begin{equation}\label{order1}
N_-(\mathcal{E}_{\alpha V, \mathbb{R}^2}) = O\left(\alpha\right) \;\;\textrm{as} \;\; \alpha \to +\infty
\end{equation}( see e.g., \cite{BS},  \cite{LapSolo}, \cite{Sol}).\\This agrees with the Weyl-asymptotic formula
\begin{equation}\label{Weyl2}
\underset{\alpha \to +\infty}\lim\alpha^{-1} N_-(\mathcal{E}_{\alpha V, \mathbb{R}^2}) = \frac{1}{4\pi}\int_{\mathbb{R}^2}V(x)dx
\end{equation}that is satisfied if the potential is nice. However, unlike the case $n \geq 3$, \eqref{order1} doesn't guarantee the existence of $\underset{\alpha \to +\infty}\lim\alpha^{-1} N_-(\mathcal{E}_{\alpha V, \mathbb{R}^2})$ and even if the limit exists, it may be different from the right hand side of \eqref{Weyl2}. The exhaustive description of the classes of potentials on $\mathbb{R}^2$ such that \eqref{order1} or \eqref{Weyl2} is satisfied, is unknown till now.  More upper estimates for $N_-(\mathcal{E}_{V, \mathbb{R}^2})$ can be found in (\cite{GNY}, \cite{Eugene}, \cite{Sol}, \cite{Sol2}) and the references therein. Estimates for the number of negative eigenvalues of two dimensional magnetic Schr\"odinger operators can be found for example in (\cite{BE}, \cite{BEL}, \cite{Bal}, \cite{Kov}) and the references therein. However, such type of operators are not considered in this thesis.\\\\

E. Shargorodsky \cite{Eugene1}, in his paper on an estimate for the  Morse index of a Stokes wave obtained an estimate for $N_-(\mathcal{E}_{V, \mathbb{R}^2)}$ with $V$  supported by a bounded Lipschitz curve. Initially, an aim of this thesis was to extend this result to unbounded Lipschitz curves which is by no means trivial since it involves decay at infinity. However, we managed to do this in a more general setting that covers potentials locally integrable on $\mathbb{R}^2$, potentials supported by curves and sets of fractional dimension $\alpha \in (0, 2]$. This is the work of Chapter 3.\\\\We also consider the problem on a strip. Previously, A. Grigor'yan and N. Nadirashvili \cite{Grig} considered this problem on strip with Neumann boundary conditions and obtained estimates in terms of weighted $L^1$ and $L^p, p > 1$ norms of $V$. We consider the case of Robin boundary conditions (including Dirichlet and Dirichlet-Neumann) and obtain stronger estimates, in particular, estimates involving $L\log L$ norms of $V$ including norms of $V$ supported by sets of fractional dimension. This is the work of Chapter 4.

 \section{Structure of the thesis}
In Chapter 1, we give a brief background to the problem and  discuss the variational approach introduced by M. Sh. Birmann and M. Z. Solomyak \cite{BS} that we use in obtaining our estimates. We review the theory and results on Orclicz spaces that we use in the sequel. We also review the basic facts in spectral theory and discuss in detail the spectrum of the Laplacian on a strip with various boundary conditions.   \\\\
In Chapter 2, we give a review of some of the known results and present the necessary auxiliary results. We extend the estimates for the number of negative eigenvalues of one-dimensional Schr\"odinger operators with potentials locally integrable on $\mathbb{R}$ (see, e.g., \cite{Sol2}) to a general class of measures including measures with atoms.    \\\\
In Chapter 3, we obtain estimates for the number of negative eigenvalues of two-dimensional Schr\"odinger operators with potentials generated by Ahlfors regular measures of dimension $\alpha \in (0, 2]$. Our estimates involve weighted $L^1$ norms and Orlicz norms of the potential.\\\\
In Chapter 4, we first consider the problem on a strip  with Neumann boundary conditions. Later, we consider the case of Robin boundary conditions and derive Neumann-Robin and Dirichlet-Robin conditions as particular cases. We present upper estimates for the number of eigenvalues (not all of them necessarily negative) of the operator \eqref{1} lying below the bottom of the essential spectrum. In both cases the estimates involve  weighted $L^1$ and $L\log L$ norms of the potential.

\section{Orlicz spaces}\label{Orliczspaces}
Let $(\Omega, \Sigma, \mu)$ be a measure space and let $\Psi : [0, +\infty) \rightarrow [0, +\infty)$ be a non-decreasing function. The Orlicz class $K_{\Psi}(\Omega)$ is the set of all (equivalence classes modulo equality a.e. in $\Omega$ of) measurable functions $f : \Omega \rightarrow \mathbb{C}\;( \textrm{or}\;\mathbb{R})$ such that
\begin{equation}\label{orliczeqn}
\int_{\Omega}\Psi(|f(x)|)d\mu(x) < \infty\,.
\end{equation} If $\Psi(t) = t^p,\; 1\le p < \infty$, this is just the $L^p(\Omega)$ space. The difficulty here is that the set of all functions satisfying \eqref{orliczeqn} is not necessarily a linear space. If $\Psi$ is rapidly increasing, e.g exponentially increasing, \eqref{orliczeqn} doesn't imply that the same integral for $2f$ is finite.
\begin{definition}
{\rm A continuous non-decreasing convex function $\Psi : [0, +\infty) \rightarrow [0, +\infty)$ is called an $N$-function if
$$
\underset{t \rightarrow 0+}\lim\frac{\Psi (t)}{t} = 0 \;\;\; \textrm{and }\;\;\;\underset{t \rightarrow \infty}\lim\frac{\Psi (t)}{t} = \infty.
$$ The function $\Phi : [0, +\infty) \rightarrow [0, +\infty)$ defined by
$$
\Phi(t) := \underset{s\geq 0}\sup\left(st - \Psi(s)\right)
$$ is called complementary to $\Psi$.
}
\end{definition}Examples of complementary functions include:
\begin{eqnarray*}
&&\Psi(t) = \frac{t^p}{p},\;\;1 < p < \infty,\;\;\;\;\Phi(t) = \frac{t^q}{q}, \;\;\frac{1}{p} + \frac{1}{q} = 1,\\&&\mathcal{A}(s) = e^{|s|} - 1 - |s| , \ \ \ \mathcal{B}(s) = (1 + |s|) \ln(1 + |s|) - |s| , \ \ \ s \in \mathbb{R} .
\end{eqnarray*}
\begin{definition}
{\rm An $N$-function $\Psi$ is said to satisfy a global $\Delta_2$-condition if there exists a positive constant $k$ such that for every $t \geq 0$,
\begin{equation}\label{global}
\Psi(2t)\le k\Psi(t).
\end{equation}
Similarly $\Psi$ is said to satisfy a $\Delta_2$-condition near infinity if there exists $t_0 > 0$ such that \eqref{global} holds for all $t \geq t_0$.}
\end{definition}
\begin{definition}
{\rm We call the pair $(\Psi, \Omega)\;\Delta$-regular if either $\Psi$ satisfies a global $\Delta_2$-condition, or $\Psi$ satisfies a $\Delta_2$-condition near infinity and $\mu(\Omega) < \infty$.}
\end{definition}
\begin{lemma}{\rm (\cite[Lemma 8.8]{Ad})}
$K_{\Psi}(\Omega)$ is a vector space if and only if $(\Psi, \Omega)$ is $\Delta$-regular.
\end{lemma}
\begin{definition}
{\rm The \textit{Orlicz space} $ L_{\Psi}(\Omega)$ is the linear span of the Orlicz class $K_{\Psi}(\Omega)$, that is, the smallest vector space containing $K_{\Psi}(\Omega)$.}
\end{definition}Consequently, $K_{\Psi}(\Omega) = L_{\Psi}(\Omega)$ if and only if $(\Psi, \Omega)$ is $\Delta$-regular.\\\\
Let $\Phi$ and $\Psi$ be mutually complementary $N$-functions, and let $L_\Phi(\Omega)$,
$L_\Psi(\Omega)$ be the corresponding Orlicz spaces. (These spaces are denoted by
$L^*_\Phi(\Omega)$,  $L^*_\Psi(\Omega)$ in \cite{KR}, where $\Omega$ is assumed to be a closed
bounded subset of $\mathbb{R}^n$
equipped with the standard Lebesgue measure.) We will use the following
norms on $L_\Psi(\Omega)$
\begin{equation}\label{Orlicz}
\|f\|_{\Psi} = \|f\|_{\Psi, \Omega} = \sup\left\{\left|\int_\Omega f g d\mu\right| : \
\int_\Omega \Phi(|g|) d\mu \le 1\right\}
\end{equation}
and
\begin{equation}\label{Luxemburg}
\|f\|_{(\Psi)} = \|f\|_{(\Psi, \Omega)} = \inf\left\{\kappa > 0 : \
\int_\Omega \Psi\left(\frac{|f|}{\kappa}\right) d\mu \le 1\right\} .
\end{equation}
These two norms are equivalent
\begin{equation}\label{Luxemburgequiv}
\|f\|_{(\Psi)} \le \|f\|_{\Psi} \le 2 \|f\|_{(\Psi)}\, , \ \ \ \forall f \in L_\Psi(\Omega),
\end{equation}(see \cite{Ad}).\\
Note that
\begin{equation}\label{LuxNormImpl}
\int_\Omega \Psi\left(\frac{|f|}{\kappa_0}\right) d\mu \le C_0, \ \ C_0 \ge 1  \ \ \Longrightarrow \ \
\|f\|_{(\Psi)} \le C_0 \kappa_0 .
\end{equation}
Indeed, since $\Psi$ is  convex and increasing on
$[0, +\infty)$, and $\Psi(0) = 0$, we get for any $\kappa \ge C_0 \kappa_0$,
\begin{equation}\label{LuxProof}
\int_{\Omega} \Psi\left(\frac{|f|}{\kappa}\right) d\mu \le
\int_{\Omega} \Psi\left(\frac{|f|}{C_0 \kappa_0}\right) d\mu \le
\frac{1}{C_0} \int_{\Omega} \Psi\left(\frac{|f|}{\kappa_0}\right) d\mu \le 1 .
\end{equation}
It follows from \eqref{LuxNormImpl} with $\kappa_0 = 1$ that
\begin{equation}\label{LuxNormPre}
\|f\|_{(\Psi)} \le \max\left\{1, \int_{\Omega} \Psi(|f|) d\mu\right\} .
\end{equation}
We will need the following
equivalent norm on $L_\Psi(\Omega)$ with $\mu(\Omega) < \infty$ which was introduced in
\cite{Sol}:
\begin{equation}\label{OrlAverage}
\|f\|^{\rm (av)}_{\Psi} = \|f\|^{\rm (av)}_{\Psi, \Omega} = \sup\left\{\left|\int_\Omega f g d\mu\right| : \
\int_\Omega \Phi(|g|) d\mu \le \mu(\Omega)\right\} .
\end{equation}

\begin{proposition}{\rm \cite[Formula (16)]{Ad}}
{\rm For any $f\in L_{\Psi}(\Omega)$ and $g\in L_{\Phi}(\Omega)$
\begin{equation}\label{Holder}
\left|\int_{\Omega}fg\,d\mu\right| \le \|f\|_{\Psi,\Omega}\|g\|_{\Phi,\Omega}.
\end{equation}
In particular, $fg\in L^1(\Omega)$. This is the H\"older inequality for Orlicz spaces.}
\end{proposition}The following are referred to as strengthened H\"older inequalities:
\begin{equation}\label{h1}
\left|\int_{\Omega}fg\,d\mu\right| \le \|f\|_{\Psi,\Omega}\|g\|_{(\Phi,\Omega)}
\end{equation}
and
\begin{equation}\label{h2}
\left|\int_{\Omega}fg\,d\mu\right| \le \|f\|_{(\Psi,\Omega)}\|g\|_{\Phi,\Omega}\,,
\end{equation}
for all $f\in L_{\Psi}(\Omega)$ and $g\in L_{\Phi}(\Omega)$ (see (9.26) and (9.27) respectively in \cite{KR}).
\begin{lemma}{\rm \cite[Lemma 1]{Sol}}
{\rm Let $\xi$ be an affine transformation of $\mathbb{R}^n, \; \Omega_{\xi} = \xi(\Omega)$ and $\mu$ be the Lebesgue measure. Then for any N-function $\Psi$ and for any
$f\in L_{\Psi}(\Omega_{\xi})$
\begin{equation}\label{scaling}
\mu(\Omega)^{-1}\|f\circ \xi\|^{(\textrm{av})}_{\Psi, \Omega} = \mu(\Omega_{\xi})^{-1}\|f\|^{(\textrm{av})}_{\Psi, \Omega_{\xi}}.
\end{equation}
That is, \eqref{OrlAverage} is invariant with respect to scaling.
}
\end{lemma}
\begin{lemma}\label{lemma7}{\rm  \cite[Lemma 3]{Sol}}
 For any finite collection of pairwise disjoint subsets $\Omega_k$ of $\Omega$
\begin{equation}\label{bsr1}
\sum_k\|f\|^{(av)}_{\Psi,\Omega_k} \le \|f\|^{(av)}_{\Psi,\Omega}.
\end{equation}
\end{lemma}
Below we provide a proof of the above Lemma in a more general form.
For $a >0$, let
\begin{equation}\label{norma}
\|f\|^{(a)}_{\Psi, \Omega} := \sup\left\{\left|\int_{\Omega}fg\,d\mu\right| : \int_{\Omega}\Phi(|g|)\,d\mu \le a\right\}.
\end{equation}
\begin{lemma}\label{a}
 Suppose $\Omega_1, \Omega_2, ... , \Omega_N \subseteq \Omega$ are such that $\Omega_j\cap\Omega_k = \emptyset,\;j\not= k$. Suppose there exists a sequence $a_1, a_2, ... , a_N \ge 0$ and $\kappa \ge 1$ such that
\begin{equation}\label{M1}
\sum_{k = 1}^Na_k \le \kappa a\,.
\end{equation} Then
\begin{equation}\label{M2}
\sum_{k=1}^N\|f\|^{(a_k)}_{\Psi, \Omega_k} \le \kappa \|f\|^{(a)}_{\Psi, \Omega}\,.
\end{equation}

\end{lemma}
\begin{proof}
Let $g = g_k$ on $\Omega_k$ and $0$ on $\Omega\backslash\cup_{k = 1}^N\Omega_k$. Then
$$
\sum_{k=1}^N\|f\|^{(a_k)}_{\Psi, \Omega_k} = \sum_{k =1}^N\sup\left\{\left|\int_{\Omega_k}fg_k\,d\mu\right| : \int_{\Omega_k}\Phi(|g_k|)\,d\mu \le a_k\right\}.
$$
Let $G = \sum_{k=1}^Ng_k$. Then
\begin{eqnarray*}
\int_{\Omega}\Phi(|G|)d\mu &=& \sum_{k=1}^N\int_{\Omega_k}\Phi(|G|)d\mu \\ &=&\sum_{k=1}^N\int_{\Omega_k}\Phi(|g_k|)d\mu\\ &\le& \sum_{k=1}^N a_k \le \kappa a\,.
\end{eqnarray*} There exists a $G$ such that
$$
\sum_{k=1}^N\|f\|^{(a_k)}_{\Psi, \Omega_k} \le \sup\left\{\left|\int_{\Omega}fG\,d\mu\right| : \int_{\Delta}\Phi(|G|)\,d\mu \le \kappa a\right\}.
$$
Let $h = \frac{1}{\kappa}G$. Since $\Phi$ is convex and $\Phi(0) = 0$, then
\begin{eqnarray*}
\int_{\Omega}\Phi(|h|)\,d\mu &=& \int_{\Omega}\Phi\left(\frac{1}{\kappa} |G|\right)d\mu\\ &\le& \frac{1}{\kappa}\int_{\Omega}\Phi(|G|)\,d\mu \\&\le&\frac{1}{\kappa}.\kappa a = a\,.
\end{eqnarray*}
Hence
\begin{eqnarray*}
\sum_{k=1}^N\|f\|^{(a_k)}_{\Psi, \Omega_k} &\le& \sup\left\{\left|\int_{\Omega}fG\,d\mu\right| : \int_{\Omega}\Phi(|G|)\,d\mu \le \kappa a\right\}\\
&\le&  \sup\left\{\left|\int_{\Omega}f|\kappa h|\,d\mu\right| : \int_{\Omega}\Phi(|h|)\,d\mu \le  a\right\} \\ &=&\kappa \sup\left\{\left|\int_{\Omega}f h\,d\mu\right| : \int_{\Omega}\Phi(|h|)\,d\mu \le  a\right\}\\ &=&\kappa \|f\|^{(a)}_{\Psi, \Omega}\,.
\end{eqnarray*}
\end{proof}

Let
\begin{equation}\label{OrlAverage*}
\|f\|^{\rm (av),\tau}_{\Psi,\Omega}  = \sup\left\{\left|\int_\Omega f \varphi d\mu\right| : \
\int_\Omega \Phi(|\varphi|) d\mu \le \tau\mu(\Omega)\right\},\;\;\;\tau > 0 .
\end{equation}
\begin{lemma}\label{lemma8}For any $\tau_1,\;\tau_2 > 0$
\begin{equation}\label{bsr2}
\min\left\{1, \frac{\tau_2}{\tau_1}\right\} \|f\|^{\rm (av),\tau_1}_{\Psi,\Omega} \le \|f\|^{\rm (av),\tau_2}_{\Psi,\Omega} \le \max\left\{1, \frac{\tau_2}{\tau_1}\right\} \|f\|^{\rm (av),\tau_1}_{\Psi,\Omega}.
\end{equation}
\end{lemma}
\begin{proof}
Let
$$
X_1:=\left\{\varphi\; :\;\int_{\Omega}\Phi(|\varphi|)d\mu \le \tau_1\mu(\Omega)\right\},\;\;\;X_2:= \left\{\varphi\; :\;\int_{\Omega}\Phi(|\varphi|)d\mu \le \tau_2\mu(\Omega)\right\}.
$$
Suppose that $\tau_1\le \tau_2$. Then, it is clear that $\|f\|^{\rm (av),\tau_1}_{\Psi,\Omega}\le \|f\|^{\rm (av),\tau_2}_{\Psi,\Omega}$. Now, since $\Phi$ is convex and $\Phi(0) = 0$, then
$$
\varphi\in X_2 \Rightarrow\;\;\frac{\tau_1}{\tau_2}\varphi\in X_1\,, \;\;\;(\textrm{cf}.\,\eqref{LuxNormImpl}).
$$ Hence,
$$
\|f\|^{\rm (av),\tau_2}_{\Psi,\Omega} = \underset{\varphi\in X_2}\sup\left|\int_{\Omega}f\varphi d\mu\right| \le \underset{\phi\in X_1}\sup\left|\int_{\Omega}f.\left(\frac{\tau_2}{\tau_1}\phi\right)d\mu\right| = \frac{\tau_2}{\tau_1}\|f\|^{\rm (av),\tau_1}_{\Psi,\Omega}.
$$
On the other hand, suppose that $\tau_1 \geq \tau_2$. Then
$$
\|f\|^{\rm (av),\tau_2}_{\Psi,\Omega}\le \|f\|^{\rm (av),\tau_1}_{\Psi,\Omega} \le \frac{\tau_1}{\tau_2}\|f\|^{\rm (av),\tau_2}_{\Psi,\Omega}.
$$Hence,
$$
\min\left\{1, \frac{\tau_2}{\tau_1}\right\} \|f\|^{\rm (av),\tau_1}_{\Psi,\Omega} \le \|f\|^{\rm (av),\tau_2}_{\Psi,\Omega}
$$ and
$$
\|f\|^{\rm (av),\tau_2}_{\Psi,\Omega} \le \max\left\{1, \frac{\tau_2}{\tau_1}\right\} \|f\|^{\rm (av),\tau_1}_{\Psi,\Omega}.
$$
\end{proof}
As a result of the above Lemma, we have the following Corollary.
\begin{corollary}\label{avequiv}{\rm (see \cite[Lemma 2.1]{Eugene})}
$$
\min\{1, \mu(\Omega)\}\, \|f\|_{\Psi, \Omega} \le \|f\|^{\rm (av)}_{\Psi, \Omega}
\le \max\{1, \mu(\Omega)\}\, \|f\|_{\Psi, \Omega}.
$$
\end{corollary}
\begin{lemma}\label{lemma9}
Let $\Omega_k,\; k = 1, 2,...,n$  be pairwise disjoint subsets of $\Omega\subset\mathbb{R}^n$ such that $ \Omega = \bigcup^{n}_{k = 1}\Omega_k $ and set
$$
M:= \underset{k = 1, 2, ..., n}\max\;\frac{\mu(\Omega)}{\mu(\Omega_k)}.
$$
 Then
\begin{equation}\label{bsr3}
\sum_{k = 1}^n\|f\|^{\rm(av)}_{\Psi,\Omega_k} \le \|f\|^{\rm(av)}_{\Psi,\Omega} \le M\sum_{k = 1}^n\|f\|^{\rm (av)}_{\Psi,\Omega_k}.
\end{equation}
\end{lemma}
\begin{proof}
$\sum_{k = 1}^n\|f\|^{\rm (av)}_{\Psi,\Omega_k} \le \|f\|^{\rm (av)}_{\Psi,\Omega}$ follows from Lemma \ref{lemma7}.
\begin{eqnarray*}
\|f\|^{\rm (av)}_{\Psi,\Omega} &=& \underset{\varphi}\sup\left\{\left|\int_{\Omega}f\varphi d\mu\right|\;:\;\int_{\Omega}\Phi(|\varphi|)\,d\mu\le \mu(\Omega)\right\}\\&=&\underset{\varphi}\sup\left\{\left|\int_{\Omega_k}f\varphi\sum_{k = 1}^n\chi_{\Omega_k} d\mu\right|\;:\;\int_{\Omega}\Phi(|\varphi|)\,d\mu\le \mu(\Omega)\right\}\\&\le&\sum_{k =1}^n\underset{\varphi}\sup\left\{\left|\int_{\Omega_k}f\varphi d\mu\right|\;:\;\int_{\Omega_k}\Phi(|\varphi|)\,d\mu\le M\mu(\Omega_k)\right\}\\
&\le& \sum_{k = 1}^n\|f\|^{\rm (av),M}_{\Psi,\Omega_k}\\&\le&M\sum_{k = 1}^n\|f\|^{\rm (av)}_{\Psi,\Omega_k}
\end{eqnarray*}
(by Lemma \ref{lemma8}).
\end{proof}
\begin{lemma}
{\rm
\begin{equation}\label{c5}
\|1\|^{(\textrm{av})}_{\Psi, \Omega} = \Phi^{-1}(1)\mu(\Omega)\,.
\end{equation}
}
\end{lemma}
\begin{proof}
Take any function $g$ such that
$$
\int_{\Omega}\Psi(|g|)\,d\mu \le \mu(\Omega).
$$
Then the Jensen's inequality \cite[Theorem 3.3]{Rudin} implies
\begin{eqnarray*}
\Phi\left( \frac{1}{\mu(\Omega)}\int_{\Omega} g\,d\mu\right)&\le& \frac{1}{\mu(\Omega)}\int_{\Omega}\Phi(|g|)\,d\mu \\&\le& \frac{1}{\mu(\Omega)}.\mu(\Omega) = 1.
\end{eqnarray*}This implies
$$
\int_{\Omega}g\,d\mu \le \Phi^{-1}(1)\mu(\Omega).
$$
This in turn implies
\begin{equation}\label{c3}
\|1\|^{(\textrm{av})}_{\Psi, \Omega} \le \Phi^{-1}(1)\mu(\Omega)\,.
\end{equation}
On the other hand take $g = \Phi^{-1}(1)$, then
$$
\int_{\Omega} \Phi(|g|)\,d\mu = \mu(\Omega)
$$ and
$$
\int_{\Omega} g d\mu = \Phi^{-1}(1)\mu(\Omega).
$$ This implies
\begin{equation}\label{c4}
\|1\|^{(\textrm{av})}_{\Psi, \Omega} \ge \Phi^{-1}(1)\mu(\Omega)\,.
\end{equation}
\end{proof}
We will use the following pair of pairwise complementary $N$-functions
\begin{equation}\label{thepair}
\mathcal{A}(s) = e^{|s|} - 1 - |s| , \ \ \ \mathcal{B}(s) = (1 + |s|) \ln(1 + |s|) - |s| , \ \ \ s \in \mathbb{R} .
\end{equation}
\begin{remark}
{\rm For $p > 1$, there exists a constant $C = C(p)$ such that
\begin{equation}\label{p}
\|f \|_{\mathcal{B}, \Omega} \le C(p)\|f\|_{L_P(\Omega)}\,.
\end{equation}The optimal constant $C(p)$ in \eqref{p} has the following asymptotics
$$
C(p) = \frac{1}{e(p -1)} - \frac{1}{e}\ln\frac{1}{p - 1} + O(1) \;\;\textrm{as}\;\; p\longrightarrow 1,\;\;p > 1\,
$$(see \cite[Remark 6.3]{Eugene}).
}
\end{remark}

\section{Review of basic facts in spectral theory}\label{basic facts}
In this section, we summarize basic definitions and various well known results from spectral theory that will be used in the sequel.\\\\ Let $\mathcal{H}$ be a complex Hilbert space with  an inner product $\langle \cdot, \cdot\rangle$ and the norm $\|\cdot\| = \sqrt{\langle \cdot, \cdot\rangle}$\,. Let $A: D(A)(\textrm{domain of A})\longrightarrow \mathcal{H}$ be a densely defined linear operator. The operator $\widetilde{A}$ is called an extension of $A$ (or $A$ is a restriction of $\widetilde{A}$) if $$D(A)\subset D(\widetilde{A})\;\; \textrm{and for all}\; f\in D(A),\;\;\widetilde{A}f = Af.$$ The operator $A$ is said to be symmetric if $$\langle Af, g\rangle = \langle f, Ag\rangle \;\;\forall f,g\in D(A).$$The operator $A^{*}$  called the adjoint of $A$ is defined as follows: $D(A^*)$ is the set  $g\in \mathcal{H}$ such that for some $h\in\mathcal{H}$  $$\langle Af, g\rangle = \langle f, h\rangle,\;\;\;\forall f\in D(A) .$$ If $A$ is densely defined then such $h$ is unique and we define $$h := A^*g. $$ We say that $A$ is self-adjoint if $A$ is symmetric and $D(A) = D(A^{*})$.\\\\ The \textit{spectrum} of a self-adjoint operator $A$ on $\mathcal{H}$ denoted by $\sigma(A)$ is defined as the set of all $\lambda\in\mathbb{R}$ such that $A- \lambda I$, where $I$ is the identity, is not invertible. $A - \lambda I$ is invertible  if and only if $ker(A - \lambda I) =\{0\}$ and $Ran(A - \lambda I) = \mathcal{H}$. The set $\sigma_p(A)\subseteq\sigma(A)$ of eigenvalues of $A$ is called the \textit{point spectrum} of $A$. That is,
$$
\sigma_p(A) :=\{\lambda\in\mathbb{R} :\;\exists f\in D(A),\;\; \|f\|= 1,\;\;Af = \lambda f\}.
$$If $\lambda$ is an eigenvalue of $A$, then the dimension of the kernel of $A - \lambda I$ is called the multiplicity of $\lambda$. The \textit{discrete spectrum} of $A$ denoted by $\sigma_{\textrm{disc}}$ is the set consisting of isolated eigenvalues of $A$ of finite multiplicity. The set $$\sigma_{\textrm{ess}} := \sigma(A)\setminus\sigma_{\textrm{disc}}$$ is called the \textit{essential spectrum} of $A$. It contains either accumulation points of $\sigma(A)$ or isolated eigenvalues of $A$ of infinite multiplicity.\\\\The following theorem is used to characterize the essential spectrum of self-adjoint operators {\rm (\cite[Lemma 4.1.2]{EBD})}.
\begin{theorem}{\rm [Weyl criterion]\label{weyl}}
{\rm Let $A$ be a self-adjoint operator on $\mathcal{H}$. A point $\lambda\in\mathbb{R}$ belongs to $\sigma_{\textrm{ess}}(A)$ if and only if there exists a sequence $\{f_n\}_{n\in\mathbb{N}}\subset D(A)$ such that for all $n\in\mathbb{N}$, $\|f_n\| = 1, \;\; f_n \;\textrm{converges weakly to}\; 0$ and $$\| Af_n - \lambda f_n\| \longrightarrow 0\;\; \textrm{as} \;\;n \longrightarrow\infty.$$ Moreover, $\{f_n\}$ can be chosen orthonormal.}
\end{theorem}
\begin{theorem}{\rm (see, e.g., \cite[$\S$ 10.1, Theorem 1]{Lax})}\label{weak}
{\rm Let $\{f_n\}_{n\in\mathbb{N}}$ be a bounded sequence in a Hilbert space $\mathcal{H}$. If $\langle f_n, g\rangle_{\mathcal{H}} \longrightarrow 0 \mbox{ as }n\longrightarrow\infty$ for $g$ in a dense subspace of $\mathcal{H}$, then $f_n \longrightarrow 0$ weakly in $\mathcal{H}$.}
\end{theorem}
\begin{definition}
 {\rm A self-adjoint operator $A$ on $\mathcal{H}$ is said to be lower \textit{semi-bounded} if there exist a real number $c$ such that
$$
\langle Af, f\rangle \geq c\|f\|^2,\;\;\;\forall f\in D(A).
$$}
\end{definition}In this case we simply write $A\geq c$. We have the following variational formula for the bottom of the spectrum of $A$ {\rm (see \cite[Sec. 4.5]{EBD})}:
\begin{theorem}\label{Ray-Litz}{\rm [Rayleigh-Ritz]}
{\rm Let $A$ be a lower semi-bounded self-adjoint operator on $\mathcal{H}$. Then
$$
\inf\sigma(A)= \underset{\underset{f\neq 0}{u\in D(A)}}\inf \frac{\langle Af, f\rangle}{\|f\|^2}.
$$}
\end{theorem}As a consequence of Theorem \ref{Ray-Litz}, $A\geq c$ implies $\sigma(A)\subseteq [c, \infty)$. On the other hand,
$$
\inf\sigma(A) \le \frac{\langle Af, f\rangle}{\|f\|^2}
$$for any test function $f\in D(A)\setminus\{0\}$.\\\\Theorem \ref{Ray-Litz} is a special case of the mini-max principle which is used to characterize the part of the spectrum of lower semi-bounded self-adjoint operators, which is located below the bottom of the essential spectrum {\rm (cf. \cite[Lemma 4.1.2]{EBD})}.
\begin{theorem}{\rm [Mini-Max principle]}
{\rm Let $A$ be a semi-bounded self-adjoint operator on $\mathcal{H}$ and let$$ \lambda_1 \le\lambda_2\le\lambda_3\le ...$$ be the eigenvalues of $A$ below the bottom of the essential spectrum each repeated according to their multiplicity. Then
$$
\lambda_n = \underset{f_1, f_2,...,f_{n - 1}\in D(A)}\sup\;\underset{f\in D(A)\perp\{f_1, f_2,...,f_{n - 1}\}}\inf\frac{\langle Af, f\rangle}{\|f\|^2}\,.
$$  }
\end{theorem}
Let $X$ and $Y$ be Banach spaces and $A: X \longrightarrow Y$ a linear operator.
\begin{definition}
{\rm The graph of $A$ is
$$
G(A) := \left\{( x, Ax) \;:\;x\in D(A)\right\},
$$ which is a subset of $X\times Y$.
}
\end{definition}
The product space $X\times Y$ is a Banach space with $\|(x, y)\| = \|x\|_X + \|y\|_Y$.
\begin{definition}
{\rm $A: X \longrightarrow Y$ is called a closed operator if the $G(A)$ is closed. Equivalently, $A$ is closed if for every sequence $x_n \in D(A)$ converging to $x\in X$ and $Ax_n = y_n \longrightarrow y$ in $Y$, we have $x\in D(A)$ and $y = Ax$.
}
\end{definition}
If the closure of the graph of $A$ is the graph of some operator $\overline{A}$, then $\overline{A}$ is called the closure of $A$. A densely defined linear operator which has a closed linear extension is said to be closable.
\begin{definition}
{\rm A core (or essential domain) of $A$ is a subset $C$ of $D(A)$ such that the closure of the restriction of $A$ to $C$ is $\overline{A}$.
}
\end{definition}

\begin{definition}
{\rm Let a linear subspace $D(q)\subset\mathcal{H}$ be fixed. A mapping $q: D(q)\times D(q)\longrightarrow\mathbb{C}$ which satisfies
$$
q[\alpha u + \beta v, w] = \alpha q[u, w] + \beta q[v, w]
$$ and
$$
q[u, \alpha v + \beta w]= \overline{\alpha}q[u, v] + \overline{\beta}q[u, w]
$$ for all $u, v, w\in D(q)$ is called a sesquilinear form. If $D(q)$ is dense in $\mathcal{H}$, then $q$ is densely defined.
}
\end{definition}The form $q$ is said to be symmetric if $$q[u, v] = \overline{q[v, u]}$$ and lower semi-bounded if $$q[u]\geq c\|u\|^2$$ where $q[u] := q[u, u]$.
\begin{definition}\label{closed}
{\rm Let $q$ be a symmetric sesquilinear form. We say that $u_n \longrightarrow_qu$ if $u_n\in D(q),\; u_n \longrightarrow u$ and $q[ u_n - u_m]\longrightarrow 0$ as $n, m\longrightarrow\infty$. If $u_n \longrightarrow_qu$ implies $u\in D(q)$ and $q[ u_n - u]\longrightarrow 0$, then we say that $q$ is closed.
}
\end{definition}
We say that a sesquilinear form $q_2$ is an extension of a sesquilinear form $q_1$ if it has a larger domain but coincides with $q_1$ on $D(q_1)$. A sesquilinear form $q$ is said to be closable if it has a closed extension, and the smallest closed extension is called its closure $\overline{q}$. A linear subspace $E$ of the domain of a closed form $q$ is called a core for $q$ if $q$ is the closure of its restriction to $E$.
\begin{theorem}\label{rep}{\rm (Second representation theorem,\cite[Theorem 5.37]{We})}
{ \rm Let $q$ be a densely defined closed, symmetric and lower semi-bounded form in a Hilbert space $\mathcal{H}$. Then there exists a unique self-adjoint operator $A$ on $\mathcal{H}$ such that $D(A)\subset D(q)$, $$q[u, v] = \langle Au, v\rangle$$ for every $u\in D(A)$ and $v\in D(q)$. If $u\in D(q)$ and $w\in\mathcal{H}$ such that $$q[u, v] = \langle w, v\rangle$$ holds for every $v$ in the core of $q$, then $u\in D(A)$ and $Au = w$.
}
\end{theorem}
\begin{definition}\label{closability}
{ \rm We say that a form $q'$ defined on a dense linear set $D(q')$ is closable if for any sequence $u_n\in D(q')$, $\|u_n\| \longrightarrow 0$, the property $q'[u_n - u_m] \longrightarrow 0$  as $n, m \longrightarrow \infty$ implies $q'[u_n] \longrightarrow 0$.}
\end{definition}By continuity $q'$ can be extended to a closed form.

\section{The spectrum of the Laplacian on a strip}\label{strip}
It is well known that the operator $-\Delta$ densely defined on the space $L^{2}(\mathbb{R}^{2})$
is a self-adjoint and its spectrum is equal to $[0, \infty)$,  which is absolutely continuous.  However, in the case of a strip the bottom of its essential spectrum depends on the boundary conditions. Let $a> 0$ and $S = \mathbb{R}\times I$, where $I = (0, a)$, with Neumann (N), Dirichlet (D), Dirichlet-Neumann (DN) or Robin (R) boundary conditions. Consider the following spectral problem.
\begin{equation}\label{prob}
\begin{cases}
-\Delta u = \lambda u\;\;\;\;\; & \textrm{in}\;\; S\\
B_lu = 0 & \textrm{on}\;\; \partial S,
\end{cases}
\end{equation}where $B_l$ is one of the boundary operators. The operators  on the transverse section $I$, $-\Delta^I_l$, are the usual Laplacians  on $L^2(I)$ with Dirichlet boundary conditions if $l = D$, the Neumann conditions if $l = N$, the Dirichlet at $0$ and the Neumann one at $a$ if $l = DN$ or the Robin conditions if $l = R$ (see \eqref{RP}).

$$
\textrm{Dom}(-\Delta) = \{u \in W^2_2(S)\;:\;B_l u = 0\}.
$$
Let $l \in \{D, N, DN\}$. Assume that \eqref{prob}  has a non-trivial solution of the form
$$
u(x, y) = X(x)Y(y),\;\;\;\;\;\;\;\;\; X \neq 0, Y \neq 0.
$$ Then one has
$$
-\frac{X''(x)}{X(x)} = \frac{Y''(y)}{Y(y)} + \lambda = C \;\;\;\; \textrm{in} \; S
$$for some suitable separation constant $C$. This sort of separation of variables gives rise to two independent one-dimensional spectral problems, that is the longitudinal and the transverse ones. The spectrum of the longitudinal Laplacian, ($-\Delta^{\mathbb{R}}$) on $L^2({\mathbb{R}})$ is the positive-real semi-axis, i.e.,
$$
\sigma(-\Delta^{\mathbb{R}}) = \sigma_{\textrm{ess}}(-\Delta^{\mathbb{R}}) = [0, \infty).
$$
The eigenvalues of the transverse Laplacian ($-\Delta^{I}_l$) on $L^2(I)$ are given by
\begin{equation}\label{eigval}
 \lambda^D_n := \left(\frac{\pi}{a}\right)^{2}n^2,\;\;\; \lambda^N_n := \left(\frac{\pi}{a}\right)^{2}(n - 1)^2,\;\;\; \lambda^{DN}_n := \left(\frac{\pi}{2a}\right)^{2}(2n- 1 )^2
\end{equation}where $n = 1, 2, ...$.  The corresponding normalized eigenfunctions $\{f_n\}^\infty_{n = 1}$  can be chosen as follows:
\begin{equation*}
f^l_n(y):= \sqrt{\frac{2}{a}}\sin\sqrt{\lambda^l_n}y \;\;\; \textrm{for} \;\;\; l\in\{D, DN\},
\end{equation*}
\begin{equation*}
f^N_n(y):=
\begin{cases}
\sqrt{\frac{1}{a}}\;\;\;\;\; & \textrm{if}\;\; n = 1,\\
\sqrt{\frac{2}{a}}\cos\sqrt{\lambda^N_n}y & \textrm{if}\;\; n \geq 2.
\end{cases}
\end{equation*}
Since the eigenfunctions $f^l_n$ form a complete orthonormal set in $L^2([0, a])$ by Fourier analysis, there are no other eigenvalues apart from those listed in \eqref{eigval} (see, e.g., \cite{EBD} for more details).
\begin{theorem}{\rm  \cite[Theorm 4.1.5]{EBD}}
{\rm The essential spectrum of a self-adjoint operator $H$ on a Hilbert space is empty iff there is a complete set of eigenfunctions $\{f_n\}_{n = 1}^{\infty}$ of $H$ such that the corresponding eigenvalues $\lambda_n$ converge in absolute values to $\infty$ as $n\rightarrow \infty$.}
\end{theorem}Thus, by the above theorem and \eqref{eigval}, $\sigma_{ess}(-\Delta^I_l) = \emptyset$.

\begin{theorem}\label{spect}{\rm(cf. \cite[Theorem 4.1]{KK})}
$\sigma(-\Delta^S_l) = \sigma_{\textrm{ess}}(-\Delta^S_l) = [\lambda^l_1, \infty)$.
\end{theorem}
\begin{proof}Let $\mathcal{E}^l_I[u]$ and $\mathcal{E}^l_S[u]$ denote the quadratic forms of the free Laplacian on $I$ and $S$ respectively, subject to the boundary conditions $l$.
Since $\sigma(-\Delta^I_l)$ starts by $\lambda^l_1$, then for all $u\in \textrm{Dom}(-\Delta^I_l)$, we have
$$
\mathcal{E}^l_I[u] = \int_0^a|u_y|^2dy \geq \lambda^l_1\parallel u\parallel^2_{L^2(I)}\,.
$$
Now
\begin{eqnarray*}
\mathcal{E}^l_S[u] &=& \int_S\mid u_x\mid^2dxdy + \int_S\mid u_y\mid^2dxdy \\ &\geq&\int_S\mid u_y\mid^2dxdy \\ &=&\int_{\mathbb{R}}dx\left(\int^a_0\mid u_y\mid^2dy \right)\\&\geq&  \lambda^l_1\parallel u\parallel^2_{L^2(S)}
\end{eqnarray*}This implies $\sigma(-\Delta^S_l) \subseteq [\lambda^l_1, \infty)$.\\\\

On the other hand, pick $\varphi\in C^{\infty}_0(\mathbb{R})$ with $\textrm{supp}\, \varphi = [-1, 1]$ such that\\ $\|\varphi\|_{L^2(\mathbb{R})} = 1$. Let $\varphi_n(x) := n^{-\frac{1}{2}}\varphi(\frac{x}{n})$ so that $\|\varphi_n\|_{L^2(\mathbb{R})} = 1$.\\
Take $\forall \lambda \ge \lambda_1^l$ and consider a sequence $\{u_n\}_{n =1}^{\infty}\subset \textrm{Dom}(-\Delta^S_l)$ given by
$$
u_n(x, y) := \varphi_n(x) e^{i\sqrt{\lambda - \lambda^l_1}x}f_1^l(y).
$$Then
$$
\|u_n\|_{L^2(S)} = 1.
$$
Since $-(f^l_1)''(y) = \lambda^l_1 f^l_1(y)$ and
\begin{eqnarray*}
\Delta u_n(x, y) &=& \varphi''_n(x)e^{i\sqrt{\lambda - \lambda_1^l}x}f_1^l(y) \\&+& 2i\sqrt{\lambda - \lambda_1^l}\varphi'_n(x)e^{i\sqrt{\lambda - \lambda_1^l}x}f_1^l(y)- \lambda\varphi_n(x)e^{i\sqrt{\lambda - \lambda_1^l}x}f_1^l(y)
\end{eqnarray*}
then,
\begin{eqnarray*}
-\Delta u_n(x, y) - \lambda u_n(x, y) &=& -\varphi''_n(x, y)e^{i\sqrt{\lambda - \lambda_1^l}x}f_1^l(y)\\ &-& 2i\sqrt{\lambda - \lambda_1^l}\varphi'_n(x)e^{i\sqrt{\lambda - \lambda_1^l}x}f_1^l(y)
\end{eqnarray*}
implying that
\begin{eqnarray*}
\parallel(-\Delta^S_l - \lambda)u_n\parallel &\le& \parallel\varphi''_n\parallel\|f^l_1\| + 2\sqrt{\lambda - \lambda^l_1}\parallel\varphi'_n\parallel\|f^l_1\|\\&=& \frac{1}{n^2}\parallel\varphi''\parallel + 2\sqrt{\lambda - \lambda^l_1}\frac{1}{n}\parallel\varphi'\parallel \rightarrow 0 \;\;\textrm{as}\;\; n \rightarrow\infty.
\end{eqnarray*}
Now, it remains to show that $u_n \longrightarrow 0, \, n \longrightarrow \infty$ weakly in $L^2(S)$. For any  $N\in\mathbb{N}$,  let
\begin{eqnarray*}
w_N(x, y) := \left\{\begin{array}{l}
  w(x, y), \ \;\;\mbox{ if }| w(x)| \le N,   \\ \\
   0, \ \;\;\;\;\;\;\mbox{ if } |w(x)| > N
\end{array}\right.
\end{eqnarray*}
 with $w \in W^1_2(S)$.
Then $w_N \in L^1(S)\cap L^2(S)$ and $\|w - w_N\|_{L^2(S)} \longrightarrow 0$ as $N \longrightarrow \infty$. Also
$$
|\langle u_n, w_N\rangle_{L^2(S)}| \le \|u_n\|_{L^{\infty}(S)}\|w_N\|_{L^1(S)} \le \frac{\textrm{const}}{\sqrt{n}}\|w_N\|_{L^1(S)} \longrightarrow 0 \mbox{ as }n\longrightarrow \infty.
$$Since $\|u_n\|_{L^(S)} = 1$,  $u_n$ converges weakly to $0$ in $L^2(S)$ by Theorem \ref{weak}.
Thus Theorem implies \ref{weyl}, $\lambda\in\sigma_{\textrm{ess}}(-\Delta^S_l) $. Hence\\ $$[\lambda^l_1, \infty) \subseteq \sigma_{\textrm{ess}}(-\Delta^S_l)= \sigma(-\Delta^S_l).$$
\end{proof}
Theorem \ref{spect} and its proof remain true for the case of Robin boundary conditions but the quadratic form of the Laplacian involves  boundary terms (see \eqref{n}).\\\\ Next, we discuss in detail the spectrum of the Laplacian on a straight strip subject to Robin boundary conditions. We shall see that the negative part of its spectrum is not necessarily empty as opposed to the cases of Neumann, Dirichlet and Dirichlet-Neumann boundary conditions .\\\\

Let $S_0 := [0, 1]\times[0, a]$. Consider the following eigenvalue problem
\begin{equation}\label{RP}
\begin{cases}
-\Delta u = \lambda u\;\;\;\;\; & \textrm{in}\;\; S_0\\
u_x(0, y) = u_x(1, y) = 0\\
u_y(x, 0) + \alpha u(x, 0) = 0 = u_y(x, a) + \beta u(x, a)
\end{cases}
\end{equation}for  $\alpha, \beta, \lambda\in\mathbb{R}$. Assume that a solution of \eqref{RP} has the form
$$
u(x, y) = v(x)w(y).
$$Then \eqref{RP} reduces to two one dimensional problems, namely:
\begin{equation}\label{RP1}
\begin{cases}
-v''(x) = (\lambda - \tau) v(x)\;,\;\;\;\;\; 0 < x < 1\\
v'(0) = v'(1) = 0
\end{cases}
\end{equation} and
\begin{equation}\label{RP2}
\begin{cases}
-w''(y) = \tau w(y)\;,\;\;\;\;\; 0 < y < a\\
w'(0) + \alpha w(0) = w'(a) + \beta w(a) = 0,
\end{cases}
\end{equation}where $\tau\in\mathbb{R}$ is a separation constant.\\ The solution of \eqref{RP1} is given by
\begin{equation}\label{soln1}
v(x) = \cos m\pi x\;,\;\;\;\lambda = \tau + \pi^2m^2,\;\;\;m = 0, 1, 2,...
\end{equation}
To solve \eqref{RP2}, we consider the following cases:
\begin{itemize}
\item[(i)]For $\tau = 0$ the solution of the ordinary differential equation in \eqref{RP2} is of the form $w(y) = Ay + B$ for some constants $A$ and $B$. The  first boundary condition implies that
$$
w(y) = -\alpha y + 1\;\;\;\;\;\; (\textrm{take}\; B = 1)
$$
$\tau = 0$ is in the spectrum iff the following condition holds true:
\begin{equation}\label{cond1}
(1 + \beta a)\alpha = \beta.
\end{equation}
Hence when $\tau = 0$, solutions of \eqref{RP} are given by
\begin{equation}\label{rpsoln2}
u(x, y) = \cos m\pi x (1 -\alpha y )\;,\;\;\;\lambda =  \pi^2m^2,\;\;\;m = 0, 1, 2,...
\end{equation}
\item[(ii)]$\tau > 0$ gives the following general solution
$$
w(y) = A\cos\sqrt{\tau}y + B\sin\sqrt{\tau}y.
$$The boundary conditions in \eqref{RP2} yield
\begin{equation}\label{soln2}
w(y) = \cos\sqrt{\tau}y - \frac{\alpha}{\sqrt{\tau}}\sin\sqrt{\tau}y \;\;\;\;(\textrm{take} \;A = 1)
\end{equation} and
\begin{equation}\label{soln3}
\tan\sqrt{\tau}a = \frac{(\beta - \alpha)\sqrt{\tau}}{\tau + \alpha\beta}\;.
\end{equation}
Thus, we get a sequence $\tau_n = \theta^2_n,\;\; n = 1, 2, ...$ satisfying
\begin{itemize}
\item[(a)]as $\alpha, \beta \longrightarrow 0,\, \theta_n\longrightarrow\frac{n \pi}{a}$\,,
\item[(b)]as $\alpha, \beta \longrightarrow \pm\infty,\, \theta_n\longrightarrow\frac{n\pi}{a}$\,,
\item[(c)]as $\alpha\longrightarrow 0, \beta \longrightarrow \pm\infty,\, \theta_n\longrightarrow\frac{(2n +1)\pi}{2a}$\,,
\item[(d)]as $\beta\longrightarrow 0, \alpha \longrightarrow \pm\infty, \,\theta_n\longrightarrow\frac{(2n +1)\pi}{2a}$\,.
\end{itemize}

The related eigenfunctions are
\begin{equation}\label{soln4}
w_n(y)= \left\{\cos\theta_ny - \frac{\alpha}{\theta_n}\sin\theta_ny\right\}_{n = 1, 2, ...}
\end{equation} Hence  solutions of \eqref{RP} become
\begin{eqnarray}\label{rpsoln3}
u(x, y) &=& \cos m\pi x (\cos \theta_ny - \frac{\alpha}{\theta_n}\sin\theta_ny)\nonumber\\ \lambda &=& \theta_n^2 + \pi^2m^2,\;\;\;m = 0, 1, 2,...,\, n = 1, 2,... \,.
\end{eqnarray}

 If $\alpha = \beta$, then one gets  $\tau = \left(\frac{n\pi}{a}\right)^2,\;n = 0, 1, 2, ...$. Thus \eqref{rpsoln3} becomes
\begin{eqnarray}\label{rpsoln4}
u(x, y) &=& \cos m\pi x \left(\cos\frac{n\pi}{a} y - \frac{\alpha a}{n\pi}\sin\frac{n\pi}{a}y\right)\nonumber\\ \lambda &=& \left(\frac{n\pi}{a}\right)^2 + \pi^2m^2,\;\;\;m = 0, 1, 2,..., \,n = 0, 1, 2, ... \,.
\end{eqnarray}

Also, a special case of \eqref{soln3}: $\tau = -\alpha\beta$. Then $\cos\sqrt{\tau}a = 0$, i.e. \\$\tau = \left(\frac{(2n  + 1)\pi}{2a}\right)^2, \;\; n= 0, 1, 2, ...$ So, this case occurs iff  $\alpha\beta = -\left(\frac{(2n  + 1)\pi}{2a}\right)^2$ for some $n$. Hence \eqref{rpsoln3} becomes
\begin{eqnarray}\label{rpsoln5}
u(x, y) &=& \cos m\pi x \left(\cos \frac{(2n  + 1)\pi}{2a} y - \frac{2\alpha a}{(2n + 1)\pi}\sin \frac{(2n  + 1)\pi}{2a}y\right)\nonumber\\ \lambda &=& \left(\frac{(2n  + 1)\pi}{2a}\right)^2 + \pi^2m^2,\;m = 0, 1,...,\, n = 0, 1,... \,.
\end{eqnarray}

\item[(iii)]Let $\tau_1 < \tau_2$ be the smallest eigenvalues of \eqref{RP2}.  For some values of $\alpha$ and $\beta $, $\tau_1$ or $\tau_2$ might be negative.
Suppose that $\tau = - \sigma^2 (\sigma > 0)$ is a negative eigenvalue. Then \eqref{soln2} and \eqref{soln3} respectively become
\begin{equation}\label{RP2soln}
w(y) = \cosh(\sigma y) - \frac{\alpha}{\sigma}\sinh(\sigma y),
\end{equation}
\begin{equation}\label{RP2soln*}
\tanh(\sigma a) = \frac{(\beta - \alpha)\sigma}{-\sigma^2 + \alpha\beta}\,.
\end{equation}
To investigate when this happens, note that \eqref{cond1} divides the $(\alpha,\beta)$-plane into three connected components and the number of negative eigenvalues in each of them is the same since eigenvalues are continuous with respect to $\alpha$ and $\beta$. Consider the line $\beta = -\alpha$, it transects all the three regions (see Figure \ref{hyp}).
\begin{figure}
\centering
\includegraphics[height = 9cm]{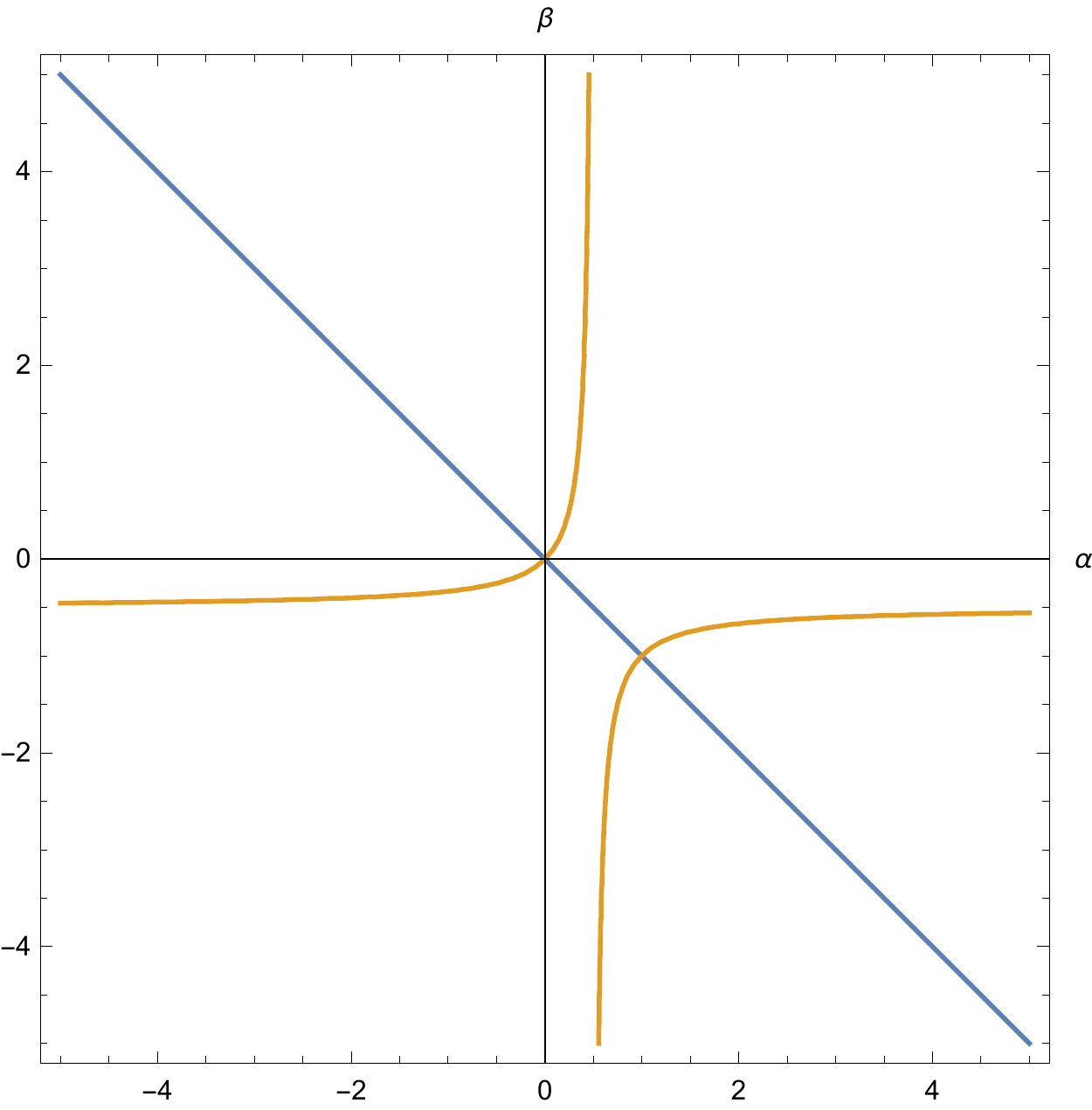}
\caption{$(1 + \beta a)\alpha = \beta,\;\beta = -\alpha$}
\label{hyp}
\end{figure}
By applying the shift $y\longmapsto y- b,\; b = \frac{a}{2}$, \eqref{RP2} becomes
\begin{equation}\label{RP3}
\begin{cases}
-w''(y) = \tau w(y)\;\;\;\;\; -b < y < b\\
w'(-b) + \alpha w(-b) = w'(b) - \alpha w(b) = 0.
\end{cases}
\end{equation}
So, if $w(y)$ is an eigenfunction, the function $Q(y) = w(-y)$ is also an eigenfunction with the same eigenvalue. Thus, we can consider separately the eigenfunctions that are even functions and those that are odd functions, described respectively by
\begin{equation}\label{RP4}
\begin{cases}
-w''(y) = \tau w(y)\;,\;\;\;\;\; 0 < y < b\\
w'(0) = w'(b) - \alpha w(b) = 0
\end{cases}
\end{equation}
and
\begin{equation}\label{RP5}
\begin{cases}
-w''(y) = \tau w(y)\;,\;\;\;\;\; 0 < y < b\\
w(0)  = w'(b) - \alpha w(b) = 0.
\end{cases}
\end{equation}
Considering $w(y) = \cosh(\sigma y)$ and $w(y) = \sinh(\sigma y)$, the boundary conditions in \eqref{RP4} and \eqref{RP5} respectively yield
\begin{equation}\label{soln5}
\alpha  = \sigma\tanh(\sigma b)
\end{equation}
and
\begin{equation}\label{soln5*}
\alpha  = \sigma \coth(\sigma b)\,.
\end{equation}
For $\sigma  >0$, both functions $\sigma\tanh(\sigma b)$ and $\sigma \coth(\sigma b)$ are monotone increasing with minima equal to 0 and $\frac{1}{b}$ respectively at $\sigma = 0$. Thus \eqref{soln5} has one solution if and only if $\alpha > 0$, and \eqref{soln5*} has one solution if and only if $\alpha b > 1$. Hence if $\alpha a > 0$ there is one negative eigenvalue with even eigenfunction and if $\alpha a > 2$, another negative eigenvalue comes from odd eigenfunction.\\\\ In general, we have the following situations. If
\begin{itemize}
\item[A)]$\alpha + \alpha\beta a - \beta < 0 $ and $\alpha < \frac{1}{a} \left( \beta > -\frac{1}{a}\right)$, then $\tau_1 > 0$,
\item[B)]$\alpha + \alpha\beta a - \beta = 0 $ and $\alpha < \frac{1}{a} \left(\beta > -\frac{1}{a}\right)$, then $\tau_1 = 0$ and $\tau_2 >0$,
\item[C)]$\alpha + \alpha\beta a - \beta > 0 $, then $\tau_1 < 0$ and $\tau_2 > 0$,
\item[D)]$\alpha + \alpha\beta a - \beta = 0 $ and $\alpha > \frac{1}{a} \left( \beta < -\frac{1}{a}\right)$, then $\tau_1 < 0$ and $\tau_2 = 0$,
\item[E)]$\alpha + \alpha\beta a - \beta < 0 $ and $\alpha > \frac{1}{a}\left( \beta < -\frac{1}{a}\right)$, then $\tau_1 < 0$ and $\tau_2 < 0$.
\end{itemize}
Cases $(C)- (E)$ produce the following solutions of \eqref{RP}
\begin{eqnarray}\label{rpsoln5}
u(x, y) &=& \cos m\pi x \left(\cosh\sigma_n y - \frac{\alpha }{\sigma_n}\sinh\sigma_ny\right)\nonumber\\ \lambda &=& -\sigma^2_n + \pi^2m^2,\;\;\;m = 0, 1, 2,...,
\end{eqnarray}
where $ n = 1$ in cases $(C)$ and $(D)$ while $n = 1, 2$ in case $(E)$.
\end{itemize}
\begin{remark}
{\rm When $\alpha = 0$, we have the Neumann condition at $0$ and the Robin conditions at $a$. By \eqref{cond1}, $\tau = 0$ if and only if $\beta = 0$. If $\beta > 0$, then $\tau_1, \tau_2 > 0$ by condition (A). If $\beta < 0$, then $\tau_1 < 0$ and $\tau_2 > 0$ by condition (C).\\\\
 When $\beta = 0$, we have the Robin condition at $0$ and the Neumann conditions at $a$. By \eqref{cond1}, $\tau = 0$ if and only if $\alpha = 0$. If $\alpha > 0$, then $\tau_1 < 0$ and  $\tau_2 > 0$ by condition (C). If $\alpha < 0$, then $\tau_1 > 0$ and $\tau_2 > 0$ by condition (A).\\\\Let $\alpha\longrightarrow\pm\infty$. Then we have the Dirichlet conditions at $0$ and the Robin conditions at $a$. Again \eqref{cond1} implies that $\tau = 0$ if and only if $1 + \beta a = 0$. \eqref{RP2soln} and \eqref{RP2soln*} respectively become
 \begin{equation}\label{dsoln1}
 w(y) = \sinh(\sigma y)
 \end{equation}
 and
 \begin{equation}\label{dsoln2}
 \tanh(\sigma a) = -\frac{\sigma}{\beta}\,.
 \end{equation} \eqref{dsoln2} implies
 \begin{equation}\label{dsoln3}
 \beta = - \sigma \coth(\sigma a)\,,
 \end{equation} (cf. \eqref{soln5*}).
 Hence $\tau_1 > 0$ if $1 + \beta a > 0$ and $\tau_1 < 0$, $\tau_2 > 0$ if $1 + \beta a < 0$. \\\\ Let $\beta\longrightarrow\pm\infty$. Then we have Robin conditions at $0$ and Dirichlet conditions at $a$. By \eqref{cond1}, $\tau = 0$ if and only if $\alpha a - 1 = 0$.  \eqref{RP2soln*} becomes
 \begin{equation}\label{rdsoln}
 \alpha = \sigma\coth(\sigma a)\,,
 \end{equation}(cf. \eqref{soln5*}). Thus $\tau_1 < 0$, $\tau_2 > 0$ if $ \alpha a -1  > 0$ and $\tau_1 > 0$ if $\alpha a -1  < 0$.
}
\end{remark}

\section{Index of quadratic forms}\label{index}
Let $\mathcal{H}$ be a Hilbert space and let $\mathbf{q}$ be a Hermitian form with a domain
$\mbox{Dom}\, (\mathbf{q}) \subseteq \mathcal{H}$. Set
\begin{equation}\label{hermitian}
N_- (\mathbf{q}) := \sup\left\{\dim \mathcal{L}\, | \  \mathbf{q}[u] < 0, \,
\forall u \in \mathcal{L}\setminus\{0\}\right\} ,
\end{equation}
where $\mathcal{L}$ denotes a linear subspace of $\mbox{Dom}\, (\mathbf{q})$.  The number $N_- (\mathbf{q})$ is called the Morse index of $\mathbf{q}$ in $\mbox{Dom}\,(\mathbf{q})$. If $\mathbf{q}$
is the quadratic form of a self-adjoint operator $A$ with no essential spectrum in $(-\infty, 0)$, then
by the variational principle,
$N_- (\mathbf{q})$ is the number of negative eigenvalues of $A$ repeated according to their
multiplicity (see, e.g., \cite[S1.3]{BerShu} or \cite[Theorem 10.2.3]{BirSol}).\\\\
Let $\Omega\subset\mathbb{R}^2$ be an arbitrary open set, $\mu$ a positive $\sigma$-finite Radon measure on $\mathbb{R}^2$. Further, let $V$ be a non-negative $\mu$-measurable real valued function such that $V\in L^1_{\textrm{loc}}(\overline{\Omega}, \mu)$. Define the following quadratic form
$$
\mathcal{E}_{V\mu,\Omega}[w] := \int_{\Omega}|\nabla w|^2\,dx - \int_{\overline{\Omega}}V|w|^2\,d\mu(x),
$$  with domain
$$
\mathcal{H} := W^1_2(\Omega)\cap L^2(\overline{\Omega}, Vd\mu).
$$Note that $\mu$ does not have to be the two dimensional Lebesgue measure, and it may well happen that $\mu(\partial\Omega) > 0$. In case $\mu$ is absolutely continuous with respect to the Lebesgue measure, we shall simply write $\mathcal{E}_{V, \Omega}$ for the above form.
\\\\
For any $V \ge 0$ and any $\sigma$-finite Radon measure $\mu$ on $\mathbb{R}^2$, $N_-(\mathcal{E}_{V\mu,\Omega})\geq 1$. Indeed, if $V\in L^1(\overline{\Omega}, \mu)$, then $1\in \textrm{Dom}(\mathcal{E}_{V\mu,\Omega})$ and $\mathcal{E}_{V\mu,\Omega}[1] < 0$, which implies $N_-(\mathcal{E}_{V\mu,\Omega})\geq 1$. If $V \notin L^1(\overline{\Omega}, \mu)$, then consider $w_k(x) = \frac{1}{k}(k - |x|)_+ $ for any positive integer $k$. Then this function belongs to $\mathcal{H}$ since it has a compact support, $0 \le w_k\le 1$, and $\int_{\Omega}|\nabla w_k|^2\,dx \le \pi$. Since $w_k \rightarrow 1$ pointwise as $k \rightarrow\infty$, then by Fatou Lemma \cite[Lemma 1.28]{Rudin}, $$\int_{\overline{\Omega}}Vw^2_k\, d\mu(x) \rightarrow \int_{\overline{\Omega}}V\, d\mu(x) = \infty.$$ Hence, for $k$ large enough, $\mathcal{E}_{V\mu,\Omega}[w_k]< 0$ and we have that $N_-(\mathcal{E}_{V\mu,\Omega})\geq 1$.

Under certain conditions, $\mathcal{E}_{V\mu, \Omega}$ generates a self-adjoint operator
$$
H_{V\mu} := -\Delta - V\mu, \;\;\; V\ge 0
$$ as follows. Suppose that
\begin{equation}\label{sadj}
\int_{\overline{\Omega}}V(x)|w(x)|^2d\mu(x) \le a \int_{\Omega}|\nabla w(x)|^2dx + b \int_{\Omega}|w(x)|^2dx
\end{equation}for all $w\in W^1_2(\Omega)\cap C(\overline{\Omega})$ and some positive constants $ a < 1$ and $b$ (see, e.g., \cite{Aize}, \cite{Her} and \cite{Herc} for examples of measures which satisfy \eqref{sadj}). The Dirichlet integral
$$
\int_{\Omega}|\nabla w(x)|^2\,dx
$$ with the domain $W^1_2(\Omega)$ is a non-negative and closed form on $L^2(\Omega)$. Thus it follows from $\eqref{sadj}$ and the Kato-Lax-Lions-Milgram-Nelson theorem \cite[Theorem X.17]{Reed} that the quadratic form $\mathcal{E}_{V\mu,\Omega}$ is lower semi-bounded and closable on $L^2(\Omega)$. Thus the operator $H_{V\mu}$ uniquely associated with $\mathcal{E}_{V\mu,\Omega}$ in the sense of the second representation theorem (see Theorem \ref{rep}) is self-adjoint.

\begin{definition}
{\rm Let $\Omega\subset\mathbb{R}^n$ be an open set. We say that a (finite or infinite) sequence $\{\Omega_k\}$ of non-empty open subsets $\Omega_k\subset\Omega$ is a partition of $\Omega$ if $\Omega_k \cap \Omega_l = \emptyset, \;k \not= l$ and $\underset{k}\cup \overline{\Omega}_k = \overline{\Omega}$.}
\end{definition}
The following result can be found, e.g., in \cite[Ch.6, $\S$2.1, Theorem 4]{Hilb} in the case $\mu$ is absolutely continuous with respect to the Lebesgue measure.
\begin{lemma}\label{part*}
{\rm If $\{\Omega_k\}$ is a partition of $\Omega$, then
\begin{equation}\label{parti}
N_-(\mathcal{E}_{V\mu, \Omega})\le \underset{k}\sum N_-(\mathcal{E}_{V\mu,\Omega_k}), \,\,\,\,\,\forall V \ge 0.
\end{equation}
}
\end{lemma}
\begin{proof}
 Let
$$
\Sigma := \oplus\{\textrm{Dom}(\mathcal{E}_{V\mu, \Omega_k}),\,k = 1, 2, ...\}.
$$
Here $\oplus$ denotes a direct sum. We consider $\sum_k\mathcal{E}_{V\mu, \Omega_k}$ as a form defined on $\Sigma$.
Let $\mathcal{J} : \textrm{Dom}(\mathcal{E}_{V\mu, \Omega}) \longrightarrow \Sigma$ be the embedding defined by
$$
w \longmapsto (w|_{\Omega_1}, w|_{\Omega_2}, ...).
$$  Let $\Gamma := \mathcal{J}(\textrm{Dom}(\mathcal{E}_{V\mu, \Omega}))$. Then $\forall w\in \textrm{Dom}(\mathcal{E}_{V\mu, \Omega})$, we have
\begin{eqnarray*}
\mathcal{E}_{V\mu, \Omega}[w]&=& \int_{\Omega}|\nabla w(x)|^2dx - \int_{\overline{\Omega}}V(x)|w(x)|^2d\mu(x)\\ &\geq& \sum_k\left(\int_{\Omega_k}|\nabla w(x)|^2dx - \int_{\overline{\Omega}_k}V(x)|w(x)|^2d\mu(x)\right) \\ &=&\sum_k \mathcal{E}_{V\mu, \Omega_k}[w|_{\Omega_k}] = \left(\sum_k \mathcal{E}_{V\mu, \Omega_k}\right)[\mathcal{J}w].
\end{eqnarray*}
Hence
$$
N_-(\mathcal{E}_{V\mu, \Omega})\le N_-\left(\left(\sum_k\mathcal{E}_{V\mu, \Omega_k}\right)\big|_{\Gamma}\right) \le  N_-\left(\sum_k\mathcal{E}_{V\mu, \Omega_k}\right) = \underset{k}\sum N_-(\mathcal{E}_{V\mu,\Omega_k}).
$$
\end{proof}
\section{The variational approach to the problem}\label{approach}
  Let $\Omega \subseteq \mathbb{R}^2$ be an unbounded domain. Split $\Omega$ into bounded domains $\Omega_n$ such that $\overline{\Omega} = \underset{n\in\mathbb{Z}}\cup \overline{\Omega_n}$. Then by \eqref{parti}
\begin{equation}\label{variat1}
N_-\left(\mathcal{E}_{V\mu,\Omega}\right) \le \sum_{n\in\mathbb{Z}}N_-\left(\mathcal{E}_{V\mu, \Omega_n}\right).
\end{equation}
If we have  Neumann boundary conditions for $\Omega_n$ and $w$ is a constant function, say $w = 1$, $\mathcal{E}_{V\mu,\Omega_n}[1] < 0$ implying that $N_-\left(\mathcal{E}_{V\mu,\Omega_n}\right) \ge 1$ and thus the right-hand side of \eqref{variat1} diverges. Therefore, we need to get rid of constant functions. We do so by working on a space of functions whose mean value over  $\Omega_k$ is equal to zero, i.e., functions that are orthogonal to $1$ over $\Omega_k$. This makes the use of the Poincar\'e inequality (see Appendix \ref{poincare}) possible. For other boundary conditions we use different orthogonality conditions to get a variant of the Poincar\'e inequality (see \S\ref{Robin}).\\\\
Now, let $\Omega = \mathbb{R}^2$. Then the problem is split into two problems as follows: Let $(r, \theta)$ denote the polar coordinates in $\mathbb{R}^2$, $r\in\mathbb{R}_+,\;\theta\in[-\pi, \pi]$ and
\begin{equation}\label{radial}
w_{\mathcal{R}}(r) := \frac{1}{2\pi}\int_{-\pi}^{\pi}w(r, \theta)d\theta,\;\;\;w_{\mathcal{N}}(r, \theta) := w(r,\theta) - w_{\mathcal{R}}(r),\;\; w\in C(\mathbb{R}^2\setminus\{0\}).
\end{equation}
Then
\begin{equation}\label{fN0}
\int_{-\pi}^\pi w_{\mathcal{N}}(r, \theta)\, d\theta = 0 , \ \ \ \forall r > 0 ,
\end{equation}
and it is easy to see that
$$
\int_{\mathbb{R}^2} w_{\mathcal{R}} v_{\mathcal{N}}\, dy = 0 , \ \ \ \forall
w, v \in C\left(\mathbb{R}^2\setminus\{0\}\right)\cap L^2\left(\mathbb{R}^2\right) .
$$
Hence
$w \mapsto Pw := w_{\mathcal{R}}$
extends to an orthogonal projection $P : L^2\left(\mathbb{R}^2\right) \to
L^2\left(\mathbb{R}^2\right)$.

Using the representation of the gradient in polar coordinates one gets
\begin{eqnarray*}
&& \int_{\mathbb{R}^2} \nabla w_{\mathcal{R}} \nabla v_{\mathcal{N}}\, dy =
\int_{\mathbb{R}^2} \left(\frac{\partial w_{\mathcal{R}}}{\partial r}
\frac{\partial v_{\mathcal{N}}}{\partial r} + \frac1{r^2}
\frac{\partial w_{\mathcal{R}}}{\partial \theta}
\frac{\partial v_{\mathcal{N}}}{\partial \theta}\right)\, dy \\
&& = \int_{\mathbb{R}^2} \frac{\partial w_{\mathcal{R}}}{\partial r}
\frac{\partial v_{\mathcal{N}}}{\partial r}\, dy =
\int_{\mathbb{R}^2} \left(\frac{\partial w}{\partial r}\right)_{\mathcal{R}}
\left(\frac{\partial v}{\partial r}\right)_{\mathcal{N}}\, dy = 0 , \ \ \
\forall w, v \in C^\infty_0\left(\mathbb{R}^2\right) .
\end{eqnarray*}
Hence $P : W^1_2\left(\mathbb{R}^2\right) \to W^1_2\left(\mathbb{R}^2\right)$ is also
an orthogonal projection.

Since
\begin{eqnarray*}
&& \int_{\mathbb{R}^2} |\nabla w|^2\, dx = \int_{\mathbb{R}^2} |\nabla w_{\mathcal{R}}|^2\, dx +
\int_{\mathbb{R}^2} |\nabla w_{\mathcal{N}}|^2\, dx , \\
&& \int_{\mathbb{R}^2} V |w|^2\, d\mu(x) \le 2 \int_{\mathbb{R}^2} V|w_{\mathcal{R}}|^2\, d\mu(x) +
2\int_{\mathbb{R}^2} V |w_{\mathcal{N}}|^2\, d\mu(x) ,
\end{eqnarray*} we have
\begin{equation}\label{meaeqn1}
N_-(\mathcal{E}_{V\mu, \mathbb{R}^2}) \le N_-(\mathcal{E}_{\mathcal{R},2V\mu}) + N_-(\mathcal{E}_{\mathcal{N},2V\mu})
\end{equation} where $\mathcal{E}_{\mathcal{R},2V\mu}$ and $\mathcal{E}_{\mathcal{N}, 2V\mu}$ are the restrictions of the form $\mathcal{E}_{2V\mu, \mathbb{R}^2}$ to  $PW^1_2(\mathbb{R}^2)$ and $(I -P)W^1_2(\mathbb{R}^2)$ respectively. Therefore to estimate $N_-\left(\mathcal{E}_{V\mu,\mathbb{R}^2}\right)$, it is sufficient to estimate $N_-\left(\mathcal{E}_{\mathcal{R},2V\mu}\right)$ and $N_-\left(\mathcal{E}_{\mathcal{N}, 2V\mu}\right)$. The estimates for $N_-\left(\mathcal{E}_{\mathcal{N}, 2V\mu}\right)$ are different in nature from those for $N_-\left(\mathcal{E}_{\mathcal{R}, 2V\mu}\right)$ and require different techniques.\\\\
On the space $PW^1_2(\mathbb{R}^2)$, a simple exponential change of variables reduces the problem to a well studied one-dimensional Schr\"odinger operator which provides an estimate for $N_-\left(\mathcal{E}_{\mathcal{R},2V\mu}\right)$ in terms of weighted $L^1$ norms of $V$ that is optimal (see, e.g., \eqref{meaeqn3}, \eqref{meaeqn5} and Theorem \ref{measthm3}). \\On the space $(I - P)W^1_2(\mathbb{R}^2)$, one gets an estimate for $N_-\left(\mathcal{E}_{\mathcal{N},2V\mu}\right)$ in terms of $L^p \,,p>1$ or Orlicz norms of $V$ (see \eqref{meaeqn6*}) instead of $L^1$ norm since  $W^1_2(\mathbb{R}^2)$ is not embedded in $ L^{\infty}(\mathbb{R}^2)$.  Let
\begin{eqnarray*}
\mathcal{E}_{\mathcal{N},2V\mu}[w] &:=& \int_{\mathbb{R}^2}|\nabla w(x)|^2\,dx - 2\int_{\mathbb{R}^2}V(x)|w(x)|^2\,d\mu(x),\\
\textrm{Dom}\;\left(\mathcal{E}_{\mathcal{N},2V\mu}\right) &=& \left\{w\in (I -P)W^1_2(\mathbb{R}^2)\cap L^2\left(\mathbb{R}^2, Vd\mu)\right)\right\}.
\end{eqnarray*} Split $\mathbb{R}^2$ into the following annuli
$$
\Omega_n := \left\{x\in\mathbb{R}^2\;:\; 2^{n -1} < |x| < 2^n\right\}, \;\;n\in\mathbb{Z}.
$$
Then
$$
\int_{\Omega_n} w(x)\,dx = 0, \;\;\;\forall w\in (I - P)W^1_2(\Omega_n)\;\;\;\;(\textrm{cf}.\, \eqref{fN0}).
$$
 The variational principle (see \eqref{parti}) implies
\begin{equation}\label{variat}
N_-\left(\mathcal{E}_{\mathcal{N},2V\mu}\right) \le \sum_{n\in\mathbb{Z}}N_-\left(\mathcal{E}_{\mathcal{N},2V\mu,\Omega_n}\right),
\end{equation}where
\begin{eqnarray*}
\mathcal{E}_{\mathcal{N}, 2V\mu,\Omega_n}[w] &:=& \int_{\Omega_n}|\nabla w(x)|^2\,dx - 2\int_{\overline{\Omega}_n}V(x)|w(x)|^2\,d\mu(x),\\
\textrm{Dom}\;\left(\mathcal{E}_{\mathcal{N},2V\mu, \Omega_n}\right) &=& \left\{w\in (I - P)W^1_2(\Omega_n)\cap L^2\left(\overline{\Omega_n}, Vd\mu\right)\right\}.
\end{eqnarray*}
Depending on the structure of $\mu$, any estimate for $N_-\left(\mathcal{E}_{\mathcal{N}, 2V\mu,\Omega_0}\right)$ that has the right scaling leads to an estimate for $N_-\left(\mathcal{E}_{\mathcal{N}, 2V\mu,\Omega_n}\right)$  for all $n$. So it is sufficient to find an estimate for $N_-\left(\mathcal{E}_{\mathcal{N}, 2V\mu,\Omega_0}\right)$. \\\\

Let  $\Omega = S := \mathbb{R}\times (0, a),\,a > 0$. Then depending on the boundary conditions at $0$ and $a$, the form $\mathcal{E}_{V\mu, S}$ will involve  boundary terms (see \eqref{S}). Split $S$ into the following rectangles $$S_n := (n , n+1)\times(0, a),\; n\in\mathbb{Z}.$$ Then it is easy to see that the variational principle \eqref{parti} remains true, i.e.,
\begin{equation}\label{varstrip}
N_-\left(\mathcal{E}_{V\mu, S}\right) \le \sum_{n\in\mathbb{Z}}N_-\left(\mathcal{E}_{V\mu, S_n}\right).
\end{equation}
Here, $\mathcal{E}_{V\mu, S_n}$ is the form that is obtained from $\mathcal{E}_{V\mu, S}$ by restricting all the integrals to the intersections of the corresponding sets with $\overline{S_n}$ (see \eqref{S_n}).
Decomposing $W^1_2(S)$ depends on the boundary conditions at $0$ and $a$.  Let $$\mathcal{H}_1 := \left\{u\in W^1_2(S) : u(x) = w(x_1)u_1(x_2)\right\}, $$ where $u_1$ is the first eigenfunction of $-\frac{d^2}{dx^2_2}$ on $[0, a]$. Then one can define an orthogonal projection $P : W^1_2(S) \longrightarrow \mathcal{H}_1$ (see \eqref{projection}).
Again, an estimate coming from $\mathcal{H}_1$  will contribute the weighted $L^1$ norms of $V$ to the estimate for $N_-\left(\mathcal{E}_{V\mu, S}\right)$ (see \eqref{radest1}). Let $\mathcal{H}_2 := (I -P)W^1_2(S)$. Then $\mathcal{H}_2$ consists of functions in $W^1_2(S)$ that are orthogonal to $u_1$ in the inner product of $L^2([0, a])$ (see $\S$\ref{Robin}). Note that if $S$ is subject to Neumann boundary conditions,  then the  mean value over $S_n $ of the functions in $\mathcal{H}_2$ is equal to zero since $u_1 = 1$ (see $\S$\ref{Neumann}).

\chapter{Review of known results and auxilliary results}\markboth{Chapter \ref{review}.
Introduction}{}\label{Introduction}

\section{Review of known results}\label{Literature}
In this section, we present some known results  starting with the work of M. Solomyak (1994) that has had considerable influence in this line of research.\\\\ Let
\begin{eqnarray*}
\Omega_0 &=& \{x\;:\;|x| \le 1\},\;\; \Omega_n = \{x\;: \; 2^{n-1} < |x| < 2^n \}, \;\;\;n\in\mathbb{N};\\
U_0 &=& \{x\;:\;|x| \le e\},\;\; U_n = \{x\;: \; e^{2^{n-1}} < |x| < e^{2^n} \}, \;\;\;n\in\mathbb{N}.
\end{eqnarray*}Note that the radius  of the unit disk $\Omega_0$ and the radii of the annuli $\Omega_n$ form a geometric series. Similarly the logarithms of the inner radii of the annuli $U_n$ form a geometric series. For a given potential $V$, let
\begin{eqnarray}
\zeta_n (V) & =& \int_{U_n}V(x)|\ln|x||\,dx, \\
\eta_n(V) &=& \|V\|^{(\textrm{av})}_{\mathcal{B}, \Omega_n}.
\end{eqnarray}
Recall that a sequence $\{a_n\}$ belongs to the ``weak $l_1$-space'' (Lorentz space) $l_{1,w}$ if the following quasinorm
\begin{equation}\label{quasi}
\|\{a_n\}\|_{1,w} = \underset{s > 0}\sup\left(s \;\textrm{card}\{n\;:\;|a_n| > s\}\right)
\end{equation} is finite. It is a quasinorm in the sense that it satisfies the weak version of the triangle inequality:
$$
\|\{a_n\} + \{b_n\}\|_{1,w} \le 2\left(\|\{a_n\}\|_{1,w} + \|\{b_n\}\|_{1,w}\right).
$$
Indeed, if $|a_n| \le \frac{s}{2}$ and $|b_n| \le \frac{s}{2}$, then $|a_n + b_n| \le s$. This implies
$$
\left\{n : \,|a_n + b_n| > s\right\} \subseteq \left\{n : \, |a_n| > \frac{s}{2}\right\} \cup \left\{n : \, |b_n| > \frac{s}{2}\right\}.
$$
So
$$
\textrm{card}\left\{n : \,|a_n + b_n| > s\right\} \le \textrm{card}\left\{n : \, |a_n| > \frac{s}{2}\right\} +\, \textrm{card} \left\{n : \, |b_n| > \frac{s}{2}\right\}.
$$ Hence
$$
\|\{a_n\} + \{b_n\}\|_{1,w} \le 2\left(\|\{a_n\}\|_{1,w} + \|\{b_n\}\|_{1,w}\right).
$$

The quasinorm \eqref{quasi} induces a topology on $l_{1,w}$ in which this space is non-separable. The closure of the set of elements $a_n$ with only finite number of non-zero terms is a separable subspace in $l_{1,w}$.
It is well known that $l_1 \subset l_{1,w}$ and
$$
\|\{a_n\}\|_{1,w} \le \|\{a_n\}\|_1
$$ (see, e.g., \cite{BirSol} for more details).
\begin{theorem}\label{MS}{\rm \cite[Theorem 3]{Sol}}\\
Let $V\in L_{\mathcal{B}, \textrm{loc}}(\mathbb{R}^2)$. If $\zeta_n(V)\in l_{1,w}$ and $\eta_n(V)\in l_{1}$, then the following semi-classical estimate holds
\begin{equation}\label{ms}
N_-\left(\mathcal{E}_{V,\mathbb{R}^2}\right) \le 1 + C\left(\|\{\zeta_n(V)\}\|_{1,w} + \|\{\eta_n(V)\}\|_1\right).
\end{equation}where $C > 0$ is a constant.
\end{theorem} Here $L_{\mathcal{B}, \textrm{loc}}(\mathbb{R}^2)$ is the space of functions in $L_{\mathcal{B}} (\mathbb{R}^2)$ locally integrable on $\mathbb{R}^2$. The conditions of the above Theorem are only sufficient for the semi-classical behaviour of $N_-\left(\mathcal{E}_{\alpha V,\mathbb{R}^2}\right)$ unlike \eqref{CLR} on $\mathbb{R}^n$ with $n > 2$. The presence of the norms $\|\{\zeta_n(V)\}\|_{1,w}$ and $\|\{\eta_n(V)\}\|_1$ which are different in nature on  the right hand side of \eqref{ms} complicates its optimality in terms of the function spaces of $V$.\\\\
For $\sigma > 1$ and $V \geq 0$, let
\begin{eqnarray*}
\eta_0(V, \sigma) &:=& \left(\int_{\Omega_0}V(x)^{\sigma}\right)^{\frac{1}{\sigma}},\\
\eta_n(V, \sigma) &:=& \left(\int_{\Omega_n}|x|^{2(\sigma - 1)}V(x)^{\sigma}\right)^{\frac{1}{\sigma}}, \;\;\;n\in\mathbb{Z}.
\end{eqnarray*}The following result is due to M.Sh.Birman and A. Laptev \cite{BirLap}.
\begin{theorem}\label{BLAP}
{\rm Suppose $\{\eta_n(|V|, \sigma)\} \in l_1$. Then for $\alpha > 0$
\begin{equation}\label{blap}
N_-\left(\mathcal{E}_{\alpha V,\mathbb{R}^2}\right) \sim N_-\left(\mathcal{E}^{(1)}_{\alpha V,\mathbb{R}^2}\right) + N_-\left(\mathcal{E}^{(2)}_{\alpha V,\mathbb{R}^2}\right),
\end{equation}
}
\end{theorem}
 where  $N_-\left(\mathcal{E}^{(1)}_{\alpha V,\mathbb{R}^2}\right)$ and $N_-\left(\mathcal{E}^{(2)}_{\alpha V,\mathbb{R}^2}\right)$ are contributions to the asymptotics of $N_-\left(\mathcal{E}_{\alpha V,\mathbb{R}^2}\right)$ coming from the radial and non-radial components of $V$ respectively. Potentials were constructed such that \eqref{order} holds but \eqref{Weyl1} fails.\\\\
Consider the following operator on $L^2(\mathbb{R}^2)$
$$
H_{bV} = -\Delta + b|x|^{-2} - V, \;\;\;\;x\in\mathbb{R}^2, \;\;b\in\mathbb{R}.
$$ Let
$$
\mathcal{H} = \left\{u\;\;:\;\;\int_{\mathbb{R}^2}\left(|\nabla u(x)|^2 + |u(x)|^2|x|^{-2}\right)\,dx\; <\; \infty\right\}.
$$
The following result is due to A. Laptev \cite{AL}.
\begin{theorem}\label{AL}
{ \rm Let $b > 0$ and $V(x) = V(|x|) \geq 0,\;\; V\in L^1_{\textrm{loc}}(\mathbb{R}^2)$. Then
\begin{equation}\label{al}
N_-\left(\mathcal{E}_{b V,\mathbb{R}^2}\right) \le C(b)\int_{\mathbb{R}^2}V(x)\,dx\,,
\end{equation} where
\begin{eqnarray*}
\mathcal{E}_{b V,\mathbb{R}^2}[u] &=& \int_{\mathbb{R}^2}\left(|\nabla u(x)|^2 + b|u(x)|^2|x|^{-2}\right)\,dx - \int_{\mathbb{R}^2}V(x)|u(x)|^2\,dx,\\
\textrm{Dom}\;\left(\mathcal{E}_{b V,\mathbb{R}^2}\right) &=& \mathcal{H}\cap L^2\left(\mathbb{R}^2, V(x)dx\right).
\end{eqnarray*}
}
\end{theorem}
Note that \eqref{al} is a direct analogue of \eqref{CLR}. \\\\
Using an approach similar to that of A. Laptev, K. Chadin, N. Khuri, A. Martin, T-T. Wu  \cite{Chad} in 2003 proved that if $V(x) = V(|x|)$, then
\begin{equation}\label{chad}
N_-\left(\mathcal{E}_{V,\mathbb{R}^2}\right) \le 1 + C\int_{\mathbb{R}^2}V(x)\left(1 + |\ln|x||\right)\,dx.
\end{equation}
Let
\begin{eqnarray}\label{ringsR}
&& U_n :=  \{x \in \mathbb{R}^2 : \ e^{2^{n - 1}} < |x| < e^{2^n}\} , \ n > 0 ,  \nonumber \\
&& U_0 := \{x \in \mathbb{R}^2 : \ e^{-1} < |x| < e\}, \\
&& U_n :=  \{x \in \mathbb{R}^2 : \ e^{-2^{|n|}} < |x| < e^{-2^{|n| - 1}}\} , \ n < 0 , \nonumber\\
&& \Omega_n := \{x \in \mathbb{R}^2 : \ e^n < |x| < e^{n + 1}\}, \ n \in \mathbb{Z}, \nonumber
\end{eqnarray}
and
\begin{eqnarray*}
\zeta_n := \{\zeta_n (V)\}& =& \int_{U_n}V(x)|\ln|x||\,dx,\;\;\;n\neq 0,\;\;\;\zeta_0(V) = \int_{U_0}V(x)\,dx ,\\
\eta_n := \{\eta_n(V)\} &=& \|V\|^{(\textrm{av})}_{\mathcal{B}, \Omega_n}\,.
\end{eqnarray*} Then there exist constants $C, \;c_1,\;c_2 > 0$ such that
\begin{equation}\label{Nad}
N_-\left(\mathcal{E}_{V,\mathbb{R}^2}\right) \le 1 +C\left(\underset{\{\zeta_n > c_1,\; n\in\mathbb{Z}\}}\sum \sqrt{\zeta_n} + \underset{\{\eta_n > c_2,\; n\in\mathbb{Z}\}}\sum \eta_n \right).
\end{equation}
This result is due to E. Shargorodsky \cite{Eugene} which is an improvement of the estimate by A. Grigor'yan and N. Nadirashivili \cite[Theorem 1.1]{Grig}. Estimates with the type of the first sum in the right hand side of \eqref{Nad} (without explicit constants) were first obtained by M. Birman and M. Solomyak for Schr\"odinger-type operators of order $2\ell$ in
$\mathbb{R}^n$ with $2\ell > n$ (see \cite[\S 6]{BS}). This estimate implies \eqref{ms} (see \cite[Remark 6.2]{Eugene}).\\\\
We denote the polar coordinates in $\mathbb{R}^2$ by $(r, \vartheta)$, \ $r \in \mathbb{R}_+$, \
$\vartheta \in \mathbb{S} := (-\pi, \pi]$.
Let  $I \subseteq \mathbb{R}_+$ be a nonempty open interval and let
$$
\Omega_I := \{x \in \mathbb{R}^2 : \ |x| \in I\} .
$$
We denote by $\mathcal{L}_1\left(I, L_{\mathcal{B}}(\mathbb{S})\right)$ the space of measurable
functions $f : \Omega_I \to \mathbb{C}$ such that
\begin{equation}\label{curlyLnorm}
\|f\|_{\mathcal{L}_1\left(I, L_{\mathcal{B}}(\mathbb{S})\right)} := \int_I
\|f(r, \cdot)\|_{\mathcal{B}, \mathbb{S}}\, r dr < +\infty .
\end{equation}
Let
\begin{equation}\label{AnBnR}
A_0 := \int_{U_0} V(x)\, dx , \ \
A_n := \int_{U_n} V(x) |\ln|x||\, dx , \ n \neq 0\,,
\end{equation}(see \eqref{ringsR})
and
\begin{equation}\label{InDn}
\mathcal{I}_n := (e^n, e^{n + 1}) , \ \ \ \mathcal{D}_n :=
\|V\|_{\mathcal{L}_1\left(\mathcal{I}_n, L_{\mathcal{B}}(\mathbb{S})\right)} , \ \ \ n \in \mathbb{Z} .
\end{equation}
The following result is due to E. Shargorodsky \cite{Eugene}:
\begin{theorem}\label{LaptNetrSol}{\rm \cite[Theorm 7.1]{Eugene}} \\
There exist  constants $C > 0$ and $c > 0$ such that
\begin{equation}\label{LaptNetrSolEst}
N_-\left(\mathcal{E}_{V,\mathbb{R}^2}\right) \le 1 + 4 \sum_{A_n > 1/4}  \sqrt{A_n}
+ C \sum_{\mathcal{D}_n > c} \mathcal{D}_n\,  , \ \ \
\forall V\ge 0.
\end{equation}
\end{theorem}
 This is the sharpest  semi-classical estimate known so far.  If $N_-(\alpha V, \mathbb{R}^2) = O\left(\alpha \right)$ as $\alpha \longrightarrow +\infty$, then it is necessary that $A_n \in l_{1,w}$ (\cite[Theorem 9.2]{Eugene}). However, $N_-(\alpha V, \mathbb{R}^2) = O\left(\alpha \right)$ as $\alpha \longrightarrow +\infty$ does not imply that the second sum in the right-hand side of \eqref{LaptNetrSolEst} is finite.  Even the finiteness of the right-hand side of \eqref{LaptNetrSolEst} doesn't necessarily imply $N_-\left(\mathcal{E}_{\alpha V,\mathbb{R}^2}\right) = O\left(\alpha\right) \;\;\textrm{as} \;\; \alpha \to +\infty$ (see \cite{Eugene}, section 9 for details). On the other hand, no estimate of the type
 $$
 N_-(\alpha V, \mathbb{R}^2) \le \textrm{const} + \int_{\mathbb{R}^2}V(x)W(x)\,dx + \textrm{const}\|V\|_{\Psi, \mathbb{R}^2}
 $$
 can hold  with an Orlciz norm $\|.\|_{\Psi, \mathbb{R}^2}$ weaker than $\|.\|_{\mathcal{B}, \mathbb{R}^2}$  provided the weight function $W$ is bounded in the neighbourhood of at least one point (see \cite[Theorem 9.4]{Eugene}).
  If $V(x) = V(|x|)$, the last term in \eqref{LaptNetrSolEst} can be dropped and one gets the result of A. Laptev and M. Solomyak \cite{LapSolo} for radial potentials through the estimate of the first sum in the right hand side of \eqref{LaptNetrSolEst} by the norm in $l_{1, w}$ (the next remark below). The result of K. Chadan, N. Khuri and A. Martin (see \eqref{chad}) is a direct consequence of the result by A. Laptev and M. Solomyak. This is due to the estimate of the $l_{1,w}$-quasinorm through the norm in $l_1$.
\begin{remark}
{\rm For any $c > 0$
$$
\sum_{A_n > c}  \sqrt{A_n} \le \frac{2}{\sqrt{c}}\|(A_n)_{n\in\mathbb{Z}}\|_{1,w}
$$
(see (49) in \cite{Eugene}).
$$
\sum_{\mathcal{D}_n > c} \mathcal{D}_n \le \sum_{n\in\mathbb{Z}} \mathcal{D}_n  = \int_{\mathbb{R}_+}\|V\|_{\mathbb{B}, \mathbb{S}}\,r\,dr.
$$
Hence \eqref{LaptNetrSolEst} implies the following estimate
\begin{equation}\label{equiv1}
N_-\left(\mathcal{E}_{V,\mathbb{R}^2}\right) \le 1 + \textrm{const}\left(\|(A_n)_{n\in\mathbb{Z}}\|_{1,w} + \int_{\mathbb{R}_+}\|V\|_{\mathbb{B}, \mathbb{S}}\,r\,dr\right),
\end{equation}
}
\end{remark}which in turn implies the result of A. Laptev and M. Solomyak \cite[Theorem 1.1]{LapSolo} since
$$
\int_{\mathbb{R}_+}\|V\|_{\mathbb{B}, \mathbb{S}}\,r\,dr \le C(p)\int_{\mathbb{R}_+}\|V\|_{L_p(\mathbb{S})}\,r\,dr\,,\;\;\;p > 1.
$$
Let $V_{\ast} : \mathbb{R}_+ \longrightarrow [0, +\infty)$ be a non-increasing spherical rearrangement of $V$, i.e. a non-increasing right continuous function such that
$$
\left|\{x\in\mathbb{R}^2\;:\;V_{\ast}(|x|) > s\}\right| = \left|\{x\in\mathbb{R}^2\;:\;V(x) > s\}\right|, \;\;\;\forall s > 0,
$$ where $|\cdot|$ denotes the two dimensional Lebesgue measure. The following estimate was conjectured in 2002 by Khuri-Martin-Wu \cite{KMW}.
\begin{equation}\label{khuri}
N_-\left(\mathcal{E}_{V,\mathbb{R}^2}\right) \le 1 + C\left(\int_{\mathbb{R}^2}V(x)\ln\left(2 + |x|\right)\,dx + \int_{|x| < 1}V_{\ast}(|x|)\ln \frac{1}{|x|}\,dx\right).
\end{equation}
Using \eqref{equiv1}, it is shown in \cite{Eugene} that \eqref{khuri} is a direct consequence of the result by M. Solomyak in Theorem \ref{MS} and that the latter is strictly sharper. This means that actually the Khuri-Martin-Wu conjecture was proved before it was stated.\\\\

Let
$$
S:= \{(x_1, x_2)\in\mathbb{R}^2 : x_1\in\mathbb{R},\, 0 < x_2 < a\}
$$ be a strip with Neumann boundary conditions.
For a given potential $V$ and for any $n\in\mathbb{Z}$, set
$$
a_n(V) = \int_{S_n}V(x)\left(1 + \mid x_1\mid\right)\,dx,
$$ where $S_n = I_n \times I, \;\;\;I = (0, a),$
\begin{eqnarray*}
I_n := [2^{n - 1}, 2^n], \ n > 0 , \ \ \ I_0 := [-1, 1] , \ \ \
I_n := [-2^{|n|}, -2^{|n| - 1}], \ n < 0 .
\end{eqnarray*}
For $p > 1$,  also set
$$
b_n(V) = \left(\int_{S\cap\{n < x_1 < n + 1\}} V^p\, dx\right)^{\frac{1}{p}}.
$$ \begin{eqnarray*}
\mathcal{E}_{V, S}[u] &:=& \int_{S}|\nabla u(x)|^2\,dx - \int_{S}V(x)|u(x)|^2\,dx,\\
\textrm{Dom}\left(\mathcal{E}_{V, S}\right) &=& W^1_2(S)\cap L^2\left(S, V(x)dx\right).
\end{eqnarray*}

\begin{theorem}{\rm \cite[Theorm 7.9]{Grig}}.
Let $V\in L^1_{\textrm{loc}}(S)$. For any $p > 1$ there exists positive constants $C$ and $c$ such that
\begin{equation}\label{Nad*}
N_-\left(\mathcal{E}_{V, S}\right) \leq 1 + C\sum_{\{n\in\mathbb{Z},\; a_n(V)> c\}}\sqrt{a_n(V)} + C\sum_{\{n\in\mathbb{Z},\; b_n(V) > c\}}b_n(V),\end{equation}
where the constants $C$ and $c$ depend only on $p$.
\end{theorem}This result is due to A. Grigor'yan and N. Nadirashvili and it is  the best known estimate for the two dimensional Schr\"odinger operator considered on a strip subject to the Neumann boundary conditions.

\section{Auxiliary results}\label{auxiliary}
Let $I_1, I_2 \subseteq \mathbb{R}$ be nonempty open intervals.
We denote by $L_1\left(I_1, L_{\mathcal{B}}(I_2)\right)$ the space of measurable functions
$f : I_1\times I_2 \to \mathbb{C}$ such that
\begin{equation}\label{L1LlogLnorm}
\|f\|_{L_1\left(I_1, L_{\mathcal{B}}(I_2)\right)} := \int_{I_1}
\|f(x, \cdot)\|^{\rm (av)}_{\mathcal{B}, I_2}\, dx < +\infty\,.
\end{equation}
\begin{lemma}\label{lemma1}
{\rm (  \cite[Lemma 7.3]{Eugene}, see also \cite[Lemma 1]{Sol} )} Consider an affine transformation
$$
\xi : \mathbb{R}^2 \to \mathbb{R}^2, \ \xi(z) := Az + z^0, \ \ \ A = \begin{pmatrix}
    R_1  &  0  \\
  0    &  R_2
\end{pmatrix} , \ R_1, R_2 > 0 , \ \ \ z^0 \in \mathbb{R}^2 .
$$
Let $I_1\times I_2 = \xi(J_1\times J_2)$. Then
\begin{eqnarray*}
\frac{1}{|I_1\times I_2|}\, \|f\|_{L_1\left(I_1, L_{\mathcal{B}}(I_2)\right)} =
\frac{1}{|J_1\times J_2|}\, \|f\circ\xi\|_{L_1\left(J_1, L_{\mathcal{B}}(J_2)\right)}, \\
\forall f \in L_1\left(I_1, L_{\mathcal{B}}(I_2)\right) .
\end{eqnarray*}
\end{lemma}

\begin{lemma}\label{lemma2}{\rm (\cite[Lemma 7.4]{Eugene}, see also  \cite[Lemma 3]{Sol} )}
Let rectangles $I_{1, k}\times I_{2, k}$, $k = 1, \dots, n$ be pairwise disjoint subsets of
$I_1\times I_2$. Then
\begin{equation}\label{subaddest}
\sum_{k = 1}^n \|f\|_{L_1\left(I_{1, k}, L_{\mathcal{B}}(I_{2, k})\right)} \le
\|f\|_{L_1\left(I_1, L_{\mathcal{B}}(I_2)\right)} , \ \ \ \forall f \in L_1\left(I_1, L_{\mathcal{B}}(I_2)\right) .
\end{equation}
\end{lemma}

Let $Q := (0, 1)^2$ and $\mathbb{I} := (0, 1)$.
We will also use the following notation:
$$
w_A := \frac{1}{|A|} \int_A w(z)\, dz ,
$$
where $A \subseteq \mathbb{R}^2$ is a set of a finite two dimensional
Lebesgue measure $|A|$.

\begin{lemma}\label{lemma3}{\rm ( \cite[Lemma 7.5]{Eugene}, see also \cite[Lemma 2]{Sol})}
There exists a constant $C_{1} > 0$ such that for any nonempty open intervals
$I_1, I_2 \subseteq \mathbb{R}$ of lengths $R_1$ and $R_2$ respectively, any
$w \in W^1_2(I_1\times I_2)\cap C\left(\overline{I_1\times I_2}\right)$ with
$w_{I_1\times I_2} = 0$, and
any $V \in L_1\left(I_1, L_{\mathcal{B}}(I_2)\right)$, $V \ge 0$
the following inequality holds:
\begin{eqnarray}\label{eqnl3}
&& \int_{I_1\times I_2}  V(z) |w(z)|^2\, dz \nonumber \\
&& \le C_{1} \max\left\{\frac{R_1}{R_2}, \frac{R_2}{R_1}\right\}
\|V\|_{L_1\left(I_1, L_{\mathcal{B}}(I_2)\right)} \int_{I_1\times I_2} |\nabla w(z)|^2\, dz .
\end{eqnarray}
\end{lemma}

\begin{lemma}\label{lemma4}{\rm (\cite[Lemma 7.6]{Eugene}, see also  \cite[Theorem 1]{Sol})}\\
For any $V \in L_1\left(\mathbb{I}, L_{\mathcal{B}}(\mathbb{I})\right)$, $V \ge 0$ and any
$n \in \mathbb{N}$  there exists a finite cover of $Q$ by
rectangles $A_k = I_{1, k}\times I_{2, k}$, $k = 1, \dots, n_0$  such that $n_0 \le n$ and
\begin{equation}\label{eqnl4}
\int_Q  V(z) |w(z)|^2\, dz \le
C_{2} n^{-1}\|V\|_{L_1\left(\mathbb{I}, L_{\mathcal{B}}(\mathbb{I})\right)}
\int_{Q} |\nabla w(z)|^2\, dz
\end{equation}
for all
$w \in W^1_2(Q)\cap C\left(\overline{Q}\right)$ with $w_{A_k} = 0$, $k = 1, \dots, n_0$, where
the constant $C_{2}$ does not depend on $V$.
\end{lemma}

Let
\begin{eqnarray*}
& \mathcal{E}_{V, Q}[w] : = \int_{Q} |\nabla w(z)|^2 dz -
\int_{Q} V(z) |w(z)|^2 dz , & \\
& \mbox{Dom}\, (\mathcal{E}_{V, Q}) =
W^1_2\left(Q\right)\cap L^2\left(Q, V(z)dz\right) . &
\end{eqnarray*}

\begin{lemma}\label{lemma5}{\rm (\cite[Lemma 7.7]{Eugene}, see also \cite[Theorem 4]{Sol})}
$$
N_- (\mathcal{E}_{V, Q}) \le C_{2}
\|V\|_{L_1\left(\mathbb{I}, L_{\mathcal{B}}(\mathbb{I})\right)} + 1 , \ \ \ \forall V \ge 0 ,
$$
where  $C_2$ is the constant from Lemma \ref{lemma4}.
\end{lemma}

\subsection{Estimates for one-dimensional Schr\"odinger operators}\label{1-destimate}
In this subsection we present the estimates for the number of negative eigenvalues of one dimensional Schr\"odinger operators that we use in the sequel.\\\\
Let $I$ be a finite interval in $\mathbb{R}$ of length $l$. For simplicity, take $I = (0, l)$. Let $0 = t_0 < t_1 < ... < t_n = l$ be a partition of the interval $I$ into $n$ subintervals $I_k = (t_{k-1}, t_{k})$. Let $P$ stand for any such partition and $|P|$ denote the number of subintervals, i.e, $|P| = n$. Let $\nu$ be a positive Radon measure on $\mathbb{R}$ and for any real number $a > 0$, consider the following function of partitions:
\begin{equation}\label{par1}
\Theta_a(P) := \underset{k}\max \left( t_k - t_{k-1}\right)^a \nu(I_k).
\end{equation}
\begin{lemma}\label{parlemma1}
{\rm Suppose $\nu(\{x\}) = 0$ for all $x\in\mathbb{R}$. Then for any $n \in\mathbb{N}$ there exists a partition $P$ of the interval $I$ such that $|P| = n$ and \begin{equation}\label{par2}
\Theta_a(P) \le l^a n^{-1-a}\nu(I).
\end{equation}
}
\end{lemma}
\begin{proof}
The proof follows a similar argument as in the proof of \cite[Lemma 7.1]{Sol2} for measures absolutely continuous with respect to the Lesbegue measure.  By scaling, it is enough to prove \eqref{par2} for $l = 1$ and $\nu(I) = 1$. For $n = 1$ there is nothing to prove. Now suppose \eqref{par2} is true for some $n$, we show that this is true for $n + 1$. Since $x \longmapsto \nu([x, 1))$ is continuous, there exists a point $x\in (0, 1)$ such that
\begin{equation}\label{cts}
(1 - x)^a\nu([x, 1)) = (n + 1)^{-1-a}.
\end{equation}
Then one has
$$
\nu([x, 1)) = (n + 1)^{-1-a}(1 - x)^{-a}.
$$
By the induction assumption, there exists a partition $P_0$ of the interval $(0, x)$ into $n$ subintervals $ 0 = t_0 < t_1 < ... < t_n = x$ such that
\begin{eqnarray*}
\Theta_a(P_0) &\le& x^a n^{-1-a}\nu((0, x))\\ &=& x^a n^{-1-a}\left( 1 - (n + 1)^{-1-a}(1 - x)^{-a}\right).
\end{eqnarray*}
Then $P$ is $0 = t_0 < t_1 < ... < t_n < t_{n + 1} = 1$.
To prove \eqref{par2}, due to \eqref{cts} it is sufficient to show that $\Theta_a(P_0) \le (n + 1)^{-1-a}$. We achieve this by showing that
$$
n^{-1-a} \le (n + 1)^{-1-a}x^{-a} + n^{-1-a}(n + 1)^{-1-a}(1 -x)^{-a}.
$$
Let $h(x) = (n + 1)^{-1-a}x^{-a} + n^{-1-a}(n + 1)^{-1-a}(1 -x)^{-a}.$ Then $h$ is convex on $(0, 1)$ and solving $h'(x) = 0$ we see that $h$ attains its minimum on $(0, 1)$ at the point $x = n(n + 1)^{-1}$ and this this minimum value is $n^{-1-a}$.
\end{proof}
\begin{lemma}\label{parlemma2}
{\rm Suppose $\nu(\{t\}) = 0$ for all $t\in I$. For any $n\in\mathbb{N}$, there exists a partition $P$ of the interval $I$ such that $|P| = n$ and
\begin{equation}\label{par3}
\int_I|u(t)|^2d\nu(t) \le \frac{l}{ n^{2}}\nu(I)\int_I|u'(t)|^2\,dt,
\end{equation}
 for all $u\in \mathcal{L}_n$ and $\mathcal{L}_n$ is the subspace of $W^1_2(I)$ of co-dimension $n$ formed by the functions  satisfying $u(t_1) = ... = u(t_n) = 0$.
}
\end{lemma}
\begin{proof}
For any $t\in I_k$, the Cauchy-Schwartz inequality implies
\begin{eqnarray*}
|u(t)|^2 &=& |u(t) - u(t_k)|^2 = \left|\int_{t}^{t_k}u'(s)\,ds\right|^2\\
&\le& |t - t_k|\int_t^{t_k}|u'(s)|^2\,ds\\&\le& |t_{k} - t_{k-1}|\int_{t_{k-1}}^{t_k}|u'(s)|^2\,ds.
\end{eqnarray*}
Hence
\begin{eqnarray*}
\int_{I_k}|u(t)|^2\,d\nu(t) &\le& \underset{t\in I_k}\sup |u(t)|^2\nu(I_k)\\ &\le& |t_k - t_{k-1}|\nu(I_k)\int_{t_{k-1}}^{t_k}|u'(s)|^2\,ds.
\end{eqnarray*}
With $a = 1$, \eqref{par1} and Lemma \ref{parlemma1} imply
\begin{eqnarray*}
\int_{I}|u(t)|^2\,d\nu(t) &=& \sum_{k = 1}^n \int_{I_k}|u(t)|^2\,d\nu(t)\\
&\le&\sum_{k = 1}^n |t_k - t_{k-1}|\nu(I_k)\int_{t_{k-1}}^{t_k}|u'(s)|^2\,ds\\
&\le& \Theta_a(P)\sum_{k = 1}^n \int_{I_k}|u'(s)|^2\,ds\\ &\le& \frac{l}{ n^{2}}\nu(I)\int_I|u'(s)|^2\,ds.
\end{eqnarray*}
\end{proof}
The above Lemma excludes measures with atoms. However, one can show that the lemma still holds true even when $\nu$ has atoms by approximating $\nu$ by measures that are absolutely continuous with respect to the Lebesgue measure.
\begin{lemma}\label{acmeasure}
{\rm  For any $c >1 $ and any $n\in\mathbb{N}$ there exists a partition $P$ of $I$ such that $|P| = n$ and
\begin{equation}\label{aceqn}
\int_{I}|u(t)|^2d\nu(t) \le c\, \frac{l}{ n^{2}}\nu(I)\int_{I}|u'(t)|^2\,dt,
\end{equation}
for all $u\in W^1_2(I)$ such that $u(t_1) = u(t_2) = ... = u(t_n) = 0$.
}
\end{lemma}
\begin{proof}
Let $\varphi \in C_0^{\infty}(\mathbb{R})$  such that $\varphi(t) = 0$ if $|t| \geq 1$ and $\int_{\mathbb{R}}\varphi(t)\,dt = 1$. For $\varepsilon > 0$, let $\varphi_{\varepsilon}(t) = \frac{1}{\varepsilon}\varphi(\frac{t}{\varepsilon})$. Then $\varphi_{\varepsilon}(t) = 0$ if $|t| \geq \varepsilon$ and $\int_{\mathbb{R}}\varphi_{\varepsilon}(t)\,dt = 1$.  Extend $\nu$ to $\mathbb{R}$ by $\nu(J) = 0$ for $J = \mathbb{R}\setminus I$.
Let $\nu_{\varepsilon} := \nu \ast \varphi_{\varepsilon}$, i.e.,
$$
d\nu_{\varepsilon}(t) = \left(\int_{\mathbb{R}}\varphi_{\varepsilon}(t - y)\,d\nu(y)\right)dt.
$$ Then $\textrm{supp}\, \nu_{\varepsilon}\subseteq I_{\varepsilon}$, where $I_{\varepsilon} := [-\varepsilon, l + \varepsilon]$.
By Lemma \ref{parlemma2}, for any $n\in\mathbb{N}$ there exists a partition $P_{\varepsilon} = \{t^{\varepsilon}_0, ... , t^{\varepsilon}_n\}$ of $I_{\varepsilon}$ such that $|P_{\varepsilon}| = n$ and
\begin{equation}\label{aceqn1}
\int_{I_{\varepsilon}}|u_{\varepsilon}(t)|^2d\nu_{\varepsilon}(t) \le  \frac{l} {n^{2}}\nu_{\varepsilon}(I_{\varepsilon})\int_{I_{\varepsilon}}|u'_{\varepsilon}(t)|^2\,dt,
\end{equation}
for all $u_{\varepsilon}\in W^1_2(I_{\varepsilon})$ such that $u(t^{\varepsilon}_1) = ... = u(t^{\varepsilon}_n) = 0$ .\\\\
Let
$$
\xi(x) := \frac{l + 2\varepsilon}{l}x - \varepsilon.
$$ Then
$$
\xi^{-1}(y) = \frac{l}{l + 2\varepsilon}(y + \varepsilon)
$$
and
$$
\xi : I \longrightarrow I_{\varepsilon} , \;\;\; \xi^{-1} : I_{\varepsilon} \longrightarrow I
.$$
Let $$t _k = \xi^{-1}\left(t^{\varepsilon}_k\right), \;\; k = 0, ... , n.$$ Take any $u \in W^1_2(I)$ such that $u(t_1) = ... = u(t_n)$. Consider $$u_{\varepsilon}(y) := u(\xi^{-1}(y)).$$ Then $u_{\varepsilon}\in W^1_2(I_{\varepsilon})$ and $u_{\varepsilon}(t^{\varepsilon}_1) = ... = u_{\varepsilon}(t^{\varepsilon}_n) = 0$, so \eqref{aceqn1} holds.\\
Now,
\begin{eqnarray}\label{aceqn2}
\nu_{\varepsilon}(I_{\varepsilon}) &=& \int_{I_{\varepsilon}}\int_{\mathbb{R}}\varphi_{\varepsilon}(t - y)d\nu(y)\,dt
=\int_{\mathbb{R}}\int_{I_{\varepsilon}}\varphi_{\varepsilon}(t - y)\,dt d\nu(y)\nonumber\\&=&\int_{I}\int_{I_{\varepsilon}}\varphi_{\varepsilon}(t - y)\,dt d\nu(y) = \int_{I}\int_{\mathbb{R}}\varphi_{\varepsilon}(t - y)\,dt d\nu(y)\nonumber\\&=& \int_{I}d\nu(y) = \nu(I),
\end{eqnarray}
\begin{eqnarray}\label{aceqn3}
\int_{I_{\varepsilon}}|u'_{\varepsilon}(t)|^2\,dt &=& \int_{I_{\varepsilon}}\left|\frac{d}{dt}u(\xi^{-1}(t))\right|^2\,dt \nonumber\\ &=& \frac{l}{l + 2\varepsilon}\int_{I}|u'(x)|^2\,dx \nonumber\\ &\le& \int_{I}|u'(x)|^2\,dx.
\end{eqnarray}
\begin{eqnarray*}
&&\left|\int_{I}|u(y)|^2\,d\nu(y) - \int_{I_{\varepsilon}}|u_{\varepsilon}(t)|^2\,d\nu_{\varepsilon}(t) \right|\\&&= \left|\int_{\mathbb{R}}|u(y)|^2\,d\nu(y) - \int_{\mathbb{R}}|u_{\varepsilon}(t)|^2\,d\nu_{\varepsilon}(t) \right|\\&&= \left|\int_{\mathbb{R}}|u(y)|^2\,d\nu(y) - \int_{\mathbb{R}}|u_{\varepsilon}(t)|^2 \int_{\mathbb{R}}\varphi_{\varepsilon}(t - y)d\nu(y)\,dt \right|\\&&=
\left|\int_{\mathbb{R}}|u(y)|^2\,d\nu(y) - \int_{\mathbb{R}}\int_{\mathbb{R}}|u_{\varepsilon}(\tau + y)|^2 \varphi_{\varepsilon}(\tau)d\tau\,d\nu(y) \right|\\&&\le \int_{\mathbb{R}}\int_{\mathbb{R}}\left| |u(y)|^2- |u_{\varepsilon}(\tau + y)|^2\right|\varphi_{\varepsilon}(\tau)d\tau d\nu(y)\\&&\le
\underset{\underset{|\tau|\le\varepsilon}{y\in I}}\max \left||u(y)|^2- |u_{\varepsilon}(\tau + y)|^2\right|\nu(I)
\end{eqnarray*}
\begin{eqnarray*}
 |u(y)|^2- |u_{\varepsilon}(\tau + y)|^2 &=& |u(y)|^2 - \left|u\left(\frac{l}{l + 2\varepsilon}(y + \tau + \varepsilon)\right)\right|^2\\&\le& |u(y)| - \left|u\left(\frac{l}{l + 2\varepsilon}(y + \tau + \varepsilon)\right)\right|\\&&\times\left(|u(y)| + \left|u\left(\frac{l}{l + 2\varepsilon}(y + \tau + \varepsilon)\right)\right|\right)\\ &\le&2\sqrt{|I|}\left|y - \frac{l}{l + 2\varepsilon}(y + \tau + \varepsilon)\right|^{\frac{1}{2}}\|u'\|^2_{L^2}\\&=&2\sqrt{l}\sqrt{\frac{1}{l + 2\varepsilon}}\underbrace{\left|2\varepsilon y - l\tau - l\varepsilon\right|^{\frac{1}{2}}}_{\le \sqrt{4l\varepsilon}} \|u'\|^2_{L^2} \\&\le& 4\sqrt{l} \sqrt{\frac{l}{l + 2\varepsilon}} \sqrt{\varepsilon}\|u'\|^2_{L^2} \\&\le&  4\sqrt{l}  \sqrt{\varepsilon}\|u'\|^2_{L^2}.
\end{eqnarray*}
Hence
$$
\left|\int_{I}|u(y)|^2\,d\nu(y) - \int_{I_{\varepsilon}}|u_{\varepsilon}(t)|^2\,d\nu_{\varepsilon}(t) \right| \le 4\sqrt{l}  \sqrt{\varepsilon}\|u'\|^2_{L^2}\nu(I).
$$
This combined with \eqref{aceqn1},\eqref{aceqn2} and \eqref{aceqn3} imply
\begin{eqnarray*}
\int_{I}|u(y)|^2\,d\nu(y) &\le& \int_{I_{\varepsilon}}|u_{\varepsilon}(t)|^2\,d\nu_{\varepsilon}(t) + 4\sqrt{l}  \sqrt{\varepsilon}\|u'\|^2_{L^2}\nu(I) \\ &\le&
ln^{-2}\nu(I)\int_I|u'(x)|^2\,dx + 4\sqrt{l}  \sqrt{\varepsilon}\nu(I)\int_{I}|u'(x)|^2\,dx\\ &=&\left(ln^{-2} + 4\sqrt{l}  \sqrt{\varepsilon}  \right)\nu(I)\int_{I}|u'(x)|^2\,dx.
\end{eqnarray*}
Now choose $\varepsilon > 0$ such that $\left(ln^{-2} + 4\sqrt{l}  \sqrt{\varepsilon}  \right)\nu(I) \le c\;ln^{-2}\nu(I)$, i.e.,
$$
\varepsilon \le \left(\frac{c- 1}{4n^2}\right)^2l.
$$
\end{proof}

Let $0 < a < b$. It follows from the embedding $W^1_2([a , b]) \hookrightarrow
C([a, b])$ that there exist constants $\alpha, \beta > 0$ such that
$$
\frac{|u(x)|^2}{|x|} \le \alpha \int_{a}^b |u'(t)|^2\, dt +  \beta \int_{a}^b \frac{|u(t)|^2}{|t|^2}\, dt ,  \ \
\forall u \in W^1_2([a, b]) \mbox{ and } \forall x\in [a, b] .
$$
 Since there are two constants involved here, it is convenient to rewrite the inequality in the following
form
\begin{equation}\label{1dHSob}
\frac{|u(x)|^2}{|x|} \le C(\kappa) \left(\int_{a}^b |u'(t)|^2\, dt +  \kappa \int_{a}^b \frac{|u(t)|^2}{|t|^2}\, dt\right) ,
\end{equation}$\forall u \in W^1_2([a, b])\mbox{ and } \forall x\in [a, b] $,
and to look for the best value of $C(\kappa)$ for a given $\kappa > 0$.
The best value of $C(\kappa) > 0$ is given by
\begin{equation}\label{1dHOpt}
C(\kappa) = \frac{1}{2\kappa}\,\left(1 + \sqrt{1 + 4\kappa}\,
\frac{b^{\sqrt{1 + 4\kappa}} + a^{\sqrt{1 + 4\kappa}}}{b^{\sqrt{1 + 4\kappa}} - a^{\sqrt{1 + 4\kappa}}}\right)\,,
\end{equation}
(see \cite[Appendix A]{Eugene}).
\begin{remark}\label{HSob01}
{\rm Suppose $u \in W^1_2([0, 1])$ and $u(0) = 0$. Then using \eqref{1dHSob},
\eqref{1dHOpt} with $b = 1$ and $a \to 0+$ one gets
\begin{eqnarray*}
\frac{|u(x)|^2}{|x|} \le \frac{1}{2\kappa}\,\left(1 + \sqrt{1 + 4\kappa}\right)
\left(\int_0^1 |u'(t)|^2\, dt +  \kappa \int_0^1 \frac{|u(t)|^2}{|t|^2}\, dt\right) ,   \\
\forall x \in (0, 1],
\end{eqnarray*}
and the right-hand side is finite due to Hardy's inequality. Note that
$$
\frac{1}{2\kappa}\,\left(1 + \sqrt{1 + 4\kappa}\right) < C(\kappa)
$$
for any $b > a > 0$ (see \eqref{1dHOpt}).}
\end{remark}

Let $\nu$ be a $\sigma$-finite positive Radon measure on $\mathbb{R}$. Consider the following operator on $L^2(\mathbb{R})$
$$
H_{\nu} := - \frac{d^2}{dx^2} - \nu\;.
$$
Define $H_{\nu}$ via its quadratic form
\begin{eqnarray*}
\mathcal{E}_{\mathbb{R}, \nu}[u] &:=& \int_{\mathbb{R}}|u'(x)|^2dx - \int_{\mathbb{R}}|u(x)|^2\, d\nu(x),\\
\textrm{Dom}(\mathcal{E}_{\mathbb{R}, \nu}) &:=& W^1_2(\mathbb{R})\cap L^2(\mathbb{R}, d\nu).
\end{eqnarray*}
Let
\begin{eqnarray*}
X: &=& W^1_2(\mathbb{R}),\\
X_0 &:=& \left\{ u\in X\; : \; u(0) = 0\right\},\\
X_1 &:=& \left\{u\in W^1_{2,\textrm{loc}}(\mathbb{R})\; :\; u(0) = 0, \; \int_\mathbb{R}\mid u'(x)\mid^2 dx < \infty\right\}.
\end{eqnarray*}
Then, $\textrm{dim}(X/X_0) = 1$ and $X_0 \subset X_1$. Let $\mathcal{E}_{X, \nu}$ ,$\mathcal{E}_{X_0, \nu}$ and $\mathcal{E}_{X_1, \nu}$ denote the forms $$\int_\mathbb{R} \mid u'(x)\mid^2 dx - \int_\mathbb{R} \mid u(x)\mid^2 d\nu(x)$$ on the domains $X \cap L^2(\mathbb{R}, d\nu),\; X_0 \cap L^2(\mathbb{R}, d\nu)$ and $X \cap L^2(\mathbb{R},  d\nu)$ respectively. Then
\begin{equation}\label{six}
N_-(\mathcal{E}_{\mathbb{R}, \nu}) =  N_-(\mathcal{E}_{X, \nu})  \leq N_-(\mathcal{E}_{X_0, \nu}) + 1 \leq N_-(\mathcal{E}_{X_1, \nu}) + 1.
\end{equation}
An estimate for the right hand of \eqref{six} is presented in \cite{Eugene} when $\nu$ is absolutely continuous with respect to the Lebesgue measure. We follow a similar argument.
It follows from Hardy's inequality (see, e.g., \cite[Theorem 327]{HLP}) that
\begin{eqnarray*}
\int_{\mathbb{R}} |u'(x)|^2\, dx + \kappa\, \int_{\mathbb{R}} \frac{|u(x)|^2}{|x|^2}\, dx &\le&
\int_{\mathbb{R}} |u'(x)|^2\, dx + 4\kappa \,\int_{\mathbb{R}} |u'(x)|^2\, dx \\&=&
(4\kappa + 1) \int_{\mathbb{R}} |u'(x)|^2\, dx ,  \ \ \ \forall u \in X_1 , \ \
\forall \kappa \ge 0 .
\end{eqnarray*}
Hence
\begin{equation}\label{kappaV}
N_- (\mathcal{E}_{X_1, \nu}) \le N_- (\mathcal{E}_{\kappa, \nu}) ,
\end{equation}
where
\begin{eqnarray*}
\mathcal{E}_{\kappa, \nu}[u] &:=& \int_{\mathbb{R}} |u'(x)|^2\, dx +
\kappa\, \int_{\mathbb{R}} \frac{|u(x)|^2}{|x|^2}\, dx -
(4\kappa + 1)\, \int_{\mathbb{R}}  |u(x)|^2\, d\nu(x) , \\
\mbox{Dom}\, (\mathcal{E}_{\kappa, \nu}) &=& X_1\cap L^2\left(\mathbb{R}, d\nu\right) .
\end{eqnarray*}
It follows from \eqref{six} and \eqref{kappaV} that
\begin{equation}\label{VkappaV}
N_- (\mathcal{E}_{\mathbb{R}, \nu}) \le N_- (\mathcal{E}_{\kappa, \nu}) + 1 .
\end{equation}
Let us  partition $\mathbb{R}$ into the intervals $I_n$ as follows:
\begin{eqnarray}\label{interval}
I_n := [2^{n - 1}, 2^n], \ n > 0 ,   \ I_0 := [-1, 1] ,  \ \
I_n := [-2^{|n|}, -2^{|n| - 1}], \ n < 0.
\end{eqnarray}
Then the variational principle (see \eqref{parti}) implies
\begin{equation}\label{VkappaVn}
N_- (\mathcal{E}_{\kappa, \nu}) \le \sum_{n \in \mathbb{Z}} N_- (\mathcal{E}_{\kappa, \nu, n}) ,
\end{equation}
where
\begin{eqnarray*}
&& \mathcal{E}_{\kappa, \nu, n}[u] := \int_{I_n} |u'(x)|^2\, dx +
\kappa\, \int_{I_n} \frac{|u(x)|^2}{|x|^2}\, dx -
(4\kappa + 1)\, \int_{I_n}  |u(x)|^2\, d\nu(x) , \\
&& \mbox{Dom}\, (\mathcal{E}_{\kappa, \nu, n}) =
W^1_2(I_n)\cap L^2\left(I_n, d\nu\right) , \ n \in \mathbb{Z}\setminus\{0\} , \\
&& \mbox{Dom}\, (\mathcal{E}_{\kappa, \nu, 0}) =
\{u \in W^1_2(I_0) : \ u(0) = 0\}\cap L^2\left(I_0, d\nu\right) .
\end{eqnarray*}

Let $n > 0$. For any $c >1$ and $N \in \mathbb{N}$, by Lemma \ref{acmeasure} there exists a subspace $\mathcal{L}_N \in
\mbox{Dom}\, (\mathcal{E}_{\kappa, \nu, n})$ of co-dimension $N$ such that
$$
\int_{I_n} |u(x)|^2\, d\nu(x) \le c\left(\frac{|I_n|}{N^2}\, \int_{I_n} \, d\nu(x)\right) \int_{I_n} |u'(x)|^2\, dx , \ \ \
\forall u \in \mathcal{L}_N.
$$
 If
$$
c(4\kappa + 1)\, \frac{|I_n|}{N^2}\, \int_{I_n} \, d\nu(x) \le 1 ,
$$
then $ \mathcal{E}_{\kappa, \nu, n}[u] \ge 0$, $\forall u \in \mathcal{L}_N$, and
$N_- (\mathcal{E}_{\kappa, \nu, n}) \le N$. Let
\begin{equation}\label{calAn}
\mathcal{A}_n := \int_{I_n} |x| \, d\nu(x) , \ n \not= 0 , \ \ \ \mathcal{A}_0 := \int_{I_0} \, d\nu(x) .
\end{equation}
Since $|I_n|\, \int_{I_n} \, d\nu(x) \le \mathcal{A}_n$, $n \not= 0$, it follows from the above that
$$
c(4\kappa + 1) \mathcal{A}_n \le N^2 \ \Longrightarrow \ N_- (\mathcal{E}_{\kappa, \nu, n}) \le N .
$$
Hence
\begin{equation}\label{ceilfn}
N_- (\mathcal{E}_{\kappa, \nu, n}) \le \left\lceil\sqrt{c(4\kappa + 1) \mathcal{A}_n}\right\rceil ,
\end{equation}
where $\lceil\cdot\rceil$ denotes the ceiling function, i.e. $\lceil a\rceil$ is the smallest integer
not less than $a$.
Suppose $\textrm{supp}\nu \cap I_n \neq \{2^{n-1}\}$, i.e., $\nu|_{I_n} \neq$ const.$\delta_{2^{n-1}}$. Then
$$
|I_n|\int_{I_n}d\nu(x) < \mathcal{A}_n\;.
$$
Take $c > 1$ such that
$$
c|I_n|\int_{I_n}d\nu(x) \le \mathcal{A}_n\,.
$$
Then applying Lemma \ref{acmeasure} with this $c$ implies
\begin{equation}\label{ceilfn*}
N_- (\mathcal{E}_{\kappa, \nu, n}) \le \left\lceil\sqrt{(4\kappa + 1) \mathcal{A}_n}\right\rceil .
\end{equation}
If $\nu|_{I_n} =$ const.$\delta_{2^{n-1}}$, then
$$
\int_{I_n} |u(x)|^2\,d\nu(x) = 0
$$ on the the subspace of co-dimension one consisting of functions $u\in W^1_2(I_n)$ such that $u(2^{n-1}) = 0$, and clearly \eqref{ceilfn*} holds.
The right-hand side of \eqref{ceilfn*} is at least 1, so one cannot feed it straight
into \eqref{VkappaVn}. One needs to find conditions under which
$N_- (\mathcal{E}_{\kappa, \nu, n}) = 0$.
 By \eqref{1dHSob},
we have that
\begin{eqnarray*}
\int_{I_n}  |u(x)|^2\, d\nu(x) &\le& C(\kappa)\int_{I_n}|x|\,d\nu(x)  \left(\int_{I_n} |u'(x)|^2\, dx +
\kappa\, \int_{I_n} \frac{|u(x)|^2}{|x|^2}\, dx\right)\\&=&\mathcal{A}_n C(\kappa) \left(\int_{I_n} |u'(x)|^2\, dx +
\kappa\, \int_{I_n} \frac{|u(x)|^2}{|x|^2}\, dx\right),
\end{eqnarray*}for all $u\in W^1_2(I_n)$.

Hence $N_- (\mathcal{E}_{\kappa, \nu, n}) = 0$, i.e. $\mathcal{E}_{\kappa, \nu, n}[u] \ge 0$, provided
$\mathcal{A}_n \le \Phi(\kappa)$, where
\begin{equation}\label{Psikappa}
\Phi(\kappa) := \frac{2\kappa}{4\kappa + 1}\,\left(1 + \sqrt{4\kappa + 1}\,
\frac{2^{\sqrt{4\kappa + 1}} + 1}{2^{\sqrt{4\kappa + 1}} - 1}\right)^{-1}\, .
\end{equation}
The above estimates for $N_- (\mathcal{E}_{\kappa, \nu, n})$ clearly hold for $n < 0$ as well, but
the case $n = 0$ requires some changes. Since $u(0) = 0$ for any
$u \in \mbox{Dom}\, (\mathcal{E}_{\kappa, \nu, 0})$, one can use the same argument as the one
leading to \eqref{ceilfn}, but with two differences: a) $\mathcal{L}_N$ can be chosen to be of
co-dimension $N - 1$, and b) $|I_0|\, \int_{I_0} \, d\nu(x) = 2\mathcal{A}_0$. This gives the following
analogue of \eqref{ceilfn}
$$
N_- (\mathcal{E}_{\kappa, \nu, 0}) \le \left\lceil\,\sqrt{2c(4\kappa + 1) \mathcal{A}_0}\right\rceil - 1 \, .
$$ for any $c > 1$.  We can choose $c > 1$ such that
\begin{equation}\label{ceilfn**}
N_- (\mathcal{E}_{\kappa, \nu, 0}) \le \,\sqrt{2(4\kappa + 1) \mathcal{A}_0}\,,
\end{equation} (see Lemma \ref{c} below).
In particular, $N_- (\mathcal{E}_{\kappa, \nu, 0}) = 0$ if
$\mathcal{A}_0 < 1/(2(4\kappa + 1))$. By Remark \ref{HSob01}, one can easily see that the
implication $\mathcal{A}_n \le \Phi(\kappa) \  \Longrightarrow \
N_- (\mathcal{E}_{\kappa, \nu, n}) = 0$ remains true for $n = 0$.
Now it follows from \eqref{VkappaV} and \eqref{VkappaVn} that
\begin{equation}\label{XGenEst}
N_- (\mathcal{E}_{\mathbb{R}, 2\nu}) \le 1 +
\sum_{\{n \in \mathbb{Z}\setminus\{0\} : \ \mathcal{A}_n > \Phi(\kappa)\}}
\left\lceil\sqrt{(4\kappa + 1) \mathcal{A}_n}\right\rceil +
\,\sqrt{2(4\kappa + 1) \mathcal{A}_0}\, ,
\end{equation}
and one can drop the last term if $\mathcal{A}_0 \le \Phi(\kappa)$. The presence of the
parameter $\kappa$ in this estimate allows a degree of flexibility. In order to decrease the number
of terms in the sum in the right-hand side, one should choose $\kappa$ in such a way that
$\Phi(\kappa)$ is close to its maximum. A Mathematica calculation shows that the maximum is
approximately $0.092$ and is achieved at $\kappa \approx 1.559$. For values of $\kappa$
close to 1.559, one has
$$
\mathcal{A}_n > \Phi(\kappa)  \  \Longrightarrow \
\sqrt{(4\kappa + 1) \mathcal{A}_n}  > \sqrt{(4\kappa + 1) \Phi(\kappa)} \approx 0.816 .
$$
Since $\lceil a\rceil \le 2a$ for $a \ge 1/2$, \eqref{XGenEst} implies
$$
N_- (\mathcal{E}_{\mathbb{R}, \nu}) \le 1 + 2\sqrt{(4\kappa + 1)}
\sum_{\mathcal{A}_n > \Phi(\kappa)}  \sqrt{\mathcal{A}_n}
$$
with $\kappa \approx 1.559$. Hence
\begin{equation}\label{Est1}
N_-(\mathcal{E}_{\mathbb{R}, \nu})\leq 1 + 5.06\sum_{\{n\in\mathbb{Z},\; \mathcal{A}_n > 0.092\}}\sqrt{\mathcal{A}_n}\,.
\end{equation}
\begin{lemma}\label{c}
 {\rm For every $y\in\mathbb{R}_+$, there exists  $c > 1$ such that
$$
\lceil cy \rceil - 1 \le y\,.
$$
}
\end{lemma}
\begin{proof}
Case 1: Suppose $y\in\mathbb{R}_+\backslash\mathbb{Z}_+$. Then there exists $l\in\mathbb{Z}_+$ such that
$$
l < y < l + 1\,.
$$
Take $c > 1$ such that
$$
l < cy < l + 1\,.
$$
Then
$$
\lceil cy\rceil - 1 = l + 1 - 1 = l < y\,.
$$
Case 2: Suppose $y \in\mathbb{Z}_+$. Take $c > 1$ such that
$$
cy < y + 1\,.
$$
Then
$$
\lceil cy \rceil - 1 = y+ 1 - 1  = y\,.
$$

\end{proof}

\chapter{Two dimensional Schr\"odinger Operators with potentials generated by a Radon measure}\markboth{Chapter \ref{Twostrip}\label{measure}.
Introduction}{}\label{Introduction}
\section{Introduction}\label{measure-introduc}
In this chapter, we obtain upper estimates for the number of negative eigenvalues of two dimensional Schr\"odinger operators with potentials of the form $V\mu$, where $\mu$ is a $\sigma$-finite positive Radon measure on $\mathbb{R}^2$  and $V \ge 0$ is a real valued function locally integrable on $\mathbb{R}^2$ with respect to $\mu$  as in $\S$ \ref{index}.\\\\
Consider the operator on $L^2(\mathbb{R}^2)$,
\begin{equation}\label{measure1}
H_{V\mu} := -\Delta - V\mu,\;\;\;\; V\geq 0.
\end{equation}  We define $H_{V\mu}$ via its quadratic form by
\begin{equation}\label{measure2}
\mathcal{E}_{V\mu, \mathbb{R}^2}[u] := \int_{\mathbb{R}^2}|\nabla u|^2dx - \int_{\mathbb{R}^2}V|u|^2d\mu
\end{equation}
with domain $W^1_2(\mathbb{R}^2)\cap L^2(\mathbb{R}^2, Vd\mu)$. We shall denote by $N_-\left(\mathcal{E}_{V\mu, \mathbb{R}^2}\right)$ the number of negative eigenvalues (counted with multiplicities) of the operator \eqref{measure1}.   \\\\
If one introduces the ``supporting'' term $|x|^{-2}$ to the Laplacian in \eqref{measure1}, we have
\begin{equation}\label{measure2}
T_{V\mu} := -\Delta + |x|^{-2} - V\mu,\;\;\;\; V\geq 0.
\end{equation}
Under certain restrictions on $\mu$, it was shown by A. Laptev and Yu. Nestrusov \cite{LapNest} that
\begin{equation}\label{lapnest}
N_-\left(\mathcal{E}_{T, V\mu, \mathbb{R}^2}\right) \le C\int_{\mathbb{R}^2}V(x)\,d\mu(x),
\end{equation}
where $N_-\left(\mathcal{E}_{T, V\mu, \mathbb{R}^2}\right)$ denotes the number of negative eigenvalues of $T_{V\mu}$ and $C$ is a constant independent of $V$ and $\mu$. They also proved if $\mu$ is absolutely continuous with respective to the Lebesgue measure, then
\begin{equation}\label{lapnest1}
N_-\left(\mathcal{E}_{T,V\mu, \mathbb{R}^2}\right) \le C(p)\|V\|_{L_1\left(\mathbb{R}_+, L_p(\mathbb{S})\right)}\;\;\;\;\mathbb{S}= [0, 2\pi],\;\; p > 1, \;\; V\geq 0.
\end{equation} However, \eqref{lapnest1} fails if the supporting term is dropped.\\\\  In this Chapter, we restrict ourselves to the operator \eqref{measure1},  where the measure $\mu$ is of the ``power'' type (see \eqref{Ahlfors}).\\\\
Let $\Psi$ and $\Phi$ be mutually complementary N-functions and let $L_{\Phi}(\Omega, \mu)$ and $L_{\Psi}(\Omega, \mu)$ be the respective Orlicz spaces on a set $\Omega\subseteq\mathbb{R}^2$ of finite measure $\mu(\Omega)$.
We will use the notation $\|f\|_{\Psi, \Omega, \mu}$ and $\|f\|^{\textrm{(\textrm{av})}}_{\Psi, \Omega, \mu}$ respectively  for the norms \eqref{Orlicz} and \eqref{OrlAverage} on $L_{\Psi}(\Omega, \mu)$.\\

\begin{lemma}\label{direction}
{\rm Let $\mu$ be a $\sigma$-finite Radon measure on $\mathbb{R}^2$ such that $\mu(\{x\})= 0,\;\forall x\in\mathbb{R}^2$. Let
\begin{equation}\label{sim}
\Sigma := \left\{\theta \in [0, \pi)\;:\;\exists\;l_{\theta}\mbox{ such that }\mu(l_{\theta}) > 0\right\},
\end{equation}
where $l_{\theta}$ is a line in $\mathbb{R}^2$ in the direction of the vector $(\cos\theta, \sin\theta)$. Then $\Sigma$ is at most countable.
}
\end{lemma}
\begin{proof}
Let $$
\Sigma_N := \left\{\theta \in [0, \pi)\;:\;\exists\; \l_{\theta} \mbox{ such that } \mu(l_{\theta}\cap B(0, N)) > 0\right\},
$$
where $B(0, N)$ is the ball of radius $N\in\mathbb{N}$ centred at $0$. Then
$$
\Sigma = \underset{N\in\mathbb{N}}\cup \Sigma_N.
$$It is now enough to show that $\Sigma_N$ is at most countable for $\forall N\in\mathbb{N}$. Suppose that $\Sigma_N$ is uncountable. Then there exists a $\delta > 0$ such that
$$
\Sigma_{N,\delta} := \left\{\theta \in [0, \pi)\;:\;\exists\; \l_{\theta} \mbox{ such that }\mu(l_{\theta}\cap B(0, N)) > \delta\right\}
$$ is infinite.
Otherwise, $\Sigma_N = \underset{n\in\mathbb{N}}\cup \Sigma_{N, \frac{1}{n}}$ would have been finite or countable. Now take $\theta_1,..., \theta_k,... \in\Sigma_{N, \delta}$. Then
$$
\mu\left(l_{\theta_k}\cap B(0, N)\right) > \delta,\;\;\;\forall k\in\mathbb{N}\,.
$$
Since $l_{\theta_j}\cap l_{\theta_k} ,\;j\neq k$ contains at most one point, then $$\mu\left(\underset{j \neq k}\cup (l_{\theta_j}\cap l_{\theta_k})\right) = 0.$$
Let
$$
\tilde{l}_{\theta_k}:= l_{\theta_k}\backslash\underset{j \neq k}\cup (l_{\theta_j}\cap l_{\theta_k})\,.
$$Then $\tilde{l}_{\theta_j}\cap\tilde{l}_{\theta_k} = \emptyset,\;j \neq k$ and $\tilde{l}_{\theta_k}\cap B(0, N) \subset B(0, N)$.
So
$$
\mu\left(\underset{k\in\mathbb{N}}\cup (\tilde{l}_{\theta_k}\cap B(0, N))\right) \le \mu\left(B(0, N)\right) < \infty\,.
$$ But
$$
\mu\left(\tilde{l}_{\theta_k}\cap B(0, N)\right) = \mu\left(l_{\theta_k}\cap B(0, N)\right) \ge \delta
$$
which implies
$$
\underset{k\in\mathbb{N}}\sum \mu\left( \tilde{l}_{\theta_k}\cap B(0, N)\right) \geq \underset{k\in\mathbb{N}}\sum\delta = \infty\,.
$$
This contradiction means that $\Sigma_N$ is at most countable for each $N\in\mathbb{N}$. Hence $\Sigma$ is at most countable.
\end{proof}
\begin{corollary}\label{cor-direct}
{\rm There exists $\theta_0 \in [0, \pi)$ such that $\theta_0 \notin \Sigma$ and $\theta_0 + \frac{\pi}{2} \notin \Sigma$.
}
\end{corollary}
\begin{proof}
The set
$$
\Sigma - \frac{\pi}{2} := \left\{ \theta - \frac{\pi}{2} \;: \theta\in \Sigma\right\}
$$ is at most countable. This implies that there exists a $\theta_0\notin \Sigma \cup (\Sigma - \frac{\pi}{2})$. Thus $\theta_0 + \frac{\pi}{2}\notin \Sigma$.
\end{proof}
Let $Q$ be an arbitrary unit square with its sides in the directions determined by $\theta_0$ and $\theta_0 + \frac{\pi}{2}$ in Corollary \ref{cor-direct}.  For a given $x\in\overline{Q}$ and $t > 0$, let $Q_x(t)$ be a square centred at $x$ with edges of length $t$ parallel to those of $Q$.
\begin{lemma}\label{measlemma2}
{\rm Suppose that $\Psi$ satisfies the $\Delta_2$-condition (see \ref{global}). Then for all $f \in L_{\Psi}(Q, \mu)$, the function $ f \longmapsto\|f\|^{\textrm{(\textrm{av})}}_{\Psi, Q_x(t), \mu}$ is continuous.}
\end{lemma}
\begin{proof}
For every interval $I\subseteq Q$ parallel to the sides of $Q$, $\mu(I) = 0$. Let $t > t_0 > 0$. Then
\begin{eqnarray*}
0 &\le& \|f\|^{\textrm{(av)}}_{\Psi, Q_x(t), \mu} - \|f\|^{\textrm{(av)}}_{\Psi, Q_x(t_0), \mu}\\&=& \textrm{sup}\left\{ \left|\int_{Q_x(t)}fg\;d\mu\right| : \int_{Q_x(t)}\Phi(|g(x)|)\,d\mu \le \mu(Q_x(t))\right\}\\ &-& \textrm{sup}\left\{ \left|\int_{Q_x(t_0)}fh\;d\mu\right| : \int_{Q_x(t_0)}\Phi(|h(x)|)\,d\mu \le \mu(Q_x(t_0))\right\}.
\end{eqnarray*}
Take $ \rho = \frac{\mu(Q_x(t_0))}{\mu(Q_x(t))} \le 1$ and $h = \rho g$. Then
\begin{eqnarray*}
\int_{Q_x(t_0)}\Phi(|\rho g|)\,d\mu &\le& \int_{Q_x(t)}\Phi(|\rho g|)\,d\mu\\ &=& \rho\int_{Q_x(t)}\Phi(| g|)\,d\mu\\ &\le& \rho\mu(Q_x(t))\\ &\le& \mu(Q_x(t_0)).
\end{eqnarray*}
So,
\begin{eqnarray*}
0 &\le& \|f\|^{\textrm{(av)}}_{\Psi, Q_x(t), \mu} - \|f\|^{\textrm{(av)}}_{\Psi, Q_x(t_0), \mu}\\&=& \textrm{sup}\left\{ \left|\int_{Q_x(t)}fg\;d\mu\right| : \int_{Q_x(t)}\Phi(|g(x)|)\,d\mu \le \mu(Q_x(t))\right\}\\ &-& \textrm{sup}\left\{ \left|\int_{Q_x(t_0)}fh\;d\mu\right| : \int_{Q_x(t_0)}\Phi(|h(x)|)\,d\mu \le \mu(Q_x(t_0))\right\}\\&\le& \textrm{sup}\left\{ \left|\int{Q_x(t)}fg\;d\mu\right| : \int_{Q_x(t)}\Phi(|g(x)|)\,d\mu \le \mu(Q_x(t))\right\}\\ &-& \textrm{sup}\left\{\rho \left|\int_{Q_x(t_0)}fg\;d\mu\right| : \int_{Q_x(t)}\Phi(|g(x)|)\,d\mu \le \mu(Q_x(t))\right\}\\&\le& \textrm{sup}\left\{\left|\int_{Q_x(t)}fg\;d\mu\right|-  \rho\left|\int_{Q_x(t_0)}fg\;d\mu\right| : \int_{Q_x(t)}\Phi(|g(x)|)\,d\mu \le \mu(Q_x(t))\right\}\\&\le&\textrm{sup}\left\{\left|\int_{Q_x(t)\setminus Q_x(t_0)}fg\;d\mu\right| : \int_{Q_x(t)}\Phi(|g(x)|)\,d\mu \le \mu(Q_x(t))\right\} \\&+& (1 - \rho)\textrm{sup}\left\{\left|\int_{Q_x(t_0)}fg\;d\mu\right| : \int_{Q_x(t)}\Phi(|g(x)|)\,d\mu \le \mu(Q_x(t))\right\}.
\end{eqnarray*}
For each $g$, $\int_{Q_x(t)\setminus Q_x(t_0)}fg\;d\mu  \longrightarrow 0$ as $t \longrightarrow t_0$ follows from the absolute continuity of the Lebesgue integral.\\ However, $\underset{g}\sup\left\{\int_{Q_x(t)\setminus Q_x(t_0)}fg\;d\mu :\int_{Q_x(t)}\Phi(|g(x)|)\,d\mu \le \mu(Q_x(t))\right\} \longrightarrow 0$ as $t \longrightarrow t_0$ is not immediate. We can estimate this term using  the H\"older inequality (see \eqref{h2})
\begin{eqnarray*}
&&\underset{g}\sup\left\{\int_{Q_x(t)\setminus Q_x(t_0)}fg\,d\mu  : \int_{Q_x(t)}\Phi(|g(x)|)\,d\mu \le \mu(Q_x(t))\right\}\\&&\le \|f\|_{\left(\Psi, Q_x(t)\setminus Q_x(t_0),\mu\right)}.\sup\left\{\|g\|_{\Phi, Q_x(t)\setminus Q_x(t_0),
\mu} : \int_{Q_x(t)}\Phi(|g(x)|)\,d\mu \le \mu(Q_x(t))\right\} \\&&\le \|f\|_{\left(\Psi,Q_x(t)\setminus Q_x(t_0),\mu\right)}.2 \max\{1, \mu(Q_x(t))\}
\end{eqnarray*}(see \eqref{Luxemburgequiv} and \eqref{LuxNormPre}).
It is clear that $\|f\|_{\left(\Psi, Q_x(t)\setminus Q_x(t_0),\mu\right)}$ is a non-increasing function of $t$. Suppose
$$
\underset{t \longrightarrow t_0}\lim \|f\|_{\left(\Psi, Q_x(t)\setminus Q_x(t_0),\mu\right)} = \kappa_0 > 0.
$$ Since $\Psi$ satisfies the $\Delta_2$ condition, then
$$
\int_{Q_x(t)\setminus Q_x(t_0)}\Psi \left(\frac{|f|}{\kappa_0}\right)\,d\mu \geq 1\,,\;\;\;\; \forall t > t_0
$$ (see (9.21) and (9.22) in \cite{KR})
which contradicts the absolute continuity of the Lebesgue integral. Hence
$$
\underset{t \longrightarrow t_0}\lim \|f\|_{\left(\Psi, Q_x(t)\setminus Q_x(t_0),\mu\right)} = 0.
$$
Now,  $\mu \left(Q_x(t)\setminus Q_x(t_0)\right) \longrightarrow \mu(\partial Q_x(t_0)) = 0$ as $t \longrightarrow t_0$. This implies
$$
\rho = \frac{\mu(Q_x(t_0))}{\mu(Q_x(t))} = 1 - \frac{\mu\left(Q_x(t)\setminus Q_x(t_0)\right)}{\mu(Q_x(t))}\longrightarrow 1 \;\;\textrm{as} \;\;t \longrightarrow t_0.
$$
Hence
$$
(1 - \rho)\textrm{sup}\left\{\left|\int_{Q_x(t_0)}fg\;d\mu\right| : \int_{Q_x(t)}\Phi(|g(x)|)\,d\mu \le \mu(Q_x(t))\right\} \longrightarrow 0
$$ as $t \longrightarrow t_0$. The case $t_0 > t > 0$ is proved similarly and the proof is complete.
\end{proof}

\begin{definition}{\rm (Ahlfors regularity)
 Let $\mu$ be a positive Radon measure on $\mathbb{R}^2$. We say the measure $\mu$ is Ahlfors regular of dimension $\alpha > 0$ if there exist positive constants $c_0$ and $c_1$ such that
\begin{equation}\label{Ahlfors}
c_0r^{\alpha} \le \mu(B(x, r)) \le c_1r^{\alpha}\;
\end{equation}for all $0< r \le$ diam(supp$\,\mu$), where $B(x, r)$ is a ball of radius $r$ centred at $x\in \textrm{supp}\,\mu$ and the constants $c_0$ and $c_1$ are independent of the balls.
}
\end{definition}
 Assume that supp$\,\mu$ is unbounded, so \eqref{Ahlfors} is satisfied for all $ r > 0$. If the measure $\mu$ is $\alpha$-dimensional Ahlfors regular, then it is equivalent to the $\alpha$-dimensional Hausdorff measure (see, e.g., \cite[Lemma 1.2]{Dav} ). For more details and examples of unbounded Ahlfors regular sets, see for example  \cite[Lemma 13.4]{Dav}, \cite{HUT} and \cite{STR} . In the sequel, unless otherwise stated we shall assume that $ 0 < \alpha \le 2$. Suppose that $\mu$ is the usual one-dimensional Lebesgue measure on a horizontal or vertical line, then \eqref{Ahlfors} holds with $\alpha = 1$. This implies $\mu(I) \neq 0$ for every nonempty subinterval $I$ of that line, hence the need of Lemma \ref{direction} and Corollary \ref{cor-direct} for the validity of Lemma \ref{measlemma2} if $\alpha = 1$ .

\section{Statement and proof of the main result}\label{mainresult}
Let
$$
J_n = [e^{2^{n - 1}}, e^{2^n}],\;\;n > 0\;\;\;J_0 := [e^{-1}, e],\;\;\;J_n = [e^{-2^{|n|}}, e^{-2^{|n|-1}}],\;\;n < 0,
$$ and
\begin{equation}\label{meaeqn4}
G_n := \int_{J_n}|\ln|x||V(x)\,d\mu(x), \;\;\; n\neq 0,\;\;\;\; G_0 := \int_{J_0}V(x)d\mu(x).
\end{equation}
Assume without loss of generality that $0\in$ supp$\,\mu$. Further, let
$$
Q_n := \left\{x\in\mathbb{R}^2 \;:\;\left(2\frac{c_1}{c_0}\right)^{\frac{n -1}{\alpha}} \le |x| \le \left(2\frac{c_1}{c_0}\right)^{\frac{n }{\alpha}}\right\}, \; n\in\mathbb{Z}
$$
 and  $$\mathcal{D}_n := \|V\|^{(\textrm{av})}_{\mathcal{B}, Q_n, \mu}\,.$$  Let
\begin{eqnarray*}
\mathcal{E}_{V\mu,\mathbb{R}^2}[w] &:=& \int_{\mathbb{R}^2}|\nabla (x)|^2\,dx - \int_{\mathbb{R}^2}V(x)|w(x)|^2\,d\mu(x)\,,\\
\textrm{Dom}(\mathcal{E}_{V\mu, \mathbb{R}^2}) &=&  W^1_2(\mathbb{R}^2)\cap L^2(\mathbb{R}^2, Vd\mu).
\end{eqnarray*}

Then we have the following theorem:
\begin{theorem}\label{mainthm}Let $\mu$ be a positive Radon  measure on $\mathbb{R}^2$ that is Ahlfors regular and $V\in L_{\mathcal{B}}(Q_n, \mu), \,V\ge 0$. Then
there exist constants $C > 0$ and $c > 0$ such that
\begin{equation}\label{maineqn}
N_-(\mathcal{E}_{V\mu, \mathbb{R}^2}) \le 1 + 4 \sum_{\{n\in\mathbb{Z},\,G_n > 0.25\}}  \sqrt{G_n} + C\underset{\{n\in\mathbb{Z},\;\mathcal{D}_n > c \}}\sum \mathcal{D}_n\,.
\end{equation}
\end{theorem}
\begin{remark}
{\rm If $\alpha = 1$, the above result gives the case when the potential $V$ is supported by a one-dimensional set in $\mathbb{R}^2$ (e.g., a curve that is Ahlfors regular and $\mu$ is the arc length measure of the curve). In fact, Theorem \ref{mainthm} extends the result by E. Shargorodsky (\cite[Theorem 3.1]{Eugene1}) to general Ahlfors regular curves. In his result, he considered the case where the potential is compactly supported on a family of Lipschitz curves in $\mathbb{R}^2$, a condition that is much stronger than Ahlfors regularity. If $\alpha = 2$, then $\mu$ is absolutely continuous with respect to the two dimensional Lebesgue measure, a case that has been widely studied (see $\S$\ref{Literature}). If $\alpha \in (0, 1)\cup(1,2)$, then $\mu$ is supported by sets of fractional dimension (e.g., $\alpha$- Hausdorff dimensional sets) in $\mathbb{R}^2$.
}
\end{remark}
In the proof of Theorem \ref{mainthm}, we will use the variational argument discussed in $\S$\ref{approach}. We will only need to find an estimate for the right-hand side of \eqref{meaeqn1}. We shall start with the first term. Let $I$ be an arbitrary interval in $\mathbb{R}_+$. Define a measure on $\mathbb{R}_+$ by
\begin{equation}\label{1dmeas}
\nu(I) := \int_{|x|\in I}V(x)\,d\mu(x).
\end{equation}
Since $w_{\mathcal{R}}\in L^2(\mathbb{R}^2)$ (see \eqref{radial}), one can approximate $w_{\mathcal{R}}$ by simple measurable functions, i.e.,
$$
w_{\mathcal{R}} = \sum_{k = 1}^Nc_k \chi_{J_k}, \;\; J_k\cap J_j = \emptyset,\;k \neq j,
$$ where $c_k$'s are constants and $\chi_{J_k}$ is the characteristic function of the subinterval $J_k\subseteq I$.
Then
\begin{eqnarray*}
\int_{\mathbb{R}^2}|w_{\mathcal{R}}(x)|^2 V(x)\,d\mu(x) &=& \underset{N\longrightarrow\infty}\lim \sum_{k = 1}^N |c_k|^2 \int_{\mathbb{R}^2}\chi_{J_k}V(x)\,d\mu(x)\\&=& \underset{N\longrightarrow\infty}\lim\sum_{k = 1}^N |c_k|^2 \int_{|x|\in J_k}\chi_{J_k}V(x)\,d\mu(x)\\&=& \underset{N\longrightarrow\infty}\lim\sum_{k = 1}^N |c_k|^2\nu(J_k)\\&=& \int_{\mathbb{R}_+}|w_R(r)|^2d\nu(r).
\end{eqnarray*}

 Let $r = e^t, w(x) = w_{\mathcal{R}}(r) = v(t)$. Then
 $$
 \int_{\mathbb{R}^2}|\nabla w(x)|^2dx = 2\pi\int_{\mathbb{R}}|v'(t)|^2dt
 $$ and
\begin{eqnarray*}
\int_{\mathbb{R}^2}V| w(x)|^2d\mu(x) &=& \int_{\mathbb{R}_+}|w_{\mathcal{R}}(r)|^2d\nu(r)= \int_{\mathbb{R}}|w_{\mathcal{R}}(e^t)|^2d\nu(e^t)\\&=&\int_{\mathbb{R}}|v(t)|^2\,d\nu(e^t).
\end{eqnarray*}
Let
\begin{equation}\label{meaeqn2}
\mathcal{G}_n := \frac{1}{2\pi}\int_{I_n}|t|\,d\nu(e^t), \;\;\; n\neq 0,\;\;\;\; \mathcal{G}_0 := \frac{1}{2\pi}\int_{I_0}d\nu(e^t)\,,
\end{equation}(see \eqref{interval}).
Then similarly to \eqref{Est1} one has
\begin{equation}\label{meaeqn3}
N_-(\mathcal{E}_{\mathcal{R}, 2\nu}) \le  1 + 7.61 \sum_{\{n\in\mathbb{Z},\,\mathcal{G}_n > 0.046\}}  \sqrt{\mathcal{G}_n}\,,
\end{equation}where
\begin{eqnarray*}
\mathcal{E}_{\mathcal{R}, 2\nu}[v] &:=& \int_{\mathbb{R}}|v'(t)|^2\,dt - 2\int_{\mathbb{R}}|v(t)|^2\,d\nu(e^t),\\
\textrm{Dom}(\mathcal{E}_{\mathcal{R}, 2\nu}) &=& W^1_2(\mathbb{R})\cap L^2(\mathbb{R}, d\nu).
\end{eqnarray*}
It follows from \eqref{meaeqn4}, \eqref{1dmeas} and \eqref{meaeqn2} that $G_n = 2\pi\mathcal{G}_n$ and thus \eqref{meaeqn3} implies
\begin{equation}\label{meaeqn5}
N_-(\mathcal{E}_{\mathcal{R}, 2V\mu}) \le  1 + 4 \sum_{\{n\in\mathbb{Z},\,G_n > 0.25\}}  \sqrt{G_n}.
\end{equation} Next, we find an estimate for the second term in the right-hand side of \eqref{meaeqn1}.\\\\
Let $\Omega$ be an arbitray subset of $\mathbb{R}^2$. To the Sobolev space $W^1_2(\Omega)$, we associate a set function called the capacity. Namely, for any compactum $e\subset \Omega$, we put
\begin{equation}\label{capacity}
\textrm{cap}(e, W^1_2(\Omega)) := \inf\left\{\|w\|^2_{W^1_2(\Omega)}\; : \;w\in C_0^{\infty}(\Omega), w \geq 1 \; \textrm{on}\; e\right\}.
\end{equation}
Let $$\mathcal{N}_t(w) = \left\{ x \;:\; |w(x)| \geq t\right\},\;\;w \in W^1_2(\Omega).$$
Then  we have the following capacitary inequality
\begin{equation}\label{capineq}
\int_0^{\infty}\textrm{cap}(\mathcal{N}_t(w), W^1_2(\Omega))t\,dt \le C_3\|w\|^2_{W^1_2(\Omega)},
\end{equation}where $ C_3 > 0$ is a constant independent of $w$ (see \cite{Maz}, (11.2.9)).
\begin{theorem}\label{measthm1}{\rm  \cite[Theorem 11.3.]{Maz}}
{\rm Let $\Psi$ and $\Phi$ be mutually complementary N-functions and let $\mu$ be a positive Radon measure on $\mathbb{R}^2$. Further, let $\Omega$ be a domain in $\mathbb{R}^2$, then there exists a constant $C_4 > 0$ (possibly infinite) such that
\begin{equation}\label{maz1}
\|w^2\|_{\Psi, \Omega, \mu} \le C_4\|w\|^2_{W^1_2(\Omega)}
\end{equation} for all $w\in W^1_2(\Omega)\cap C(\overline{\Omega})$. The best constant $C_4$ is equivalent to
\begin{equation}\label{B}
B = \sup \left\{\frac{\mu(E)\Phi^{-1}\left(\frac{1}{\mu(E)}\right)}{\textrm{cap}(E, W^1_2(\Omega))}\;: \;E\subseteq\Omega,\;\textrm{cap}(E, W^1_2(\Omega)) > 0\right\} < \infty.
\end{equation} That is, $B \le C_4 \le 2 BC_3$, where $C_3$ is the constant in \eqref{capineq}.}
\end{theorem}
Let $\varphi$ be a nonnegative increasing function on $[0, +\infty)$ such that $t\varphi(t^{-1})$ decreases and tends to zero as $t \longrightarrow\infty$. Further, suppose
\begin{equation}\label{maz2}
\int_u^{+\infty}t\sigma(t) dt \le cu\sigma(u),
\end{equation} for all $u > 0$, where
$$
\sigma(v) = v\varphi\left(\frac{1}{v}\right)
$$ and
$c$ is a positive constant (see (11.7.2) in \cite{Maz}).
\begin{theorem}\label{measthm2}{\rm (cf. \cite[Theorem 11.8]{Maz})}
{\rm Let $\varphi$ be the inverse function of $t \longrightarrow t\Phi^{-1}(t^{-1})$ subject to condition \eqref{maz2}. Then the best constant in \eqref{maz1} is equivalent to
\begin{equation}\label{maz3}
B_1 = \textrm{sup}\left\{|\log r| \mu(B(x, r))\Phi^{-1}\left(\frac{1}{\mu(B(x, r))}\right) : x\in\Omega,\, 0 < r < \frac{1}{2}\right\}.
\end{equation}where $B(x, r)$ is a ball of radius $r$ centred at $x$.}
\end{theorem}
\begin{proof}
According to Theorem \ref{measthm1}, it is enough to prove the equivalence $ B \sim B_1$.
For $ 0 < r \le 1$, $\textrm{cap}(B(x, r), W^1_2(\Omega))\sim \frac{1}{|\log r|}$ (see (10.4.15) in \cite{Maz}). So it follows from \eqref{B} that
\begin{eqnarray*}
|\log r|\mu(B(x, r))\Phi^{-1}\left(\frac{1}{\mu(B(x, r))}\right) &\le& |\log r| B \textrm{cap}\left(B(x, r), W^1_2(\Omega)\right)\\&\le& k_0 B
\end{eqnarray*} where $k_0 > 0$ is a constant. Hence
$$
k_0^{-1}B_1 \le B .
$$
Also \eqref{maz3} implies
$$
\mu(B(x, r))\Phi^{-1}\left(\frac{1}{\mu(B(x, r))}\right)\le B_1\frac{1}{|\ln r|} \le B_1k_1\;\textrm{cap}(B(x, r), W^1_2(\Omega)),
$$ where $k_1 >0$ is a constant.
Since $\varphi$ is the inverse function of $t \longrightarrow t\Phi^{-1}(t^{-1})$, then
$$
\mu(B(x, r)) \le \varphi\left(B_1k_1\;\textrm{cap}(B(x, r), W^1_2(\Omega))\right).
$$
Thus for any Borel set $E$ with the finite capacity $\textrm{cap}(E, W^1_2(\Omega))$, Corollary 11.7/2 in \cite{Maz} implies
$$
\mu(E) \le k_2\varphi\left(B_1k_1\;\textrm{cap}(E, W^1_2(\Omega))\right),
$$where $k_2 \geq 1$ is a constant that depends on $\varphi$. So
$$
\mu(E)\Phi^{-1}\left(\frac{k_2}{\mu(E)}\right) \le k_2 B_1 k_1\;\textrm{cap}(E, W^1_2(\Omega)).
$$Since $\Phi$ is a nondecreasing function, dropping $k_2$ in the left-hand side of the above inequality gives
$$
\frac{\mu(E)\Phi^{-1}\left(\frac{1}{\mu(E)}\right)}{\textrm{cap}(E, W^1_2(\Omega))} \le k_2B_1k_1\,.
$$ Hence
$$
B \le k_2B_1k_1
$$ (see \eqref{B}).
\end{proof}
\begin{lemma}\label{small}
{\rm Let $b > 0$ and $\Psi : [b , +\infty) \longrightarrow \mathbb{R}_+$ be a non-increasing function such that there exist $l > 0$ and $q \in (0, 1)$ for which
\begin{equation}\label{sm1}
\Psi(t + l) \le q\, \Psi(t),\,\,\forall t \in [b , +\infty).
\end{equation}
Then $\forall m \ge 0$ there exists a constant $C_5 = C_5(b, q, l, m) < +\infty$ such that
\begin{equation}\label{sm2}
\int_u^{+\infty}t^m \Psi(t)\,dt \le C_5 u^m \Psi(u)\,, \;\;\;\forall u \in [b, +\infty).
\end{equation}
}
\end{lemma}
\begin{proof}
\begin{eqnarray*}
\int_u^{+\infty}t^m \Psi(t)\,dt &=& \sum _{k = 0}^{\infty}\int_{u + kl}^{u + (k+1)l}t^m \Psi(t)\,dt \\ &=& \sum _{k = 0}^{\infty}\int_{u }^{u + l}(x + kl)^m \Psi(x+kl)\,dx \\ &\le& \sum_{k =0}^{\infty}\left(1 + \frac{kl}{u}\right)^m q^k \int_u^{u + l}x^m \Psi(x)\,dx \\ &\le& \underbrace{\left(\sum_{k =0}^{\infty}\left(1 + \frac{kl}{b}\right)^m q^k\right)}_{< \,\infty} (u + l)^m \Psi(u)l \\ &\le& C_0(b, q, l, m)l\left( 1 + \frac{l}{u}\right)^mu^m\Psi(u)\\ &\le&C_0(b, q, l, m)l\left( 1 + \frac{l}{b}\right)^mu^m\Psi(u) \\ &=:& C_5(b, q, l, m)u^m\Psi(u)\,,\,\,\,\forall u \in [b, +\infty).
\end{eqnarray*}
\end{proof}
\begin{lemma}\label{meascor}{\rm  (cf.\;\cite[  Corollary 11.8/2]{Maz})}
{\rm Consider the complementary N-functions $\mathcal{B}(t)= (1 + t)\ln(1 + t) - t$ and  $\mathcal{A}(t)= e^t - 1 - t$. Let $Q$ be the unit square with edges chosen in any direction. Then the inequality
$$\|w^2\|_{\mathcal{A}, Q, \mu} \le C_4\|w\|^2_{W^1_2(Q)}$$ holds if for some $\alpha > 0$
\begin{equation}\label{ball}
\mu(B(x, r)) \le r^{\alpha}\;,\;\;\forall x\in\overline{Q} \;\;\; \textrm{and}\;\;\;  \forall r \in(0, 1]\,,
\end{equation}
 where $C_4$ (see \eqref{maz1}) is a constant depending only on $\alpha$ .
}
\end{lemma}
\begin{proof}
 First note that $\mathcal{B}'(t)$ and $\mathcal{A}'(t)$ are direct inverses of each other. Let $\varrho(t) = t\mathcal{B}^{-1}\left(\frac{1}{t}\right)$ and $\frac{1}{t} = \mathcal{B}(s)$. Then $\varrho(t) = \frac{s}{\mathcal{B}(s)}$. Since $\frac{d}{ds}(\frac{\mathcal{B}(s)}{s}) = -\frac{1}{s^2}\ln(1 + s) + \frac{1}{s}> 0$, for $s > 0$, then $\frac{s}{\mathcal{B}(s)}$ is a decreasing function of $s$. It is also clear that $\frac{s}{\mathcal{B}(s)} \longrightarrow 0$ as $s \longrightarrow\infty$. Hence $\varrho(t)$ is an increasing function of $t$ and $\varrho(t) \longrightarrow 0$ as $t \longrightarrow 0+$. So
 \begin{equation}\label{*}
 \varrho(t) = t\mathcal{B}^{-1}\left(\frac{1}{t}\right) = \sqrt{2t}\left(1 + o(1)\right)\;\;\textrm{as}\;\;t\longrightarrow \infty
 \end{equation} and
 \begin{equation}\label{**}
   \varrho(t) = t\mathcal{B}^{-1}\left(\frac{1}{t}\right) = \frac{1}{\ln\frac{1}{t}}\left(1 + o(1)\right)\;\;\textrm{as}\;\;t\longrightarrow 0
 \end{equation}
 (see Appendix \ref{proof}). Let $\varphi(\tau) = \varrho^{-1}(\tau)$. Then
 for small values of $\tau$
\begin{equation}\label{***}
\varphi(\tau)= \tau e^{-\frac{1}{\tau}}e^{O(1)}
\end{equation}(see \eqref{large} in Appendix \ref{proof})
and for large values of $\tau$
$$
\varphi(\tau) = \frac{\tau^2}{2}(1 + o(1)).
$$
 By \eqref{maz2}, for large values of $u$ (small values of $t$) we have
 \begin{eqnarray*}
 \int_u^{+\infty}t \sigma(t)\,dt &=&\int_u^{+\infty}t^2\varphi\left(\frac{1}{t}\right)\,dt\\ &=& \int_u^{+\infty} te^{-t}e^{O(1)}\,dt.
 \end{eqnarray*} By Lemma \ref{small},
 $$
 \int_u^{+\infty} te^{-t}e^{O(1)}\,dt  \le C ue^{-u} e^{O(1)}\;\;\;\textrm{as}\;\;\;u\longrightarrow +\infty.
 $$
 Hence
 \begin{eqnarray*}
 \int_u^{+\infty}t \sigma(t)\,dt &\le& C u e^{-u}e^{O(1)}\\&=& C u^2\varphi\left(\frac{1}{u}\right)e^{O(1)}\\&=&
 C u\sigma(u)e^{O(1)}\;\;\;\textrm{as}\;\;\;u\longrightarrow +\infty.
 \end{eqnarray*}
 For small values of $v$ (large values of $ t$), the function $v\varphi(\frac{1}{v}) = \frac{1}{2v}\left(1 + o(1)\right)$ is not integrable at zero. However,
 $$
 t\sigma(t)= t^2\varphi\left(\frac{1}{t}\right) = \frac{1}{2}\left(1 + o(1)\right)\;\;\;\textrm{as}\;\;\;t \longrightarrow 0+\;.
 $$
 Hence
 $$
 \int_u^{+\infty} t \sigma(t)\,dt \longrightarrow\;\textrm{constant} \;\;\;\textrm{as}\;\;\; u\longrightarrow 0+
 $$
 and
 $$
 u\sigma (u) \;\longrightarrow \frac{1}{2} \;\;\;\textrm{as}\;\;\;u\longrightarrow 0+\;.
 $$
 Thus $\varphi(\tau)$ satisfies  condition \eqref{maz2} for all values of $u$.\\ It now remains to establish the finiteness of  the constant $B_1$ in theorem \ref{measthm2}.  For $0 < r \le \frac{1}{2}$
 \begin{eqnarray*}
  B_1 = \sup\left\{|\ln r|\mu(B(x, r))\mathcal{B}^{-1}\left(\frac{1}{\mu(B(x, r))}\right)\right\}&=& \sup\frac{|\ln r|}{|\ln \mu(B(x, r))|}\left( 1 + o(1)\right)\\&\le& \textrm{const}\sup\frac{|\ln r|}{|\ln r^{\alpha}|}\\ &=& \textrm{const}\sup\frac{|\ln r|}{|\alpha\ln r|}\\&=& \frac{\textrm{const}}{\alpha}.
 \end{eqnarray*} Thus one can take $C_4 \sim \frac{1}{\alpha}$.
 \end{proof}
  We will also use the following notation:
\begin{equation}\label{ave}
w_E := \frac{1}{|E|}\int_E w(x)\,dx\,,
\end{equation}where $E\subset\mathbb{R}^2$ is a set with finite Lebesgue measure $|E|$.
\begin{lemma}\label{measlemma3*}
{\rm Let $\mu$ be a positive  Radon measure satisfying \eqref{ball}. Then for any $V\in L_{\mathcal{B}}(Q, \mu), \;V\geq 0$ there is a constant $C_6 > 0$  such that
\begin{equation}\label{maz6}
\int_{\overline{Q}} V|w(x)|^2d\mu(x) \le C_6\|V\|_{\mathcal{B}, Q, \mu}\int_{Q}|\nabla w |^2dx,
\end{equation} for all $w\in W^1_2(Q)\cap C(\overline{Q})$ with $w_Q = 0$.
}
\end{lemma}
\begin{proof}
The proof of this Lemma follows from the H\"older inequality for Orlicz spaces (see \eqref{Holder}), Lemma \ref{meascor} and the Poincar\'e inequality (see \cite[1.1.11]{Maz}, see also Appendix \ref{poincare}).
\end{proof}
 \begin{definition}
 {\rm Let $(X_1, \Sigma_1)$ and $(X_2, \Sigma_2)$ be a pair of measurable spaces. Given a measure $\mu_1$ on $(X_1, \Sigma_1)$ and a measurable function $\xi$ on $(X_1, \Sigma_1)$ into $(X_2, \Sigma_2)$, we define the pushforward measure or image $\mu_2 := \mu_1\circ \xi^{-1}$ of $\mu_1$ under $\xi$ by $\mu_2(E) = \mu_1(\xi^{-1}(E))$ for $E\in \Sigma_2$. $\mu_2$ is a measure on $(X_2, \Sigma_2)$.
 }
 \end{definition}
 \begin{lemma}\label{pushlemma}{\rm \cite[Lemma 5.0.1]{strok}
 For every non-negative measurable function $\varphi$ on $(X_2, \Sigma_2)$
 \begin{equation}\label{push}
 \int_{X_2}\varphi \,d\mu_2 = \int_{X_1}\varphi\circ \xi\,d\mu_1\;,
 \end{equation} where $\mu_2 = \mu_1 \circ \xi^{-1}$.
 Moreover, $\varphi \in L^1(X_2, \Sigma_2, \mu_2)$ if and only if $\varphi \circ \xi\in L^1(X_1, \Sigma_1, \mu_1)$ and \eqref{push} holds for all $\varphi\in L^1(X_2, \Sigma_2, \mu_2)$.
 }
 \end{lemma}
Let $\xi(x) = \begin{pmatrix}
R_1 & 0\\0 & R_2
\end{pmatrix}x + x^0 , \;\;x, x^0\in\mathbb{R}^2,\;\; R_1, R_2 > 0$, be an affine transformation. Let $l_1$ and $l_2$ be any two perpendicular lines in $\mathbb{R}^2$ chosen in directions that satisfy Corollary \ref{cor-direct} and $Q$ the unit square with sides in directions also satisfying Corollary \ref{cor-direct}. Let $I_1$ and  $I_2$  be respective segments of $l_1$ and $l_2$ such that $I_1\times I_2$ is a rectangle and $\xi(Q) = I_1\times I_2$. Further, let  $\tilde{\mu} := \mu\circ \xi$ and $\tilde{V} := V\circ \xi$. Then for any $c > 0$, one gets  using \eqref{OrlAverage} and \eqref{push}
\begin{eqnarray}\label{scale}
\|V\|^{(\textrm{av})}_{\mathcal{B}, I_1\times I_2, \mu} &=&\textrm{sup}\left\{\left|\int_{\overline{I_1\times I_2}}V f\,d\mu\right| \;:\;\int_{\overline{I_1\times I_2}}\mathcal{A}(|f|)\,d\mu \le \mu(I_1\times I_2)\right\}\nonumber\\ &=& \textrm{sup}\left\{\frac{1}{c}\left|\int_{\overline{Q}}\tilde{V} g\,d(c\tilde{\mu})\right| \;:\;\int_{\overline{Q}}\mathcal{A}(|g|)\,d(c\tilde{\mu}) \le c\tilde{\mu}(Q)\right\}\nonumber\\
&=& \frac{1}{c}\|\tilde{V}\|^{(\textrm{av})}_{\mathcal{B}, Q, c\tilde{\mu}}\;,
\end{eqnarray}
where $g := f\circ \xi$.  Hence by Corollary \ref{avequiv} we have
\begin{eqnarray}\label{maz7}
\|\tilde{V}\|_{\mathcal{B}, Q, c\tilde{\mu}} &\le& \frac{1}{\min\{1, c\tilde{\mu}(Q)\}} \|\tilde{V}\|^{(\textrm{av})}_{\mathcal{B}, Q, c\tilde{\mu}}\nonumber\\ &\le& \frac{c}{\min\{1, c\tilde{\mu}(Q)\}}\|V\|^{\textrm{(\textrm{av})}}_{\mathcal{B}, I_1\times I_2, \mu}\;.
\end{eqnarray}
\begin{lemma}\label{measlemma3}
{\rm Let $I_1$ and $I_2$ be defined as above. Suppose $\mu$ satisfies \eqref{Ahlfors} and the rectangle $I_1\times I_2$ is centred in the support of the measure $\mu$. Then for any $V\in L_{\mathcal{B}}(I_1\times I_2, \mu), \; V \geq 0$
\begin{eqnarray}\label{maz8}
&&\int_{\overline{I_1\times I_2}}V(y)|w(y)|^2d\mu(y)\nonumber\\&& \le C_6\frac{c_1}{c_0}2^{\alpha}\max\left\{\frac{R_1}{R_2}, \frac{R_2}{R_1}\right\}^{\alpha + 1}\|V\|^{(\textrm{av})}_{ \mathcal{B},I_1\times I_2, \mu}\int_{I_1\times I_2}|\nabla w(y)|^2 dy
\end{eqnarray}for all $w\in W^1_2(I_1\times I_2)\cap C(\overline{I_1\times I_2})$ with $w_{I_1\times I_2} = 0$ (see \eqref{ave}), where $C_6$ is the constant in Lemma \ref{measlemma3*}.
}
\end{lemma}
\begin{proof}
Take $\forall x \in \overline{Q}$ such that $\xi(x) \in$ supp$\,\mu$. For any disk $B(x, r)$ the image $\xi(B(x,r))$ contains the disk of radius $\min\{R_1, R_2\}r$ and is contained in the disk of radius $\max\{R_1, R_2\}r$ both centred at $\xi(x)$. Hence, \eqref{Ahlfors} implies\\
$$
c_0\left(\min\{R_1, R_2\}r\right)^{\alpha} \le \tilde{\mu}(B(x, r)) = \mu\left(\xi(B(x, r))\right) \le c_1\left(\max\{R_1, R_2\}r\right)^{\alpha}.
$$
Let
$$
c := \frac{1}{c_1\max \{R_1, R_2\}^{\alpha}}.
$$
Then Lemma \ref{measlemma3*} applies to the measure $c\tilde{\mu}$. Using \eqref{maz7} and the inequality
$$
\int_Q|\nabla (w \circ\xi)(x)|^2 dx \le \textrm{max}\left\{\frac{R_1}{R_2}, \frac{R_2}{R_1}\right\}\int_{I_1\times I_2}|\nabla w(y)|^2 dy
$$
(see the proof of \cite[Lemma 2]{Sol}), we get

\begin{eqnarray}\label{disk}
&&\int_{\overline{I_1\times I_2}}V(y)|w(y)|^2d\mu(y) = \frac{1}{c}\int_{\overline{Q}}V(\xi(x))|w(\xi(x))|^2d(c\mu(\xi(x)))\nonumber\\&&= \frac{1}{c}\int_{\overline{Q}}\widetilde{V}(x)|(w\circ\xi)(x)|^2d(c\tilde{\mu}(x))\nonumber\\&&\le \frac{1}{c} C_6\|\widetilde{V}\|_{\mathcal{B}, Q, c\tilde{\mu}}\int_Q|\nabla (w \circ\xi)(x)|^2 dx\nonumber \\&&\le \frac{1}{c} C_6\;\textrm{max}\left\{\frac{R_1}{R_2}, \frac{R_2}{R_1}\right\} \frac{c}{\min\{1, c\tilde{\mu}(Q)\}}\|V\|^{(\textrm{av})}_{\mathcal{B}, I_1\times I_2, \mu}\int_{I_1\times I_2}|\nabla w(y)|^2 dy\nonumber\\&&
\end{eqnarray} But
\begin{eqnarray}\label{disk1}
\frac{1}{\min\{1, c\tilde{\mu}(Q)\}} &=& \max \left\{1, \frac{1}{c\tilde{\mu}(Q)}\right\}\nonumber\\ &=& \max \left\{1, \frac{c_1 \max\{R_1, R_2\}^{\alpha}}{\mu(I_1\times I_2)}\right\}\nonumber\\
&\le& \max\left\{ 1, \frac{c_1 \{R_1, R_2\}^{\alpha}}{c_0\left(\frac{\min\{R_1, R_2\}}{2}\right)^{\alpha}}\right\}\nonumber\\&=&
\frac{c_1}{c_0}2^{\alpha}\max\left\{\frac{R_1}{R_2}, \frac{R_2}{R_1}\right\}^{\alpha}.
\end{eqnarray}In the last inequality we use the fact $I_1\times I_2$ contains a disk of radius $\frac{\min\{R_1, R_2\}}{2}$ centred in the support of $\mu$. Hence \eqref{disk} and \eqref{disk1} yield \eqref{maz8}.
\end{proof}
Below we construct an example for which Lemma \ref{measlemma3} fails if $I_1\times I_2$ is not centred in the support of $\mu$.
\begin{example}\label{counter}
{\rm Let $I_1\times I_2 = Q$ (the unit square). Let $x_0$ be the lower left corner of $Q$ and suppose that supp$\,\mu \cap Q \subseteq B(x_0, r)$ for small $r$. Let
\begin{eqnarray*}
w(x) := \left\{\begin{array}{l}
  1, \ \;\;\;\;\;\;\;\;\;\;\;\;\;\;\; |x - x_0| \le r,   \\ \\
   \frac{\ln\left(|x - x_0|\sqrt{2}\right)}{\ln(r\sqrt{2})}\,, \  r < |x - x_0| < \frac{1}{\sqrt{2}}, \\ \\
   0,\ \;\;\;\;\;\;\;\;\;\;\;\;\;\;\; |x - x_0| \ge  \frac{1}{\sqrt{2}}\,.
\end{array}\right.
\end{eqnarray*}
Then introducing polar coordinates centred at $x_0$, $x - x_0 = (\rho\cos \theta, \rho\sin\theta)$, one has
\begin{eqnarray}\label{c1}
\int_Q |\nabla w(x)|^2\,dx &=& \int_0^{\frac{\pi}{2}}\int_r^{\frac{1}{\sqrt{2}}}\left|\frac{d}{d\rho}\left(\frac{\ln\left(\rho\sqrt{2}\right)}{\ln(r\sqrt{2})}\right)\right|^2\rho d\rho d\theta\nonumber \\ &=& \frac{\pi}{2}\left(\frac{1}{\ln(r\sqrt{2})}\right)^2\int_r^{\frac{1}{\sqrt{2}}}\frac{1}{\rho}\,d\rho\nonumber\\ &=& \frac{\pi}{2}\left(\frac{1}{\ln(r\sqrt{2})}\right)^2(-\ln(r\sqrt{2}))\nonumber \\&=& \frac{\pi}{2}\left(\frac{1}{-\ln(r\sqrt{2})}\right).
\end{eqnarray}
Let $V \equiv 1$. Then the left-hand side of \eqref{maz8} becomes
\begin{equation}\label{c2}
\int_{\overline{Q}} V(x)|w(x)|^2\,d\mu(x) = \int_{\overline{Q}}d\mu(x) = \mu(Q).
\end{equation}
By \eqref{c1},\eqref{c2} and \eqref{c5}, \eqref{maz8} becomes
$$
\mu(Q) \le C_6 \frac{c_1}{c_0}2^{\alpha}\Phi^{-1}(1)\mu(Q) \frac{\pi}{2}\left(\frac{1}{-\ln(r\sqrt{2})}\right)
$$ which does not hold for small values of $r$.
}
\end{example}
However, one can prove an analogue of Lemma \ref{measlemma3} with a different norm in the right-hand side of \eqref{maz8} if $I_1\times I_2$ is not centred in the support of $\mu$ (see Appendix \ref{analogue}).\\\\
Let $G\subset\mathbb{R}^2$ be a bounded set with Lipschitz boundary such that \\$0 < \mu(G) < \infty$. Let $G^*$ be a square with the same centre and  edges of length $3$ times the length of the smallest square containing $G$ chosen in the directions $\theta_0$ and $\theta_0 + \frac{\pi}{2}$ from Corollary \ref{cor-direct}. Let
$$
\kappa_0 (G) := \frac{\mu(G^*)}{\mu(G)}\,.
$$

Further, let
  \begin{eqnarray*}
V_{*}(x) := \left\{\begin{array}{l}
  V(x), \ \;\;\mbox{ if } x\in G,   \\ \\
   0, \ \;\;\mbox{ if } x\notin G.
\end{array}\right.
\end{eqnarray*}
Then
\begin{equation}\label{exteqn}
\|V_{*}\|^{(av)}_{\mathcal{B}, G^*, \mu} = \|V\|^{(av), \kappa_0}_{\mathcal{B}, G, \mu} \le \kappa_0\|V\|^{(av)}_{\mathcal{B}, G, \mu}\,,
\end{equation}(see Lemma \ref{lemma8}).
\begin{theorem}\label{ext}{\rm \cite[Ch.VI, Theorem 5]{Stein}}
{\rm There exists a bounded linear operator $$T: W^1_2(G) \longrightarrow W^1_2(\mathbb{R}^2)$$  which satisfies
$$
Tw|_G = w, \;\;\;\forall w \in W^1_2(G)\,.
$$
}
\end{theorem}
This together with the Poincar\'e inequality (see Appendix \ref{poincare}) imply
\begin{eqnarray}\label{ext1}
\int_{G^*}|\nabla (Tw)(x)|^2\,dx &\le& \int_{\mathbb{R}^2}|\nabla (Tw)(x)|^2\,dx \nonumber \\&\le& \|T\|^2\|w\|^2_{W^1_2(G)}\nonumber  \\ &\le& \|T\|^2 (1 + C^2)\int_G|\nabla w(x)|^2\,dx
\end{eqnarray}
for all $w\in W^1_2(G)$ with $w_G = 0$.
\begin{lemma}\label{measlemma4}
{\rm Let $\mu$ be a positive  Radon measure on $\mathbb{R}^2$ that is Ahlfors regular and $G$ be defined as above. Choose and fix a direction satisfying Corollary \ref{cor-direct}. Further, for all $x\in\overline{G}$ and for all $r > 0$, let $Q_x(r)$ be a square with edges of length $r$ in the chosen direction centred at $x\in \textrm{supp}\,\mu \cap \overline{G}$. Then for any $V\in L_{\mathcal{B}}(G, \mu), \; V\geq 0$ and any $n\in \mathbb{N}$ there exists a finite cover of $\textrm{supp}\,\mu\cap \overline{G}$  by squares $Q_{x_k}(r_{x_k}), r_{x_k} > 0, k = 1, 2, ..., n_0$,  such that $n_0\le n$  and
\begin{equation}\label{maz9}
\int_{\overline{G}} V(x)|w(x)|^2d\mu(x) \le C_7n^{-1}\|V\|^{(av)}_{\mathcal{B}, G, \mu}\int_G|\nabla w(x)|^2\,dx
\end{equation}for all $w\in W^1_2(G)\cap C(\overline{G})$ with $(Tw)_{Q_{x_k}(r_{x_k})} = 0, k = 1,..., n_0$ and $w_G = 0$, where the constant $C_7$ is independent of $V$ and $n$.
}
\end{lemma}
\begin{proof}
Let $N\in\mathbb{N}$ be the bound in the Besicovitch covering Lemma (see Appendix \ref{besico}). If $n \le \kappa_0N$, one can take $n_0 = 0$ and then using \eqref{maz8}, \eqref{exteqn} and \eqref{ext1}, we have
\begin{eqnarray*}
\int_{\overline{G}} V(x)|w(x)|^2\,d\mu(x) &=& \int_{\overline{G^*}} V_*(x)|w(x)|^2\,d\mu(x)\\ &\le& C_6 \frac{c_1}{c_0}2^{\alpha}\kappa_0N n^{-1}\|V_*\|^{(\textrm{av})}_{\mathcal{B}, G^*, \mu}\int_{G^*}|\nabla (Tw)(x)|^2\,dx\\  &\le& C_6 \frac{c_1}{c_0}2^{\alpha} \kappa_0N n^{-1}\kappa_0\|V\|^{(\textrm{av})}_{\mathcal{B}, G, \mu}\|T\|^2(1 + C^2)\int_{G}|\nabla (w)(x)|^2\,dx \\&=& C_7n^{-1}\|V\|^{(\textrm{av})}_{\mathcal{B}, G, \mu}\int_{G}|\nabla (w)(x)|^2\,dx\,,
\end{eqnarray*}
where $C_7 := \frac{ C_6 c_1 2^{\alpha}\kappa^2_0N\|T\|^2(1 + C^2)}{c_0}$.\\
Now assume that $n > \kappa_0N$, then for any $x\in \textrm{supp}\,\mu\cap\overline{G}$, Lemma \ref{measlemma2} implies there is a closed square $Q_x(r_x)$ centred at $x$ such that
\begin{equation}\label{use}
 \|V_{*}\|^{(\textrm{av})}_{\mathcal{B}, Q_x(r_x), \mu} =  \kappa_0N n^{-1}\|V\|^{(av)}_{\mathcal{B}, G, \mu}.
\end{equation}
Since $\kappa_0N n^{-1} < 1$, it is not difficult  to see that $Q_x(r_x) \subseteq G^*$.
Consider the covering $\Xi = \{Q_x(r_x)\}$ of $\textrm{supp}\,\mu\cap\overline{G}$. Then according to the Besicovitch covering Lemma, $\Xi$ has a countable or finite subcover $\Xi'$ that can be split into $N$ subsets $\Xi'_j,\,j = 1, ..., N$ in such  a way that the closed squares in each subset are pairwise disjoint.
 Applying Lemma \ref{lemma7} and  \eqref{exteqn}  one has
\begin{eqnarray*}
\kappa_0Nn^{-1}\|V\|^{(av)}_{\mathcal{B}, G, \mu} \textrm{card}\,\Xi'_j &=&  \underset{Q_{x}(r_{x})\in\Xi'_j}\sum \|V_{*}\|^{(\textrm{av})}_{\mathcal{B}, Q_{x}(r_{x}), \mu}\le \|V_{*}\|^{(\textrm{av})}_{\mathcal{B}, G^{*}, \mu}\\ &\le& \kappa_0\|V\|^{(\textrm{av})}_{\mathcal{B}, G, \mu}\,.
\end{eqnarray*}
Hence $\textrm{card}\,\Xi'_j \le nN^{-1}$ and
$$
n_0 := \textrm{card}\,\Xi' = \sum_{j= 1}^N \textrm{card}\,\Xi'_j \le  n.
$$
Again, using \eqref{maz8}, \eqref{ext1} and \eqref{use}, we have
\begin{eqnarray*}
\int_{\overline{G}} V(x)|w(x)|^2d\mu(x) &=& \int_{\textrm{supp}\,\mu \cap \overline{G}} V(x)|w(x)|^2d\mu(x) \\ &\le& \sum_{k = 1}^{n_0}\int_{Q_{x_k}(r_{x_k})} V_{*}(x)|(Tw)(x)|^2\,d\mu(x)\\&\le& C_6 \frac{c_1}{c_0}2^{\alpha}\sum_{k = 1}^{n_0} \|V_{*}\|^{(\textrm{av})}_{\mathcal{B}, Q_{x_k}(r_{x_k}), \mu}\int_{Q_{x_k}(r_{x_k})}|\nabla (Tw)(x)|^2dx\\ &=& C_6 \frac{c_1}{c_0}2^{\alpha} n^{-1}\kappa_0N\|V\|^{(\textrm{av})}_{\mathcal{B}, G, \mu}\sum_{k = 1}^{n_0}\int_{Q_{x_k}(r_{x_k})}|\nabla (Tw)(x)|^2\,dx\\ &=&
C_6 \frac{c_1}{c_0}2^{\alpha}\kappa_0 n^{-1}N\|V\|^{(\textrm{av})}_{\mathcal{B}, G, \mu}\sum_{j = 1}^{N}\underset{Q_{x_k}(r_{x_k})\in\Xi'_j}\sum\int_{Q_{x_k}(r_{x_k})}|\nabla (Tw)(x)|^2\,dx\\ &\le& C_6 \frac{c_1}{c_0}2^{\alpha}\kappa_0 n^{-1}N\|V\|^{(\textrm{av})}_{\mathcal{B}, G, \mu}\sum_{j = 1}^{N}\int_{G^*}|\nabla (Tw)(x)|^2\,dx\\&\le& C_6 \frac{c_1}{c_0}2^{\alpha}\kappa_0 n^{-1}N^2\|T\|^2(1 + C^2)\|V\|^{(\textrm{av})}_{\mathcal{B}, G, \mu}\int_{G}|\nabla w(x)|^2\,dx \\&=& C_7n^{-1}\|V\|^{(\textrm{av})}_{\mathcal{B}, G, \mu}\int_{G}|\nabla w(x)|^2\,dx\;,
\end{eqnarray*} where $C_7 := \frac{ C_6 c_1 2^{\alpha}\kappa_0N^2\|T\|^2(1 + C^2)}{c_0}$.
\end{proof}

Let $G$ in Lemma \ref{measlemma4} be the unit square $Q$ centred in the support of $\mu$.   Then
$$
\mu(Q) \geq c_0\left(\frac{1}{2}\right)^{\alpha} \;\;\;\textrm{and} \;\;\; \mu(Q^{*}) \le c_1\left(\frac{3}{\sqrt{2}}\right)^{\alpha}.
$$ Hence
\begin{equation}\label{disk2}
\mu(Q^{*}) \le \frac{c_1}{c_0}\left(3\sqrt{2}\right)^{\alpha}\mu(Q) = \kappa_0\mu(Q)\,,
\end{equation} where
$$
\kappa_0 := \frac{c_1}{c_0}\left(3\sqrt{2}\right)^{\alpha}.
$$

Let
\begin{eqnarray*}
& \mathcal{E}_{2V\mu, Q}[w] : = \int_{Q} |\nabla w(x)|^2 dx - 2\int_{\overline{Q}} V(x) |w(x)|^2 d\mu(x) , & \\
& \mbox{Dom}\, (\mathcal{E}_{2V\mu, Q}) =
W^1_2\left(Q\right)\cap L^2\left(Q, Vd\mu\right) . &
\end{eqnarray*}

\begin{lemma}\label{measlemma5}{\rm [ cf. Lemma \ref{lemma5}]
\begin{equation}\label{meain1}
N_- (\mathcal{E}_{2V\mu, Q}) \le C_8
\|V\|^{(\textrm{av})}_{\mathcal{B}, Q, \mu} + 2 , \ \ \ \forall V \ge 0 ,
\end{equation}
where $C_8 := 2C_7$ and $C_7$ is the constant in Lemma \ref{measlemma4}.
}
\end{lemma}
\begin{proof}
Let $n = \left[C_8\|V\|^{(\textrm{av})}_{\mathcal{B}, Q, \mu}\right]  + 1$
in Lemma \ref{measlemma4}, where $[a]$ denotes the largest integer not greater than $a$. Take any
linear subspace $\mathcal{L} \subset \mbox{Dom}\, (\mathcal{E}_{2V\mu, Q})$
such that
$$
\dim \mathcal{L} > \left[C_8 \|V\|^{(\textrm{av})}_{\mathcal{B}, Q, \mu}\right] + 2 .
$$
Since $n_0 \le n$, there exists $w \in \mathcal{L}\setminus\{0\}$ such that
$w_{Q_{x_k}(r_{x_k})} = 0$, $k = 1, \dots, n_0$ and $w_{Q} = 0$. Then
\begin{eqnarray*}
\mathcal{E}_{2V\mu, Q}[w]  &=& \int_{Q} |\nabla w( x)|^2dx - 2\int_{\overline{Q}} V( x) |w( x)|^2 d\mu(x) \\
&\ge& \int_{Q} |\nabla w(x)|^2 dx -
\frac{C_8 \|V\|^{(\textrm{av})}_{\mathcal{B}, Q, \mu}}{\left[C_8
\|V\|^{(\textrm{av})}_{\mathcal{B}, Q, \mu}\right] + 1}\,
\int_{Q} |\nabla w(x)|^2 dx \\
&\ge& \int_{Q} |\nabla w(x)|^2 dx - \int_{Q} |\nabla w(x)|^2 dx = 0 .
 \end{eqnarray*}
Hence
$$
N_- (\mathcal{E}_{2V\mu, Q}) \le \left[C_8
\|V\|^{(\textrm{av})}_{\mathcal{B}, Q, \mu}\right] + 2 \le
C_8 \|V\|^{(\textrm{av})}_{\mathcal{B}, Q, \mu} + 2.
$$
\end{proof}
\begin{lemma}\label{measlemma5*}
{\rm There is a constant $C_9 > 0$ such that
\begin{equation}\label{meain2}
N_- (\mathcal{E}_{2V\mu, Q}) \le C_9 \|V\|^{(\textrm{av})}_{\mathcal{B}, Q, \mu} \;\;\;\;\forall V\geq 0.
\end{equation}
}
\end{lemma}
\begin{proof}
By Lemma \ref{measlemma3} there is a constant $C_{10} $ such that
$$
2\int_{\overline{Q}}V(x)|w(x)|^2d\mu(x) \le C_{10}\|V\|^{\textrm{(av)}}_{\mathcal{B}, Q, \mu}\int_{Q}|\nabla w(x)|^2dx
$$for all $w\in W^1_2(Q)\cap C(\overline{Q}) \mbox{ and } w_{Q} = 0$, where $C_{10} := C_6\frac{c_1}{c_0}2^{\alpha + 1}$.\\
If $\|V\|^{(\textrm{av})}_{\mathcal{B}, Q, \mu} \le \frac{1}{C_10}$, then $$N_- (\mathcal{E}_{2V\mu, Q}) = 0.$$ If $\|V\|^{(\textrm{av})}_{\mathcal{B}, Q, \mu} > \frac{1}{C_10}$, then Lemma \ref{measlemma5}  implies
\begin{eqnarray}\label{meaeqn7}
N_- (\mathcal{E}_{2V\mu, Q}) &\le& 2\left(\frac{1}{2} C_8
\|V\|^{(\textrm{av})}_{\mathcal{B}, Q, \mu} + 1\right)\nonumber\\
&\le&  C_9\|V\|^{\textrm{(av)}}_{\mathcal{B}, Q_, \mu},
\end{eqnarray} where $C_9 := C_8 + 2C_10$.
\end{proof}
\begin{remark}\label{nopoincare}
{\rm Let
\begin{eqnarray*}
& \mathcal{E}^1_{2V\mu, Q}[w] : = \|w\|^2_{W^1_2(Q)} - 2\int_{\overline{Q}} V(x) |w(x)|^2 d\mu(x) , & \\
& \mbox{Dom}\, (\mathcal{E}^1_{2V\mu, Q}) =
W^1_2\left(Q\right)\cap L^2\left(Q, Vd\mu\right) . &
\end{eqnarray*}
Then there is no need for the Poincar\'e inequality and the condition that $w_Q = 0$, and one can obtain as above
\begin{equation}\label{noeqn}
N_- (\mathcal{E}^1_{2V, Q, \mu}) \le C_9 \|V\|^{(\textrm{av})}_{\mathcal{B}, Q, \mu} \;\;\;\;\forall V\geq 0.
\end{equation}
Let $\Delta_1$ be a square of side of length $2 \left(2\frac{c_1}{c_0}\right)^{\frac{1}{\alpha}}$ centred at $0$. Then it is easy to see that a similar estimate holds with perhaps a different constant $C'_9$ in place of $C_9$.
}
\end{remark}
Assume without loss of generality that $0 \in$ supp$\,\mu$. Let
$$
Q_n := \left\{x\in\mathbb{R}^2 \;:\;\left(2\frac{c_1}{c_0}\right)^{\frac{n -1}{\alpha}} \le |x| \le \left(2\frac{c_1}{c_0}\right)^{\frac{n }{\alpha}}\right\}, \; n\in\mathbb{Z}.
$$   Let
\begin{eqnarray*}
\mathcal{E}_{\mathcal{N}, 2V\mu, Q_n}[w] = \int_{Q_n}|\nabla w(x)|^2dx - 2 \int_{Q_n}V(x)|w(x)|^2\,d\mu(x) ,\\
\textrm{Dom}(\mathcal{E}_{\mathcal{N}, 2V\mu, Q_n}) = W^1_2(Q_n)\cap L^2(Q_n, Vd\mu),\, w_{Q_n} = 0.
\end{eqnarray*}
Then we have the following Lemma.
\begin{lemma}\label{measlemma6}
{\rm There exists a constant $C_{11} > 0$ such that
\begin{equation}\label{meaeqn6}
N_-\left(\mathcal{E}_{\mathcal{N}, 2V\mu, Q_n}\right) \le C_{11} \|V\|^{(\textrm{av})}_{\mathcal{B}, Q_n ,\mu}\;,\;\;\;\forall V\ge 0.
\end{equation}
}
\end{lemma}
\begin{proof}
We start with the case $n =1$. Since $w_{Q_1} = 0$, then it follows from the Poincar\'e inequality that there is a constant $C_{12} > 0$ such that
$$
\|w\|^2_{W^1_2(Q_1)} \le C_{12} \int_{Q_1}|\nabla w(x)|^2\, dx.
$$ This implies
$$
\mathcal{E}_{\mathcal{N}, 2V\mu, Q_1}[w] \ge \frac{1}{C_{12}}\|w\|^2_{W^1_2(Q_1)} - 2\int_{Q_1}V(x)|w(x)|^2\,d\mu(x).
$$
Let $\Delta_1$ be the square of side of length $2\left(2\frac{c_1}{c_0}\right)^{\frac{1}{\alpha}}$ with the same centre as $Q_1$ and
\begin{eqnarray*}
V_{*}(x) := \left\{\begin{array}{l}
  V(x), \ \;\;\mbox{ if } x\in Q_1,   \\ \\
   0, \ \;\;\mbox{ if } x\notin Q_1.
\end{array}\right.
\end{eqnarray*} Let
$$
T_1 : W^1_2(Q_1) \longrightarrow W^1_2(\mathbb{R}^2)
$$ be a bounded linear operator such that $T_1w|_{Q_1} = w$ for all $w\in W^1_2(Q_1)$ (see Theorem \ref{ext}). Then
\begin{eqnarray*}
\mathcal{E}_{\mathcal{N}, 2V\mu, Q_1}[w] &\ge& \frac{1}{C_{12}\|T_1\|^2}\|T_1w\|^2_{W^1_2(\Delta_1)} - 2\int_{\Delta_1}V_{*}(x)|(T_1w)(x)|^2\,d\mu(x)\\
&\ge& \frac{1}{C_{12}\|T_1\|^2}\left(\|T_1w\|^2_{W^1_2(\Delta_1)} - 2C_{12}\|T_1\|^2\int_{\Delta_1}V_{*}(x)|(T_1w)(x)|^2\,d\mu(x)\right)
\end{eqnarray*}
Hence
\begin{eqnarray*}
N_-\left(\mathcal{E}_{\mathcal{N}, 2V\mu, Q_1}\right) &\le& N_-\left(\mathcal{E}_{ 2C_{12}\|T_1\|^2V_{*}\mu, \Delta_1}\right)\\
&\le& 2C'_9C_{12}\|T_1\|^2\|V_{*}\|^{(\textrm{av})}_{\mathcal{B}, \Delta_1, \mu}\;,
\end{eqnarray*}
where $C'_9$ is the constant in Remark \ref{nopoincare} and
\begin{eqnarray*}
\mathcal{E}_{ 2C_{12}\|T_1\|^2V_{*}\mu, \Delta_1}[w] = \|w\|^2_{W^1_2(\Delta_1)} - 2C_{12}\|T_1\|^2 \int_{\Delta_1}V_{*}(x)|(w)(x)|^2\,d\mu(x) ,\\
\textrm{Dom}(\mathcal{E}_{ 2C_{12}\|T_1\|^2V_{*}\mu, \Delta_1}) = \left\{w\in W^1_2(\Delta_1)\cap L^2(\Delta_1, V_{*}d\mu)\right\}.
\end{eqnarray*}
Now
$$
\mu(\Delta_1) \le c_1\left(\sqrt{2}\left(2\frac{c_1}{c_0}\right)^{\frac{1}{\alpha}}\right)^{\alpha} = \frac{c_1^2}{c_0}2^{\frac{\alpha}{2} + 1}
$$ and
$$
\mu(Q_1)= \mu\left(B\left(0, \left(2\frac{c_1}{c_0}\right)^{\frac{1}{\alpha}}\right)\right) - \mu\left(B(0, 1)\right) \ge c_02\frac{c_1}{c_0} - c_1 = c_1.
$$So
\begin{equation}\label{q}
\mu(\Delta_1) \le \frac{c^2_1}{c_0}2^{\frac{\alpha}{2} + 1} = \frac{c_1}{c_0}2^{\frac{\alpha}{2} + 1}\mu(Q_1).
\end{equation}
This implies
\begin{eqnarray*}
\|V_{*}\|^{(\textrm{av})}_{\mathcal{B}, \Delta_1, \mu} &=& \sup\left\{\left|\int_{\Delta_1}V_{*}u \,d\mu\right|\;:\;\int_{\Delta_1}\mathcal{A}(|u|)d\mu \le \mu(\Delta_1)\right\}\\ &=& \sup\left\{\left|\int_{Q_1}V u \,d\mu\right|\;:\;\int_{Q_1}\mathcal{A}(|u|)d\mu \le  \frac{c_1}{c_0}2^{\frac{\alpha}{2} + 1}\mu(Q_1)\right\}\\ &\le&\frac{c_1}{c_0}2^{\frac{\alpha}{2} + 1}\|V\|^{(\textrm{av})}_{\mathcal{B}, Q_1, \mu}
\end{eqnarray*}(see Lemma \ref{lemma8}).
Hence
\begin{equation}\label{fol}
N_-\left(\mathcal{E}_{\mathcal{N}, 2V\mu, Q_1}\right) \le C_{11} \|V\|^{(\textrm{av})}_{\mathcal{B}, Q_1, \mu}\;,\;\;\;\forall V\ge 0,
\end{equation} where
$$
C_{11} := 2C'_9C_{12}\|T_1\|^2\frac{c_1}{c_0}2^{\frac{\alpha}{2} + 1}.
$$
As far as the dependency on the measure $\mu$ is concerned, $C_{11}$ depends only on the ratio $\frac{c_1}{c_0}$.\\\\
Let $\xi : Q_1 \longrightarrow Q_n$ by given by $\xi(x) := x\left(2\frac{c_1}{c_0}\right)^{\frac{n -1}{\alpha}}$.  Let $\tilde{V} := V \circ \xi,\, \tilde{\mu} := \mu \circ \xi$ and $\tilde{w} := w \circ \xi$. Then it is easy to see that $\tilde{\mu}$ satisfies the following analogue of \eqref{Ahlfors}
$$
\tilde{c_0}r^{\alpha} \le \tilde{\mu}(B(x, r)) \le \tilde{c_1}r^{\alpha}
$$ for all $0 < r \le$ diam(supp $\tilde{\mu}$) and $x\in$ supp\,$\tilde{\mu}$,  where $\tilde{c_0} := c_0\left(2\frac{c_1}{c_0}\right)^{n-1}$, $\tilde{c_1} := c_1\left(2\frac{c_1}{c_0}\right)^{n - 1}$, and $\frac{\tilde{c_1}}{\tilde{c_0}} = \frac{c_1}{c_0}$.
Now
\begin{eqnarray*}
&&\int_{Q_n}|\nabla w(y)|^2 dy - 2\int_{Q_n}V(y)|w(y)|^2d\mu(y)\\&& = \int_{Q_1}|\nabla\tilde{w}(x)|^2dx - 2 \int_{Q_1}\tilde{V}(x)|\tilde{w}(x)|^2d\tilde{\mu}(x).
\end{eqnarray*}
It follows from \eqref{fol} that
\begin{eqnarray*}
N_-\left(\mathcal{E}_{\mathcal{N}, 2V\mu, Q_n}\right) = N_-\left(\mathcal{E}_{\mathcal{N}, 2\tilde{V}\tilde{\mu}, Q_1}\right) \le C_{11} \|\tilde{V}\|^{(\textrm{av})}_{\mathcal{B}, Q_1, \tilde{\mu}}\;,\;\;\;\forall \tilde{V}\ge 0.
\end{eqnarray*}
Similarly to \eqref{scale} with $c = 1$, we have $\|\tilde{V}\|^{(\textrm{av})}_{\mathcal{B}, Q_1, \tilde{\mu}} = \|V\|^{(\textrm{av})}_{\mathcal{B}, Q_n, \mu}$. Thus
$$
N_-\left(\mathcal{E}_{\mathcal{N}, 2V\mu, Q_n}\right) \le C_{11} \|V\|^{(\textrm{av})}_{\mathcal{B}, Q_n, \mu}\;\;\;\;\;\forall V\ge 0.
$$  Hence the scaling $x \longmapsto x\left(2\frac{c_1}{c_0}\right)^{\frac{n -1}{\alpha}}$ allows one to reduce the case of any $n\in\mathbb{Z}$ to the case $n =1$.
\end{proof}
Let $$\mathcal{D}_n := \|V\|^{(av)}_{\mathcal{B}, Q_n, \mu}.$$ Then for any $c < \frac{1}{C_{11}}$, the above Lemma  together with the variational principle  imply
\begin{equation}\label{meaeqn6*}
N_-\left(\mathcal{E}_{\mathcal{N}, 2V\mu}\right) \le C_{11}\underset{\{n\in\mathbb{Z},\;\mathcal{D}_n > c\}}\sum \mathcal{D}_n\;,\;\;\;\forall V\ge 0.
\end{equation}
Thus Theorem \ref{mainthm} follows from \eqref{meaeqn1},  \eqref{meaeqn5} and \eqref{meaeqn6*}.
\begin{theorem}\label{measthm3}
{\rm Let $V \ge 0$. If $N_-(\mathcal{E}_{\gamma V\mu, \mathbb{R}^2}) = O(\gamma)$ as $\gamma \longrightarrow + \infty$, then $\|G_n\|_{1,w} < \infty$.
}
\end{theorem}
\begin{proof}
This follows by replacing the Lebesgue measure with $\mu$ in the proof of \cite[Theorem 9.1]{Eugene} and the proof of \cite[Theorem 9.2]{Eugene}.
\end{proof}
No estimate of the type
\begin{equation}\label{type}
N_-(\mathcal{E}_{V\mu, \mathbb{R}^2}) \le \textrm{const} + \int_{\mathbb{R}^2}V(x)W(x)\,d\mu(x) + \textrm{const}\|V\|_{\Psi, \mathbb{R}^2, \mu}
\end{equation}
can hold with norm $\|V\|_{\Psi, \mathbb{R}^2, \mu}$ weaker than $\|V\|_{\mathcal{B}, \mathbb{R}^2, \mu}$ provided the weight function $W$ is bounded in a neighbourhood of at least one point in the support of $\mu$ (see the next Theorem).
\begin{theorem}\label{notype}{\rm(cf. \cite[Theorem 9.4]{Eugene})
 Let $W \ge 0$ be bounded in a neighbourhood of at least point in the support of $\mu$ and $\Psi$ an a N-function such that
$$
\underset{s \longrightarrow \infty}\lim\frac{\Psi(s)}{\mathcal{B}(s)} = 0.
$$
Then there exists a compactly supported $V\ge 0$ such that
$$
\int_{\mathbb{R}^2}V(x)W(x)\,d\mu(x) + \|V\|_{\Psi, \mathbb{R}^2, \mu} < \infty
$$ and $N_-(\mathcal{E}_{V\mu, \mathbb{R}^2}) = \infty$.
}
\end{theorem}
\begin{proof}
Shifting the independent variable if necessary, we can assume that $0 \in$ supp $\mu$ and $W$ is bounded
in a neighborhood of $0$. Let $r_0 > 0$ be such that $W$ is bounded in the open ball
$B(0, r_0)$.

Let
$$
\beta(s) := \sup_{t \ge s}\, \frac{\Psi(t)}{\mathcal{B}(t)}\, .
$$
Then $\beta$ is a non-increasing function, $\beta(s) \to 0$ as $s \to \infty$, and
$\Psi(s) \le \beta(s) \mathcal{B}(s)$.
Since $\Psi$ is an $N$-function, $\Psi(s)/s \to \infty$ as $s \to \infty$
(see $\S$\ref{Orliczspaces}). Hence there exists $s_0 \ge e^{\frac{1}{\alpha}} > 1,\,0 < \alpha \le 2$ such that $\Psi(s) \ge s$ and
$\beta(s) \le 1$ for $s \ge s^{\alpha}_0$.
Choose $\rho_k \in (0, 1/s_0)$ in such a way that
$$
\sum_{k = 1}^\infty \beta\left(\frac1{\rho^{\alpha}_k}\right) < \infty .
$$
It follows from \eqref{Ahlfors} that $\forall r > 0$, the disk $B(0, r)$ contains points of the support of $\mu$ different from $0$. Let $x^{(1)}\in$ supp$\,\mu \setminus\{0\}$ be such that $|x^{(1)}| < \frac{2}{3}r_0$. We can choose $x^{(k)},\,k\in\mathbb{N}$ inductively as follows:\\ Suppose $x^{(1)}, ... , x^{(k)} \in$ supp$\,\mu \setminus\{0\}$ have been chosen. Take \\$x^{(k + 1)} \in$ supp $\mu\setminus\{0\}$ such that
$$
|x^{(k+1)}| < \min\left\{ \frac{1}{3}|x^{(k)}|, 2\rho_{k+1}\right\}.
$$
Since $|x^{(k+1)}| < \frac{1}{3}|x^{(k)}|$, it is easy to see that the open disks $B(x^{(k)}, \frac{1}{2}|x^{(k)}|)$, $k \in \mathbb{N}$ lie in $B(0, r_0)$
and are pairwise disjoint. Let $r_k := \frac{1}{2}|x^{(k)}|$. Then $r_k \le \rho_k,\,k\in\mathbb{N}$. For some constant $C_{13} > 0$, let
\begin{eqnarray*}
& t_k := \frac{C_{13}}{\ln\frac1{r_k}}\, r_k^{-2\alpha} & \\ \\
& V(x) := \left\{\begin{array}{ll}
  t_k ,   &  x \in B\left(x^{(k)}, r_k^2\right) , \  k \in \mathbb{N}, \\
    0 ,  &   \mbox{otherwise.}
\end{array}\right. &
\end{eqnarray*}
Since the function $r^{\alpha}_k\ln\frac1{r_k}$ has its maximum as $\frac{e}{\alpha} > 1,\,0 < \alpha \le 2$, then one can choose  $C_{13} > 0$ such that $C_{13}\frac{\alpha}{e} > 1$ and
$$
t_k = \frac{C_{13}}{\ln\frac1{r_k}}\, r_k^{-2\alpha} = \frac{C_{13}}{r_k^{\alpha} \ln\frac1{r_k}}\, r_k^{-\alpha} > \frac1{r^{\alpha}_k}
$$
and
\begin{eqnarray*}
\int_{\mathbb{R}^2}\Psi(V(x))\,d\mu(x) &=& \sum_{k= 1}^{\infty}\Psi(t_k)\mu\left(B(x^{(k)}, r^2_k)\right) \le \sum_{k= 1}^{\infty}\Psi(t_k)c_1 r^{2\alpha}_{k}\\ &\le& c_1\sum_{k =1}^{\infty}r^{2\alpha}_k\beta(t_k)\mathcal{B}(t_k) \\ &\le& c_1\sum_{k =1}^{\infty}r^{2\alpha}_k\beta(t_k)(1 + t_k)\ln(1 + t_k) \\&<& 4c_1\sum_{k =1}^{\infty}r^{2\alpha}_k\beta(t_k)t_k\ln t_k \\ &=& 4c_1\sum_{k =1}^{\infty}\beta(t_k)\frac{C_{13}}{\ln\frac{1}{r_k}}\ln\frac{C_{13}}{r^{2\alpha}_k\ln\frac{1}{r_k}} \\ &\le& 4c_1C_{13}\sum_{k =1}^{\infty}\beta\left(\frac{1}{r^{\alpha}_k}\right)\frac{1}{\ln\frac{1}{r_k}}\ln\frac{C_{13}}{r_k^{2\alpha}}\\&\le& \textrm{const} \sum_{k = 1}^{\infty} \beta\left(\frac{1}{r^{\alpha}_k}\right) \le \textrm{const} \sum_{k = 1}^{\infty} \beta\left(\frac{1}{\rho^{\alpha}_k}\right) < \infty.
\end{eqnarray*} Thus $\|V\|_{\Psi, \mathbb{R}^2, \mu} < \infty$. Since $t_k > \frac{1}{r^{\alpha}_k} \ge s^{\alpha}_0$, then $t_k \le \Psi(t_k)$ and
$$
\int_{\mathbb{R}^2}V(x)\,d\mu(x) \le \int_{\mathbb{R}^2}\Psi(V(x))\,d\mu(x) < \infty.
$$
By the assumption that $W$ is bounded in $B(0, r_0)$, we have
$$
\int_{\mathbb{R}^2}V(x)W(x)\,d\mu(x) < \infty\,.
$$
Let
$$
w_k(x) := \left\{\begin{array}{cl}
  1 ,   & \  |x - x^{(k)}| \le r_k^2 , \\ \\
    \frac{\ln (r_k/|x - x^{(k)}|)}{\ln(1/r_k)} ,  & \ r_k^2 <  |x - x^{(k)}| \le r_k  , \\ \\
 0 , & \  |x - x^{(k)}| > r_k
\end{array}\right.
$$
(cf. \cite{Grig}). Then
$$
\int_{\mathbb{R}^2} |\nabla w_k(x)|^2\, dx = \frac{2\pi}{\ln(1/r_k)}\,.
$$
Further,
\begin{eqnarray*}
\int_{\mathbb{R}^2} V(x) |w_k(x)|^2\, d\mu(x) &\ge& \int_{B\left(x^{(k)}, r_k^2\right)} V(x)\, d\mu(x) =
t_k \mu\left(B\left(x^{(k)}, r_k^2\right)\right)\\ &>& t_k c_0 r^{2\alpha}_k = c_0\frac{C_{13}}{\ln \frac{1}{r_k}}\, .
\end{eqnarray*}
Hence for any $C_{13}> \frac{2\pi}{c_0}$
$$
\mathcal{E}_{V\mu, \mathbb{R}^2}[w_k] < 0 , \ \ \ \forall k \in \mathbb{N}
$$
and $N_- (\mathcal{E}_{V\mu, \mathbb{R}^2}) = \infty$.

\end{proof}

\chapter{Two dimensional Schr\"odinger Operators on a Strip}\markboth{Chapter \ref{Twostrip}\label{Twostrip}.
Introduction}{}\label{Introduction}
\section{Introduction}
In this chapter, we establish upper estimates for the number of  eigenvalues below the essential spectrum of two dimensional Schr\"odinger operators on a strip in terms of weighted $L^1$ and Orlicz ($L\log L$) norms of the potential. Depending on the boundary conditions, the discrete spectrum of such operators might contain positive eigenvalues. However, we still use the strategy used in the previous chapter although some details are different.
Generally, we study the operator on $L^2(S)$
\begin{equation}\label{self-adj}
H_V = -\Delta - V\mu, \;\;\;V\ge 0,
\end{equation}
where $S := \{(x_1, x_2)\in\mathbb{R}^2\;:\;x_1\in\mathbb{R},\; 0 < x_2 < a\},\;a > 0$ is a strip, $V$ and $\mu$ are as defined in $\S$ \ref{index}. Like in the previous case, the problem is also split into two problems. The first one is defined by the restriction of the quadratic form associated with the operator \eqref{self-adj} to the subspace of functions of the form $w(x_1)u_1(x_2)$, where $u_1$ is the first eigenfunction of the one dimensional differential operator on $(0, a)$ and hence, is  reduced to a well studied one-dimensional Schr\"odinger operator with the potential equal to a weighted mean value $\widetilde{V}$ of $V$ over $(0, a)$. The second problem is defined by a class of functions orthogonal to $u_1$ in the $L^2([0, a])$ inner product.\\\\
As a motivation to our problem, we start by looking at a simple case where the operator \eqref{self-adj} is considered on $S$ with  Neumann boundary conditions and $\mu$ being the two dimensional Lebesgue measure (see the next section below). In this case, $\sigma_{\textrm{ess}}(H_V) = [0, \infty)$. Later, we consider $\eqref{self-adj}$ on $S$ with Robin boundary conditions and derive Dirichlet and Dirichlet-Neumann boundary conditions as particular cases.
\section{Motivation: The case of Neumann boundary conditions}\label{Neumann}
Let $S $  be the strip defined above and let $V$ be locally integrable on $S$. Consider the following Schr\"odinger operator on $L^2(S)$
$$H_V := -\Delta - V,\;\;\; V\geq 0  $$
with Neumann boundary conditions both at $x_2 = 0$ and $x_2 = a$.
In this case the first eigenvalue $\lambda_1$ of $-\Delta$ on $[0, a]$ is equal to zero  and the corresponding eigenfunction $u_1(x_2) = 1$. Thus $\sigma(H_V) = [0, \infty)$ (see Theorem\ref{spect}). Define $H_V$ via its quadratic form as follows:
\begin{eqnarray*}
\mathcal{E}^N_{V, S}[u] &:=& \int_S \mid\nabla u(x)\mid^2\, dx  - \int_S V(x)\mid u(x)\mid^2\, dx,\\  \textrm{ Dom}(\mathcal{E}^N_{V, S})&=& W^1_2(S)\cap L^2(S, Vdx).
\end{eqnarray*}Here the superscript $N$ represents the Neumann boundary conditions.
We denote by $N_-(\mathcal{E}^{N}_{V, S})$  the number of negative eigenvalues of $H_V$ counting multiplicities. The task is to find an upper estimate for $N_-(\mathcal{E}^N_{V, S})$ in terms of the norms of $V$.\\\\

Let $$L_1 := \{u\in L^2(S) \; : \; u(x) = u(x_1)\}$$ and  $P : L^2(S)\longrightarrow L_1$ be a projection defined by  $$Pv(x) := \frac{1}{a}\int^a_0 v(x)dx_2 = Pv(x_1).$$
Indeed, $P$ is a projection since $P^2 = P$. Let $L_2 := (I - P)L^2(S)$, then one can show that $L^2(S) = L_1 \oplus L_2$. Here and below $\oplus$ denotes a direct orthogonal sum.\\ Indeed for all $v\in L_2$ we have,
\begin{eqnarray*}
\int^a_0 v(x)dx_2 &=& \int^a_0 (I - P)v(x)dx_2\\ &=& \int^a_0 v(x)dx_2 - \int^a_0 Pv(x)dx_2\\ &=&\int^a_0 v(x)dx_2 - \int^a_0 Pv(x_1)dx_2\\&=& \int^a_0 v(x)dx_2 - \int^a_0 v(x)dx_2\\ &=& 0.
\end{eqnarray*} Now pick $w\in L_1$ and $v\in L_2$, then,
\begin{eqnarray*}
\langle v, w\rangle_{L^2(S)} = \int_S v(x)\overline{w(x_1)}\,dx &=&\int_{\mathbb{R}}\left(\int^a_0 v(x)\,dx_2\right)\overline{w(x_1)}dx_1 \\&=&0.
\end{eqnarray*}
Similarly, let $$\mathcal{H}_1 := \{u\in W^1_2(S) \; : \; u(x) = u(x_1)\} \; \textrm{and}\; \mathcal{H}_2 := (I- P)W^1_2(S),$$ then $$W^1_2(S) = \mathcal{H}_1 \oplus \mathcal{H}_2.$$ Indeed, for all $v\in\mathcal{H}_1$ and all $w\in\mathcal{H}_2$ we have $$<v, w>_{W^1_2(S)} =  \int_S \left(v(x_1)\overline{w(x)} +  v_{x_1}(x_1)\overline{w_{x_1}(x)} +  v_{x_2}(x_1)\overline{w_{x_2}(x)}\right)dx = 0.$$ This is so because $v, v_{x_1}\in L_1, \overline{w}, \overline{w_{x_1}}\in L_2$ and $v_{x_2} = 0$. To see this note that $v(x_1)$ and $v_{x_1}(x_1)$ do not depend on $x_2$ implying that $v_{x_1}\in L_1$. Also, $w\in L_2\Leftrightarrow \int^a_0w(x)dx_2 = 0$. So, $\frac{d}{dx_1}\int^a_0w(x)\,dx_2 = 0 \Rightarrow \int^a_0w_{x_1}(x)\,dx_2 = 0 \Rightarrow w_x\in L_2$. Hence $\int_Sv_x(x)\overline{w_x(x, y)}dxdy = 0$.\\\\ Now  for all $u\in W^1_2(S)$,  $u = v + w, \;\; v\in\mathcal{H}_1,\;\;w\in\mathcal{H}_2$ one has
\begin{eqnarray*}
\int_S \mid\nabla u(x)\mid^2\,dx &=& a\int_{\mathbb{R}}\mid v'(x_1)\mid^2\,dx_1 + \int_S\mid\nabla w(x)\mid^2\,dx\\ &+& \underbrace{\int_S\nabla v(x_1).\overline{\nabla w(x)}\,dx + \int_S\nabla w(x).\overline{\nabla v(x_1)}\,dx}_{= 0}\\ &=&a\int_{\mathbb{R}}\mid v'(x_1)\mid^2\,dx_1 + \int_S\mid\nabla w(x)\mid^2\,dx
\end{eqnarray*}and
\begin{eqnarray*}
\int_S V(x)\mid u(x)\mid^2\,dx &=& \int_SV(x)\mid v(x_1)\mid^2\,dx + \int_SV(x)\mid w(x)\mid^2\,dx \\&+& \underbrace{\int_S 2V(x)\textrm{Re}(v.w)dx}_{\textrm{possibly}\neq 0\; \textrm{because of V}}\\&\le& 2\int_{\mathbb{R}}\widetilde{V}(x_1)\mid v(x_1)\mid^2\,dx_1 + 2\int_SV(x)\mid w(x)\mid^2\,dx\,
\end{eqnarray*} where
$$\widetilde{V}(x_1) = \int^a_0V(x)dx_2\,.$$
So
\begin{eqnarray*}
&&\int_S \mid\nabla u(x)\mid^2 \,dx - \int_S V(x)\mid u(x)\mid^2 \,dx\\ &&\geq \int_\mathbb{R} \mid v'(x_1)\mid^2 \,dx_1 - 2\int_\mathbb{R} \widetilde{V}(x_1)\mid v(x_1)\mid^2\, dx_1 \\&&+ \int_S \mid\nabla w(x)\mid^2\, dx - 2\int_S V(x)\mid w(x)\mid^2\, dx\,.
\end{eqnarray*}  Hence
\begin{equation}\label{estimate}
N_-(\mathcal{E}^N_{V, S}) \leq N_-(\mathcal{E}_{1, 2\widetilde{V}}) + N_-(\mathcal{E}_{2, 2V}),
\end{equation}where $\mathcal{E}_{1, 2\widetilde{V}}$ and $\mathcal{E}_{2, 2V}$ denote the restrictions of the form $\mathcal{E}^N_{2V, S}$ to the spaces $\mathcal{H}_1$ and $\mathcal{H}_2$ respectively.\\Let
\begin{eqnarray*}
I_n := [2^{n - 1}, 2^n], \ n > 0 , \ \ \ I_0 := [-1, 1] , \ \ \
I_n := [-2^{|n|}, -2^{|n| - 1}], \ n < 0 ,
\end{eqnarray*}
\begin{equation}\label{calAn}
\mathcal{A}_n := \int_{I_n} |x_1| \widetilde{V}(x_1)\, dx_1 , \ n \not= 0 , \ \ \ \mathcal{A}_0 := \int_{I_0} \widetilde{V}(x_1)\, dx_1 .
\end{equation} Then similarly to \eqref{Est1} one has
\begin{equation}\label{Est1*}
N_-(\mathcal{E}_{1, 2\widetilde{V}})\leq 1 + 7.61\sum_{\{n\in\mathbb{Z},\; \mathcal{A}_n > 0.046\}}\sqrt{\mathcal{A}_n}\,.
\end{equation}
We write \eqref{Est1*} in terms of the original potential $V$.
\begin{eqnarray*}
\mathcal{A}_n := \int_{I_n} |x_1| \widetilde{V}(x_1)\, dx_1 &=& \int_{I_n}|x_1|\left(\int_0^a V(x)\,dx\right)\,dx_1\\ &=& \int_{I_n\times [0, a]}|x_1|V(x)\,dx =: A_n
\end{eqnarray*} and
\begin{eqnarray*}
\mathcal{A}_0 := \int_{I_0} \widetilde{V}(x_1)\, dx_1 &=& \int_{I_0}\left(\int_0^a V(x)\,dx\right)\,dx_1\\ &=& \int_{I_0\times [0, a]}V(x)\,dx =: A_0\,.
\end{eqnarray*}
Thus \eqref{Est1*} becomes
\begin{equation}\label{Est1**}
N_-(\mathcal{E}_{1, 2\widetilde{V}})\leq 1 + 7.61\sum_{\{n\in\mathbb{Z},\; A_n > 0.046\}}\sqrt{A_n}\,.
\end{equation}

It now remains to find an estimate for $N_-(\mathcal{E}_{2, 2V})$ in \eqref{estimate}.\\

Let $S_n := J_n\times I ,\; n\in\mathbb{Z}$, where $J_n := (n, n + 1)$ and $I := (0, a)$.  Then the variational principle (see \eqref{varstrip}) implies that
\begin{equation}\label{estimate2}
N_-(\mathcal{E}_{2, 2V})\leq \sum_{n\in\mathbb{Z}}N_-(\mathcal{E}_{2, 2V, S_n}),
\end{equation}
where
\begin{eqnarray*}
&& \mathcal{E}_{2, 2V, S_n}[w] := \int_{S_n} |\nabla w(x)|^2\, dx  -
2\int_{S_n} V(x) |w(x)|^2\, dx , \\
&& \mbox{Dom}\, (\mathcal{E}_{2, 2V, S_n}) =
\left\{w\in W^1_2(S_n)\cap L^2\left(S_n, V(x)dx\right) : \int_{S_n}w(x)\, dx = 0\right\}.
\end{eqnarray*}

\begin{lemma}\label{lemma6}{\rm (cf. \cite[Lemma 7.8]{Eugene})}
There exists $C_{14} > 0$ such that
\begin{equation}\label{CLRl4est}
N_- (\mathcal{E}_{2, 2V, S_n}) \le C_{14}
\|V\|_{L_1\left(J_{n}, L_{\mathcal{B}}(I)\right)} , \ \ \ \forall V \ge 0
\end{equation}(see \eqref{L1LlogLnorm}).
\end{lemma}
\begin{proof}
This Lemma follows from Lemma \ref{lemma5} and the scaling $x_2 \longmapsto a x_2$.
\end{proof}
 Let $\mathcal{D}_n := \parallel V\parallel_{L_1(J_n, L_\mathcal{B}(I))}, \;\;\;\; n\in\mathbb{Z}$.
\begin{theorem}\label{thm1}
There exist   a constant $c > 0$ such that
\begin{equation}\label{Gest2}
N_-(\mathcal{E}_{V, S})\leq 1 + 7.61\sum_{\{n\in\mathbb{Z}, {A_n > 0.046\}}}\sqrt{A_n} + C_{14}\sum_{\{n\in\mathbb{Z}, \mathcal{D}_n > c\}}\mathcal{D}_n, \;\;\;\;\forall V\geq 0.
\end{equation}
\end{theorem}
\begin{proof}If $\mathcal{D}_n < \frac{1}{C_{14}}$ , then  $N_-(\mathcal{E}_{2, 2V, S_n}) = 0$ and one can drop this term in the sum \eqref{estimate2}. Hence for any $c < \frac{1}{C_{14}}$ , \eqref{estimate2}  and Lemma \ref{lemma6} imply that
$$
N_-(\mathcal{E}_{2, 2V})\leq C_{14}\sum_{\{n\in\mathbb{Z}, \mathcal{D}_n > c\}}\mathcal{D}_n  \;\;\;\;\; \; \forall V \geq 0.
$$ This together with \eqref{estimate} and \eqref{Est1**} imply  \eqref{Gest2}.
\end{proof}
One can easily show that \eqref{Gest2} is an improvement of the result  by A. Grigor'yan and N. Nadirashvili  (\cite[Theorem 7.9]{Grig}, see also \eqref{Nad*}) with a different $c$ and that \eqref{Gest2} is strictly sharper. Indeed, let
$B_{n}  := \parallel V\parallel_{\mathcal{B}, S_n}$. Then there is a constant $C(p),\; p > 1$ such that
$$
B_{n}  = \parallel V\parallel_{\mathcal{B}, S_n} \le C(p)\left(\int_{S_n}V(x)^p\,dx\right)^{\frac{1}{p}} = C(p)b_n(V),
$$ where $b_n(V) := \left(\int_{S_n}V(x)^p\,dx\right)^{\frac{1}{p}}$
(see \cite[Remark 6.3]{Eugene}).
Now suppose that $ \parallel V\parallel_{(\mathcal{B}, S_n)} = 1$. Since $\mathcal{B}(V)$ satisfies the $\Delta_2$-condition, then $\int_{S_n}\mathcal{B}(V(x))\,dx = 1$ (see (9. 21) in \cite{KR}). Using \eqref{Luxemburgequiv} and \eqref{LuxNormPre}, we have
\begin{eqnarray}\label{imp}
\mathcal{D}_n = \int_{J_n}\|V\|_{\mathcal{B}, I}\,dx_1 &\le& 2\int_{J_n}\|V\|_{(\mathcal{B}, I)}\,dx_1 \nonumber\\ &\le& 2\int_n^{n + 1}\left(1 +\int_0^a\mathcal{B}\left(V(x)\right)dx_2\right)dx_1\nonumber\\
&=& 2 + 2\int_{J_n}\int_0^a\mathcal{B}\left(V(x)\right)dx = 4 \nonumber\\ &=& 4\|V\|_{(\mathcal{B}, S_n)} \le 4\|V\|_{\mathcal{B}, S_n}\nonumber\\ &=& 4B_n \le \;4 C(p)b_n(V).
\end{eqnarray}Hence
\begin{equation}\label{equivnad}
N_- (\mathcal{E}_{2, 2V, S_n}) \le C_{15} b_n(V) , \ \ \ \forall V \ge 0\,,
\end{equation}where $C_{15} := 4C_{14}C(p)$.
The scaling $V \longmapsto t V,\; t > 0$, allows one to extend the above inequality to an arbitrary $V \geq 0$. Thus for any $c < \frac{1}{C_{15}}$, \eqref{Gest2} implies \eqref{Nad*}.\\\\
Next we will discuss different  forms of \eqref{Gest2}.
\begin{remark}
{\rm
Estimate \eqref{Gest2} implies the following estimate
\begin{equation}\label{Est2}
N_-(\mathcal{E}^N_{V, S})\leq 1 + C_{16}\left( \parallel(A_n)_{n\in\mathbb{Z}}\parallel_{1,w} + \|V\|_{L_1\left(\mathbb{R}\;, L_{\mathcal{B}}(I)\right)}\right), \;\;\;\ \forall V \geq 0.
\end{equation}
This follows from
$$
\sum_{\{n\in\mathbb{Z}, A_n > c\}}\sqrt{A_n} \leq \frac{2}{\sqrt{c}}\parallel(A_n)_{n\in\mathbb{Z}}\parallel_{1,w}
$$
(see (49) in \cite{Eugene}) and
 $$
\sum_{\{n\in\mathbb{Z}, \mathcal{D}_n > c\}}\mathcal{D}_n \leq \sum_{n\in\mathbb{Z} }\mathcal{D}_n = \int_\mathbb{R}\parallel V(x_1, .)\parallel_{\mathcal{B}, I}dx_1 = \|V\|_{L_1\left(\mathbb{R}, L_{\mathcal{B}}(I)\right)}.
$$
\eqref{Est2} in turn implies the following
\begin{equation}\label{Est3}
N_-(\mathcal{E}^N_{V, S})\leq 1 + C_{17}\left( \parallel(A_n)_{n\in\mathbb{Z}}\parallel_{1,w} + \int_{\mathbb{R}}\left(\int_{I}\mid V(x)\mid^p dx_2\right)^{\frac{1}{p}}dx_1\right), \;\;\;\ \forall V \geq 0,
\end{equation}
which is equivalent to

\begin{equation}\label{Est4}
N_-(\mathcal{E}^N_{V, S})\leq 1 + C_{18}\left( \parallel(A_n)_{n\in\mathbb{Z}}\parallel_{1,w} + \int_{\mathbb{R}}\left(\int_{I}\mid V_\ast(x )\mid^p dx_2\right)^{\frac{1}{p}}dx_1\right), \;\;\;\ \forall V \geq 0,
\end{equation}
where $V_{\ast}(x) = V(x) - \widetilde{V}(x_1), \;\; \widetilde{V}(x_1) = \int^a_0 V(x)dx_2$. \\Indeed,
\begin{eqnarray*}
&&\left|\int_{\mathbb{R}}\left(\int_{I}\mid V(x)\mid^p dx_2\right)^{\frac{1}{p}}dx_1 - \int_{\mathbb{R}}\left(\int_{I}\mid V_\ast(x)\mid^p dx_2\right)^{\frac{1}{p}}dx_1\right| \\&\leq& \int_{\mathbb{R}}\left(\int_{I}\mid \widetilde{V}(x_1)\mid^p dx_2\right)^{\frac{1}{p}}dx_1 = a^{\frac{1}{p}}\int_{\mathbb{R}}\widetilde{V}(x_1)dx_1\\ &=& a^{\frac{1}{p}}\sum_{n\in\mathbb{Z}}\int_{I_n\times [0, a]}V(x)dx \le  a^{\frac{1}{p}}\sum_{n\in\mathbb{Z}}2^{-|n| + 1}A_n\\&\leq&\textrm{const}\; \underset{{n\in\mathbb{Z}}}\sup A_n \leq \textrm{const}\parallel(A_n)_{n\in\mathbb{Z}}\parallel_{1,w}.
\end{eqnarray*}
Thus \eqref{Est3} and \eqref{Est4} are equivalent.\\\\Similarly,
\begin{eqnarray*}
&&\left| \|V\|_{L_1\left(\mathbb{R}, L_{\mathcal{B}}(I)\right)} - \|V_{\ast}\|_{L_1\left(\mathbb{R}, L_{\mathcal{B}}(I)\right)}\right|\\ &&= \left| \int_{\mathbb{R}}\parallel V(x_1, .)\parallel_{\mathcal{B},I}dx_1 - \int_{\mathbb{R}}\parallel V_{\ast}(x_1, .)\parallel_{\mathcal{B},I}dx_1\right|\\&&\le \int_{\mathbb{R}}\|\widetilde{V}(x_1)\|_{\mathcal{B}, I}dx_1 = \textrm{const}\int_{\mathbb{R}}|\widetilde{V}(x_1)|\,dx_1 \\&&\le \textrm{const}\; \underset{{n\in\mathbb{Z}}}\sup A_n \\&& \leq \textrm{const}\parallel(A_n)_{n\in\mathbb{Z}}\parallel_{1,w}.
\end{eqnarray*}
Hence \eqref{Est2} is equivalent to the following estimate
\begin{equation}\label{Est5}
N_-(\mathcal{E}^N_{V, S})\leq 1 + C_{19}\left( \parallel(A_n)_{n\in\mathbb{Z}}\parallel_{1,w} + \|V_{\ast}\|_{L_1\left(\mathbb{R}, L_{\mathcal{B}}(I)\right)}\right), \;\;\;\ \forall V \geq 0.
\end{equation}
Note the the last term in right hand side of \eqref{Est5} (and \eqref{Est4}) drops out if $V$ does not depend on $x_2$.
}
\end{remark}

\section{Robin boundary conditions}\label{Robin}
These conditions can be considered as a generalization or a linear combination of the Dirichlet and Neumann boundary conditions. Let $V : \mathbb{R}^2 \longrightarrow \mathbb{R}$ be a Borel measurable function and $\mu$ a positive Radon measure on $\mathbb{R}^2$. Consider the following operator
\begin{equation}\label{H}
H^R_V := -\Delta - V\mu\,,\;\;V\ge 0, \;\;\;\;\;\textrm{on}\;\;L^2(S)
\end{equation}
with the following the Robin boundary conditions
\begin{equation}\label{R}
u_{x_2}(x_1, 0) + \alpha u(x_1, 0) = u_{x_2}(x_1,a) + \beta u(x_1, a) = 0,
\end{equation}
 where $\alpha, \beta \in\mathbb{R}$. Here the superscript $R$ represents Robin boundary conditions and $u_{x_2}$ is the  derivative of $u$ with respect to $x_2$. When $\alpha = \beta = 0$ and $\mu$ is the Lebesgue measure, we have the case of Neumann boundary conditions discussed in the previous section.\\\\ As shown earlier  in $\S$ \ref{strip}, depending on the values of $\alpha$ and $\beta$ the negative spectrum of the  free Laplacian on $S$ is not necessarily empty. Let
$$S_n := (n, n + 1)\times (0, a), \; n\in\mathbb{Z}.$$ Throughout, the boundary conditions at $x_1 = n$ and $x_1 = n + 1$ are assumed to be Neumann.
Recall that $\sigma_{\textrm{ess}}(H^R_V) = [\lambda_1, \infty)$ (see Theorem \ref{spect}), where $\lambda_1$ is the first eigenvalue of $-\Delta$ on the rectangle $S_n$. So, instead of $H^R_V$, we consider the operator
\begin{equation}\label{newprob}
H^R_{\lambda_1,V} := -\Delta - \lambda_1 - V\mu\;\;\;\;\;  \textrm{on}\;\; L^2(S)
\end{equation}
with boundary conditions \eqref{R}.
We now have that $\sigma_{\textrm{ess}}(H^R_{\lambda_1,V}) = [0, \infty)$. Define $H^R_{\lambda_1,V}$ via its quadratic form
\begin{eqnarray}\label{S}
\mathcal{E}^R_{\lambda_1, V\mu, S}[u] &:=& \int_S |\nabla u(x)|^2\,dx - \lambda_1\int_{S}|u(x)|^2\,dx\nonumber\\ &-& \alpha\int_{\mathbb{R}}|u(x_1, 0)|^2\,dx_1 + \beta\int_{\mathbb{R}}|u(x_1, a)|^2\,dx_1\nonumber\\ &-& \int_S V(x)|u(x)|^2\,d\mu(x),\\
\textrm{Dom}\left( \mathcal{E}^R_{\lambda_1, V\mu, S}\right) &=& \left\{u\in W^1_2(S)\cap L^2(S, Vd\mu\right\}\nonumber.
\end{eqnarray}
Denote by $N_-\left(\mathcal{E}^R_{\lambda_1, V\mu, S} \right)$ the number of negative eigenvalues of $H^R_{\lambda_1, V}$ counting multiplicities. Then our aim in this section is to find upper estimates for $N_-\left(\mathcal{E}^R_{\lambda_1, V\mu, S} \right)$ in terms of  norms of $V$.\\\\
\begin{definition}{\rm (Local Ahlfors regularity)}
{\rm We say that the measure $\mu$ is locally Ahlfors regular on $S$ if the restriction of $\mu$ to $S_n$ for all $n\in\mathbb{Z}$ is Ahlfors regular (see \eqref{Ahlfors}) and there exist constants $c_2, c_3 > 0$ such that
\begin{equation}\label{locAhlfors}
c_2\mu(S_{n \pm 1}) \le \mu(S_n)\le c_3\mu(S_{n \pm 1}),\,\,\,\,\forall  n\in\mathbb{Z}\,.
\end{equation}
}
\end{definition}
Thus for each $n\in\mathbb{Z}$,
\begin{equation}\label{locAhlfors*}
c^k_2\mu(S_{n \pm k}) \le \mu(S_n)\le c^k_3\mu(S_{n \pm k}),\,\,\,\,\forall  k\in\mathbb{N}\,.
\end{equation}From now onwards, it will be assumed that $\mu$ is locally Ahlfors regular on $S$.
 Let
\begin{eqnarray*}
I_n := [2^{n - 1}, 2^n], \ n > 0 , \ \ \ I_0 := [-1, 1] , \ \ \
I_n := [-2^{|n|}, -2^{|n| - 1}], \ n < 0 ,
\end{eqnarray*}
\begin{eqnarray*}
\mathcal{F}_n &:=& \int_{I_n}\int_0^a|x_1|V(x)|u_1(x_2)|^2\,d\mu(x)\,\,\,\, n\neq 0\,,\\ \mathcal{F}_0 &:=& \int_{I_0}\int_0^aV(x)|u_1(x_2)|^2\,d\mu(x)\,,\\
M_n &:=& \|V\|_{\mathcal{B}, S_n,\mu}\,,
\end{eqnarray*}
where $u_1$ is the normalized eigenfunction of $-\Delta$ on $S_n$ corresponding to $\lambda_1$ (here the normalization is with respect to the Lebesgue measure).
\begin{theorem}\label{rbthm}
{\rm There exist constants $C, c >0$ such that
\begin{equation}\label{rbtheqn}
N_-\left(\mathcal{E}^R_{\lambda_1, V\mu, S} \right) \le 1 + C\left(\underset{\{\mathcal{F}_n > c,\,n\in\mathbb{Z}\}}\sum \sqrt{\mathcal{F}_n} + \underset{\{M_n >c,\,n\in\mathbb{Z}\}}\sum M_n \right).
\end{equation}
}
\end{theorem}
The rest of this section is devoted to the proof of the above Theorem. We begin with the necessary auxiliary results.
\begin{equation}\label{n}
\mathcal{E}^R_{S_n}[u] := \int_{S_n} \mid\nabla u(x)\mid^2 dx + \beta\int_n^{n + 1}\mid u(x_1, a)\mid^2 dx_1 - \alpha \int_n^{n + 1}\mid u(x_1, 0)\mid^2 dx_1,
\end{equation} for all $u\in W^1_2(S_n)$.
Then
$$
\mathcal{E}^R_{S_n}[u] = \lambda\int_{S_n} \mid u(x)\mid^2 \,dx \,,
$$  implying that
$$
\lambda = \frac{\mathcal{E}^R_{S_n}[u]}{\int_{S_n} \mid u(x)\mid^2\, dx}\,,
$$ where $\lambda$ is the  eigenvalue of $-\Delta$ on $S_n$.
By the Mini-Max principle, we have
\begin{eqnarray*}
\lambda_1 = \underset{\underset{u\neq 0}{u\in W^1_2(S_n)}}\min \frac{\mathcal{E}^R_{S_n}[u]}{\int_{S_n} \mid u(x)\mid^2\, dx},\\
\lambda_2 = \underset{\underset{u\neq 0, u \perp u_1}{u\in W^1_2(S_n)}}\min\frac{\mathcal{E}^R_{S_n}[u]}{\int_{S_n} \mid u(x)\mid^2\, dx},
\end{eqnarray*}
where $u_1$ is the normalized eigenfunction of $-\Delta$ on $S_n$ corresponding to $\lambda_1$.
Hence for all $u\in W^1_2(S_n),\;\; u \perp u_1$, one has
$$
 \lambda_2\int_{S_n}\mid u(x)\mid^2\,dx \le \mathcal{E}^R_{S_n}[u],
$$ which in turn implies
\begin{eqnarray*}
\mathcal{E}^R_{S_n}[u] - \lambda_1\int_{S_n}\mid u(x)\mid^2\,dx &=& \mathcal{E}^R_{S_n}[u] - \lambda_2\int_{S_n}\mid u(x)\mid^2\,dx + (\lambda_2 - \lambda_1)\int_{S_n}|u(x)|^2\,dx\\&\geq&(\lambda_2 - \lambda_1)\int_{S_n}|u(x)|^2\,dx .
\end{eqnarray*}
Hence,
\begin{equation}\label{Rbeqn1}
\int_{S_n}|u(x)|^2\,dx \le \frac{1}{\lambda_2 - \lambda_1}\left(\mathcal{E}^R_{S_n}[u] - \lambda_1\int_{S_n}|u(x)|^2\,dx\right),\;\;\;\forall u\in W^1_2(S_n),\;\;u\perp u_1.
\end{equation} Since $\lambda_1 < \lambda_2$, then \eqref{Rbeqn1} holds for all $u\in W^1_2(S_n),\;\;u \perp u_1$.
\begin{lemma}\label{rbthm1}{\rm [Ehrling's Lemma]}
{\rm Let $X_0,\;X_1\; \textrm{and} \;X_2$ be Banach spaces such that $X_2 \hookrightarrow X_1$ is compact and $X_1\hookrightarrow X_0$. Then for every $\varepsilon > 0$, there exists a constant $C(\varepsilon) > 0$ such that
\begin{equation}\label{eqthm1}
\|u\|_{X_1} \le \varepsilon\|u\|_{X_2} + C(\varepsilon)\|u\|_{X_0}, \;\;\;\;\; \forall u\in X_2\,.
\end{equation}
}
\end{lemma}
See, e.g., \cite{MR} for details and proof.
\begin{definition}
\rm {Let $S(\mathbb{R}^2)$ be the class of all functions $\varphi\in\mathcal{C}^{\infty}(\mathbb{R}^2)$ such that for any multi-index $\gamma$ and any $k\in\mathbb{N}$,
$$
\underset{x\in \mathbb{R}^2}\sup(1 + |x|)^k|\partial^{\gamma}\varphi(x)| < \infty .
$$Denote by $S'(\mathbb{R}^2)$ the dual space of $S(\mathbb{R}^2)$. For $s > 0$, let
$$
H^{s}(\mathbb{R}^2) := \left\{u\in S'(\mathbb{R}^2): \int_{\mathbb{R}^2}(1 + |\xi|^2)^{s}|\widehat{u}(\xi)|^2d\xi < \infty\right\},\;\;s\in\mathbb{R}.
$$Here, $\widehat{u}(\xi)$ is the Fourier image of $u(x)$ defined by
$$
\widehat{u}(\xi) = \frac{1}{2\pi}\int_{\mathbb{R}^2}e^{-ix\xi}u(x)dx.
$$}
\end{definition}
Let
$$
 H^s(S_n) := \left\{v = \tilde{v}|_{S_n}\;:\; \tilde{v}\in H^s(\mathbb{R}^2)\right\},
$$
$$
\|v\|_{H^s(S_n)} := \underset{\underset{\tilde{v}|_{S_n} = v}{\tilde{v}\in H^s\left({\mathbb{R}^2}\right)}}\inf \|\tilde{v}\|_{H^s(\mathbb{R}^2)}\;.
$$
Now, let $X_0 = L^2(S_n), X_1 = H^s(S_n)\; \textrm{for}\;\frac{1}{2} < s < 1 \; \textrm{and}\; X_2 = W^1_2 (S_n)$ in Theorem \ref{rbthm1}. That $X_2\hookrightarrow X_1$ is compact follows from the Sobolev compact embedding theorem (see, e.g., \cite[Ch. VII]{Ad} or  \cite[$\S$ 1.4.6]{Maz}). Thus we have the following lemma:
\begin{lemma}\label{rblemma0}
{\rm Let $\varepsilon > 0$ be given. Then there exists a constant $C_{20} >0$ such that
\begin{eqnarray}\label{rblemmaeqn}
&&\int_{n}^{n + 1}|u(x_1, a)|^2\,dx_1 + \int_n^{n + 1}|u(x_1, 0)|^2\,dx_1 \le C_{20}\left(\mathcal{E}^R_{S_n}[u] - \lambda_1\int_{S_n}|u(x)|^2\,dx\right),\nonumber\\
&&\forall u\in W^1_2(S_n),\; u\perp u_1.
\end{eqnarray}
}
\end{lemma}
\begin{proof}
In this proof, we make use of \eqref{Rbeqn1} and Lemma \ref{rbthm1}. For $s > \frac{1}{2}$, the trace theorem  and Lemma \ref{rbthm1} imply
\begin{eqnarray*}
&&\int_n^{n + 1}|u(x_1, a)|^2dx_1 + \int_n^{n + 1}|u(x_1, 0)|^2dx_1 \le C_s\|u\|_{X_1}\\ && \le C_s\left(\varepsilon\left(\int_{S_n}|\nabla u(x)|^2 dx + \int_{S_n}|u(x)|^2dx\right) + C(\varepsilon)\int_{S_n}|u(x)|^2 dx\right)\\ &&=C_s\varepsilon\int_{S_n}|\nabla u(x)|^2dx + C_s(\varepsilon + C(\varepsilon))\int_{S_n}|u(x)|^2dx\\&& = C_s\varepsilon\Big(\mathcal{E}^R_n[u] - \lambda_1\int_{S_n}|u(x)|^2 dx - \beta\int_n^{n + 1}|u(x_1, a)|^2\,dx_1\\&& \,+ \alpha\int_n^{n + 1}|u(x_1, 0)|^2\,dx_1 + \lambda_1\int_{S_n}|u(x)|^2\,dx\Big)  +  C_s(\varepsilon + C(\varepsilon))\int_{S_n}|u(x)|^2\,dx\\&&\le C_s\varepsilon\left(\mathcal{E}^R_{S_n}[u] - \lambda_1\int_{S_n}|u(x)|^2\,dx\right)\\&&\, + C_s\varepsilon \;\textrm{max}\{|\beta|,|\alpha|\}\left( \int_n^{n + 1}|u(x_1, a)|^2\,dx_1 + \int_n^{n + 1}|u(x_1, 0)|^2\,dx_1\right)\\&&  + \; C_s\left(\varepsilon(\lambda_1 + 1) + C(\varepsilon)\right)\int_{S_n}|u(x)|^2\,dx\,.
\end{eqnarray*}Take $\varepsilon\le \frac{1}{2C_s\max\{|\beta|,|\alpha|\}}$. Then
\begin{eqnarray*}
&&\int_n^{n + 1}|u(x_1, a)|^2\,dx_1 + \int_n^{n + 1}|u(x_1, 0)|^2\,dx_1 \\&&\le C_s\varepsilon\left(\mathcal{E}^R_{S_n}[u] - \lambda_1\int_{S_n}|u(x)|^2\,dx\right) + \frac{1}{2} \int_n^{n + 1}|u(x_1, a)|^2\,dx_1\\&& + \frac{1}{2}\int_n^{n + 1}|u(x_1, 0)|^2\,dx_1   + C_s\left(\varepsilon(\lambda_1 + 1) + C(\varepsilon)\right)\int_{S_n}|u(x)|^2\,dx.
\end{eqnarray*}
Hence \eqref{Rbeqn1} yields
\begin{eqnarray*}
&&\int_n^{n + 1}|u(x_1, a)|^2\,dx_1 + \int_n^{n + 1}|u(x_1, 0)|^2\,dx_1\\ &&\le 2C_s\varepsilon\left(\mathcal{E}^R_{S_n}[u] - \lambda_1\int_{S_n}|u(x)|^2\,dx\right) + 2C_s\left(\varepsilon(\lambda_1 + 1) + C(\varepsilon)\right)\int_{S_n}|u(x)|^2\,dx\\&&\le C_s\left(2\varepsilon + \frac{2}{\lambda_2 - \lambda_1}(\varepsilon(\lambda_1 + 1) + C(\varepsilon))\right)\left(\mathcal{E}^R_{S_n}[u] - \lambda_1\int_{S_n}|u(x)|^2\,dx\right)\\&& = C_{20}\left(\mathcal{E}^R_{S_n}[u] - \lambda_1\int_{S_n}|u(x)|^2\,dx\right),
\end{eqnarray*}where
$$
C_{20} := C_s\left(2\varepsilon + \frac{2}{\lambda_2 - \lambda_1}(\varepsilon(\lambda_1 + 1) + C(\varepsilon))\right).
$$
\end{proof}As a consequence of Lemma \ref{rblemma0} and \eqref{Rbeqn1} we have the following Lemma
\begin{lemma}\label{cor1}
There exists a constant $C_{21} > 0$ such that
\begin{equation}\label{eqn0}
\int_{S_n}|\nabla u(x)|^2\,dx \le C_{21}\left(\mathcal{E}^R_{S_n}[u] - \lambda_1\int_{S_n}|u(x)|^2\,dx\right), \;\;\;\;\forall u\in W^1_2(S_n),\;u\perp u_1.
\end{equation}
\end{lemma}
\begin{proof}
\begin{eqnarray*}
\int_{S_n}|\nabla u(x)|^2\,dx &=& \mathcal{E}_{S_n}^R[u] - \lambda_1\int_{S_n}|u(x)|^2\,dx + \lambda_1\int_{S_n}|u(x)|^2\,dx \\&-& \beta \int_n^{n + 1}|u(x_1, a)|^2\,dx_1 + \alpha \int_n^{n + 1}|u(x_1, 0)|^2\,dx_1\\
&\le& \left(1 + C_{20}\max\{|\alpha|, |\beta|\} \right)\left(\mathcal{E}_{S_n}^R[u] - \lambda_1\int_{S_n}|u(x)|^2\,dx \right)\\ &+& \lambda_1\int_{S_n}|u(x)|^2\,dx \\
&\le& \left(1 + C_{20}\max\{|\alpha|, |\beta|\} \right)\left(\mathcal{E}_{S_n}^R[u] - \lambda_1\int_{S_n}|u(x)|^2\,dx \right)\\ &+& \frac{\max\{0,\lambda_1\}}{\lambda_2 - \lambda_1}\left(\mathcal{E}_{S_n}^R[u] - \lambda_1\int_{S_n}|u(x)|^2\,dx \right)\\
&=& \left(1 + C_{20}\max\{|\alpha|, |\beta|\} + \frac{\max\{0,\lambda_1\}}{\lambda_2 - \lambda_1} \right)\left(\mathcal{E}_{S_n}^R[u] - \lambda_1\int_{S_n}|u(x)|^2\,dx \right)\\&=& C_{21}\left(\mathcal{E}_{S_n}^R[u] - \lambda_1\int_{S_n}|u(x)|^2\,dx \right),
\end{eqnarray*}where
$$
C_{21} := 1 + C_{20}\max\{|\alpha|, |\beta|\} + \frac{\max\{0,\lambda_1\}}{\lambda_2 - \lambda_1}\,.
$$
\end{proof}
Let
 $\mathcal{H}_1 := P W^1_2(S)$ and $\mathcal{H}_2 := (I- P)W^1_2(S)$, where
\begin{equation}\label{projection}
Pu(x) := \left(\int_0^au(x)\overline{u_1}(x_2)\,dx_2\right)u_1(x_2) = w(x_1)u_1(x_2)\,, \;\;\;\;\;\;\; \forall u\in W^1_2(S)
\end{equation}
and
$$
w(x_1):= \int_0^au(x)\overline{u_1}(x_2)\,dx_2\,.
$$Then $P$ is a projection since $P^2 = P$.
\begin{lemma}\label{rblemma1}
{\rm For all $u\in W^1_2(S), \;\;\langle (I - P)u, u_1\rangle_{L^2[0, a]} = 0$.}
\end{lemma}
\begin{proof}
Since $Pu = \langle u, u_1\rangle_{L^2[0, a] }u_1$, then
\begin{eqnarray*}
\langle (I - P)u, u_1\rangle_{L^2[0, a]} &=& \langle u - \langle u, u_1\rangle_{L^2[0, a]} u_1, u_1\rangle_{L^2[0, a]}\\&=&\langle u, u_1\rangle_{L^2[0, a]} - \langle u, u_1\rangle_{L^2[0, a]} \langle u_1, u_1\rangle_{L^2[0, a]}\\&=& \langle u, u_1\rangle_{L^2[0, a]} - \langle u, u_1\rangle_{L^2[0, a]}  = 0.
\end{eqnarray*}
\end{proof}
\begin{lemma}\label{rblemma2}
{\rm For all $v\in\mathcal{H}_1, \; \tilde{v}\in\mathcal{H}_2,\;\;\langle v, \tilde{v}\rangle_{L^2(S)} = 0$ and\\  $\langle v_{x_1}, \tilde{v}_{x_1}\rangle_{L^2(S)} = 0$.}
\end{lemma}
\begin{proof}
\begin{eqnarray*}
\langle \tilde{v}, v\rangle_{L^2(S)} &=& \int_S(I - P)u(x) . \overline{w(x_1)u_1({x_2)}}\,dx\\ &=&\int_{\mathbb{R}}\overline{w(x_1)}\left(\int_0^a u(x)\overline{u_1(x_2)}\,dx_2\right)dx_1\\&-&  \int_{\mathbb{R}}\overline{w(x_1)}\left[\left(\int_0^au\overline{u_1(x_2)}\,dx_2\right)\left(\int_0^a u_1(x_2)\overline{u_1(x_2)}\,dx_2\right)\right]dx_1\\&=&\int_{\mathbb{R}}\overline{w(x_1)}\left(\int_0^a u(x)\overline{u_1(x_2)}\,dx_2\right)dx_1\\ &-& \int_{\mathbb{R}}\overline{w(x_1)}\left[\left(\int_0^au(x)\overline{u_1(x_2)}\,dx_2\right)\|u_1\|^2\right]dx_1\\&=&\int_{\mathbb{R}}\overline{w(x_1)}\left(\int_0^a u(x)\overline{u_1(x_2)}\,dx_2\right)dx_1\\ &-& \int_{\mathbb{R}}\overline{w(x_1)}\left(\int_0^a u(x)\overline{u_1(x_2)}\,dx_2\right)dx_1 = 0.
\end{eqnarray*}
Since for all $v\in\mathcal{H}_1 \; \textrm{and}\; \tilde{v}\in\mathcal{H}_2,\;\; v_{x_1}\in P L^2(S), \; \tilde{v}_{x_1}\in (I - P)L^2(S)$, it follows that
$$\langle v_{x_1}, \tilde{v}_{x_1}\rangle_{L^2(S)} = 0.$$
\end{proof}
\begin{lemma}\label{rblemma2*}{\rm Let
$$
\mathcal{E}^R_S[u] := \int_S|\nabla u(x)|^2\,dx  + \beta \int_{\mathbb{R}}|u(x_1, a)|^2\,dx_1 - \alpha \int_{\mathbb{R}}|u(x_1, 0)|^2\,dx_1, \;\;\forall u \in W^1_2(S).
$$ Then
$$
\mathcal{E}^R_S[u] = \mathcal{E}^R_S[v] + \mathcal{E}^R_S[\tilde{v}],\;\;\;\forall u = v + \tilde{v},\,\,v\in\mathcal{H}_1\,,\;\;\tilde{v}\in\mathcal{H}_2\,.
$$
}
\end{lemma}
\begin{proof}
 \begin{eqnarray*}
 \langle\tilde{v}_{x_2}, v_{x_2}\rangle_{L^2(S)} &=& \int_{S}\frac{\partial}{\partial x_2}(I - P)u(x) \frac{\partial}{\partial x_2}(\overline{w(x_1)u_1(x_2)})dx\\&=& \int_{\mathbb{R}}\overline{w(x_1)}\left[\int_0^a\frac{\partial}{\partial x_2}(I - P)u(x) \frac{\partial}{\partial x_2}(\overline{u_1(x_2)})\,dx_2\right]dx_1.
 \end{eqnarray*}Integration by parts and Lemma \ref{rblemma1} give
 \begin{eqnarray*}
 \langle v_{x_2}, \tilde{v}_{x_2}\rangle_{L^2(S)} &=& \int_{\mathbb{R}}\overline{w(x_1)}(I -P)u(x_1, a)\frac{\partial}{\partial x_2}\overline{u_1(a)}\,dx_1\\ &-& \int_{\mathbb{R}}\overline{w(x_1)}(I - P)u(x_1, 0)\frac{\partial}{\partial x_2}\overline{u_1(0)}\,dx_1 \\&+& \int_{\mathbb{R}}\overline{w(x_1)}\left(\underbrace{\lambda_1\int_0^a(I - P)u(x)\overline{u_1(x_2)}\,dx_2}_{= 0}\right)dx_1\\&=&-\beta\int_{\mathbb{R}}\overline{w(x_1)}(I - P)u(x_1, a)\overline{u_1(a)}dx_1\\ &+& \alpha\int_{\mathbb{R}}\overline{w(x_1)}(I - P)u(x_1, 0)\overline{u_1(0)}dx_1.
 \end{eqnarray*}Thus, this together with Lemma \ref{rblemma2} yield
 \begin{eqnarray*}
 \mathcal{E}^R_S( \tilde{v}, v)&=&\int_S\nabla \tilde{v}\overline{\nabla v}\,dx + \beta\int_{\mathbb{R}}\overline{w(x_1)}(I - P)u(x_1, a)\overline{u_1(a)}dx_1\\ &-& \alpha\int_{\mathbb{R}}\overline{w(x_1)}(I - P)u(x_1, 0)\overline{u_1(0)}dx_1\\ &=& - \beta\int_{\mathbb{R}}\overline{w(x_1)}(I - P)u(x_1, a)\overline{u_1(a)}dx_1\\ &+& \alpha\int_{\mathbb{R}}\overline{w(x_1)}(I - P)u(x_1, 0)\overline{u_1(0)}dx_1\\&+&  \beta\int_{\mathbb{R}}\overline{w(x_1)}(I - P)u(x_1, a)\overline{u_1(a)}dx_1\\ &-& \alpha\int_{\mathbb{R}}\overline{w(x_1)}(I - P)u(x_1, 0)\overline{u_1(0)}dx_1 = 0.
 \end{eqnarray*}
 This means that for all $ u \in W^1_2(S)$
 $$
 \mathcal{E}^R_S[u] = \mathcal{E}^R_S[v] + \mathcal{E}^R_S[\tilde{v}], \;\;\;\;\forall u = v + \tilde{v},\,\, v\in\mathcal{H}_1,\;\;\tilde{v}\in\mathcal{H}_2.
 $$
 \end{proof}
 Hence
 $$
 W^1_2(S) = \mathcal{H}_1 \oplus \mathcal{H}_2\,,
 $$
where $\oplus$ denotes the direct orthogonal sum.
\begin{proof}[Proof of Theorem \ref{rbthm}]
Let
\begin{eqnarray*}
\mathcal{E}^R_S[u] &:=& \int_S|\nabla u(x)|^2\,dx - \alpha\int_{\mathbb{R}}|u(x_1, 0)|^2\,dx_1 + \beta \int_{\mathbb{R}}|u(x_1, a)|^2\,dx_1,\\
\textrm{Dom}(\mathcal{E}^R_S) &=& W^1_2(S).
\end{eqnarray*}
 and
\begin{eqnarray*}
\mathcal{E}^R_{\lambda_1,V\mu, S}[u] &:=& \mathcal{E}^R_S[u] - \lambda_1 \int_{S}| u(x)|^2\,dx - \int_S V(x)|u(x)|^2\,d\mu(x),\\
\textrm{Dom} (\mathcal{E}^R_{\lambda_1, V\mu, S}) &=& W^1_2(S)\cap L^2\left(S, Vd\mu\right).
\end{eqnarray*} Then similarly to \eqref{estimate} one has
\begin{equation}\label{radstrip}
N_-\left(\mathcal{E}^R_{\lambda_1,V\mu, S}\right) \le N_-(\mathcal{E}_{1, 2V}) + N_-(\mathcal{E}_{2, 2V})\,
\end{equation}
where $\mathcal{E}_{1, 2V}$ and $\mathcal{E}_{2, 2V}$ are the restrictions of the form $\mathcal{E}^R_{\lambda_1,2V\mu, S}$ to the spaces $\mathcal{H}_1$ and $\mathcal{H}_2$ respectively. We start by estimating the first term in the right-hand side of \eqref{radstrip}.
\\\\
Recall that for all $ u\in\mathcal{H}_1\,, \; u(x) = w(x_1)u_1(x_2)$ (see \eqref{projection}). Let $I$ be an arbitrary interval in $\mathbb{R}$ and let
$$
\nu(I) := \int_I\int_0^a V(x)|u_1(x_2)|^2\,d\mu(x).
$$
Then
\begin{eqnarray*}
\int_S V(x)|u(x)|^2\,d\mu(x) &=& \sum_{k\in\mathbb{Z}}\int_{I_k}\int_0^a V(x)|w(x_1)u_1(x_2)|^2\,d\mu(x)\\&=& \sum_{k\in\mathbb{Z}}\int_{I_k}|w(x_1)|^2\,d\nu(x_1)\\
&=&\int_{\mathbb{R}}|w(x_1)|^2\,d\nu(x_1).
\end{eqnarray*}

 On the subspace $\mathcal{H}_1$, one has
 \begin{eqnarray*}
 &&\int_S\left(\mid\nabla u(x)\mid^2 - \lambda_1\mid u(x)\mid^2\right)dx + \beta\int_{\mathbb{R}}\mid u(x_1, a)\mid^2dx_1\\&& - \alpha\int_{\mathbb{R}}\mid u(x_1, 0)\mid^2dx_1 - 2\int_S V(x)\mid u(x)\mid^2d\mu(x)\\&&= \int_{\mathbb{R}}\mid w'(x_1)\mid^2\left(\int_0^a\mid u_1(x_2)\mid^2dx_2\right)dx_1\\&& + \int_{\mathbb{R}}\mid w(x_1)\mid^2\left(\int_0^a\mid u'_1(x_2)\mid^2dx_2\right)dx_1\\&& - \lambda_1\int_{\mathbb{R}}\mid w(x_1)\mid^2\left(\int_0^a\mid u_1(x_2)\mid^2dx_2\right)dx_1\\&& + \beta\int_{\mathbb{R}}\mid w(x_1)u_1(a)\mid^2dx_1 - \alpha\int_{\mathbb{R}}\mid w(x_1)u_1(0)\mid^2dx_1\\&& - 2\int_{\mathbb{R}}|w(x_1)|^2\,d\nu(x_1).
 \end{eqnarray*}But
 \begin{eqnarray*}
 &&\int_{\mathbb{R}}\mid w(x_1)\mid^2\left(\int_0^a\mid u'_1(x_2)\mid^2dx_2\right)dx_1\\&& = \lambda_1\int_{\mathbb{R}}\mid w(x_1)\mid^2\left(\int_0^a\mid u_1(x_2)\mid^2dx_2\right)dx_1 \\&& - \beta\int_{\mathbb{R}}\mid w(x_1)u_1(a)\mid^2dx_1 + \alpha\int_{\mathbb{R}}\mid w(x_1)u_1(0)\mid^2dx_1,
 \end{eqnarray*} which implies
 \begin{eqnarray}\label{stripH1}
 &&\int_S\left(\mid\nabla u(x)\mid^2 - \lambda_1\mid u(x)\mid^2\right)dx + \beta\int_{\mathbb{R}}\mid u(x_1, a)\mid^2dx_1 \nonumber\\&& - \alpha\int_{\mathbb{R}}\mid u(x_1, 0)\mid^2dx_1 - 2\int_S V(x)\mid u(x)\mid^2dx \nonumber\\&&= \|u_1\|^2\int_{\mathbb{R}}\mid w'(x_1)\mid^2dx_1 - \int_{\mathbb{R}}|w(x_1)|^2\,d\nu(x_1)\nonumber\\&&= \int_{\mathbb{R}}\mid w'(x_1)\mid^2dx_1 - 2\int_{\mathbb{R}}|w(x_1)|^2\,d\nu(x_1).
 \end{eqnarray}
 Hence, we have the following one dimensional Schr\"odinger operator
 $$
 - \frac{d^2}{dx^2_1} - 2\nu\;\;\;\;\; \textrm{on} \;\;L^2(\mathbb{R})\,.
 $$
 Let
\begin{eqnarray*}
\mathcal{E}_{1,2\nu}[w] &:=& \int_{\mathbb{R}}|w'(x_1)|^2\,dx_1 - 2\int_{\mathbb{R}}|w(x_1)|^2\,d\nu(x_1),\\
\textrm{Dom} (\mathcal{E}_{1, 2\nu}) &=& W^1_2(\mathbb{R})\cap L^2\left(\mathbb{R}, d\nu\right),
\end{eqnarray*}

\begin{eqnarray*}
F_n &:=& \int_{I_n}|x_1|\,d\nu(x_1),\;\; n\neq 0,\\ F_0 &:=& \int_{I_0}d\nu(x_1).
\end{eqnarray*}
Then
\begin{equation}\label{radest}
N_-\left(\mathcal{E}_{1,2\nu}\right) \le 1 + 7.16 \underset{\{F_n > 0.046,\;n\in\mathbb{Z}\}}\sum \sqrt{F_n}
\end{equation}(see \eqref{Est1}).
To write the above estimate in terms of the original measure, let
\begin{eqnarray*}
\mathcal{F}_n &:=& \int_{I_n}\int_0^a|x_1|V(x)|u_1(x_2)|^2\,d\mu(x),\;\; n\neq 0,\\ \mathcal{F}_0 &:=& \int_{I_0}\int_0^a V(x)|u_1(x_2)|^2\,d\mu(x).
\end{eqnarray*}
Then $F_n = \mathcal{F}_n$. Hence
\begin{equation}\label{radest1}
N_-\left(\mathcal{E}_{1,2V}\right) \le 1 + 7.16 \underset{\{\mathcal{F}_n > 0.046,\;n\in\mathbb{Z}\}}\sum \sqrt{\mathcal{F}_n}\,.
\end{equation}

Next, we consider the subspace $\mathcal{H}_2 \perp \mathcal{H}_1$ in $W^1_2(S)$. By \eqref{Rbeqn1} and \eqref{eqn0}, one has
\begin{equation}\label{W1}
\|u\|^2_{W^1_2(S_n)} \le \left(\frac{1}{\lambda_2 - \lambda_1} + C_{21}\right)\left(\mathcal{E}^R_{S_n}[u] - \lambda_1\int_{S_n}|u(x)|^2dx\right)
\end{equation}for all $u\in W^1_2(S_n), \; u \perp u_1$.\\
Let $S_n := (n , n+ 1)\times (0, a),\;\;n\in\mathbb{Z}$ be the set $G$ in Lemma \ref{measlemma4} and $S^*_n$ be defined similarly. For each $n$, $S^*_n$ intersects not more than $N_0$ rectangles to the left of $S_n$ and $N_0$ rectangles to right of $S_n$, where $N_0\in\mathbb{N}$ depends only on $a$ and $\theta_0$ in Corollary \ref{cor-direct}. Then \eqref{locAhlfors*} implies
\begin{eqnarray*}
\mu(S^*_n) &\le& \sum_{j = n - N_0}^{n + N_0}\mu({S_j})\\ &=&\mu(S_{n-N_0}) + ... + \mu(S_{n-1}) + \mu(S_n) + \mu(S_{n + 1}) + ... + \mu(S_{n + N_0})\\
&\le& \left(\frac{1}{c^{N_0}_2} + ... + \frac{1}{c_2}\right)\mu(S_n) + \mu(S_n) + \left(\frac{1}{c_2} + ... + \frac{1}{c^{N_0}_2}\right)\mu(S_n)\\
&=&\left(2\left(\frac{1}{c_2} + ... + \frac{1}{c^{N_0}_2}\right)+ 1\right)\mu(S_n)\\ &=&\kappa_0\mu(S_n)\,,
\end{eqnarray*}
where $$\kappa_0 := 2\left(\frac{1}{c_2} + ... + \frac{1}{c^{N_0}_2}\right) + 1\,.$$ Let  $$T: W^1_2(S_n) \longrightarrow W^1_2(\mathbb{R}^2)$$  be a bounded linear operator which satisfies
$$
Tu|_{S_n} = u \;\;\;\forall u \in W^1_2(S_n)\,.
$$
Then it follows from the proof of Lemma \ref{measlemma4} that for any $V \in L_{\mathcal{B}}(S_n, \mu), \, V \ge 0$ and any $n \in \mathbb{N}$
$$
\int_{\overline{S_n}}V(x)|u(x)|^2d\mu(x) \le C_6\frac{c_1}{c_0}2^{\alpha}\kappa_0N^2n^{-1}\|T\|^2\|V\|_{\mathcal{B}, S_n, \mu} \|u\|^2_{W^1_2(S_n)}
$$for all $u \in W^1_2(S_n) \cap C(\overline{S_n})$ satisfying the $n_0$ orthogonality conditions in Lemma \ref{measlemma4}.
Hence \eqref{W1} implies
\begin{equation}\label{W2}
\int_{\overline{S_n}}V(x)|u(x)|^2d\mu(x) \le C_{22}n^{-1} \|V\|_{\mathcal{B}, S_n, \mu}\left(\mathcal{E}^R_{S_n}[u] - \lambda_1\int_{S_n}|u(x)|^2dx\right)
\end{equation} for all $u \in W^1_2(S_n) \cap C(\overline{S_n}), \, u \perp u_1$ satisfying the $n_0$ orthogonality conditions, where
$$
C_{22}:=  C_6\frac{c_1}{c_0}2^{\alpha}\kappa_0N^2\|T\|^2\left(\frac{1}{\lambda_2 - \lambda_1} + C_{21}\right).
$$
Let
\begin{eqnarray}\label{S_n}
\mathcal{E}_{2,2V\mu, S_n}[u] &:=& \mathcal{E}_{S_n}^R[u] - \lambda_1\int_{S_n}| u(x)|^2\,dx - 2\int_{\overline{S_n}}V(x)|u(x)|^2\,d\mu(x),\nonumber\\
\textrm{Dom} (\mathcal{E}_{2, 2V\mu, S_n}) &=&  W^1_2(S_n)\cap L^2\left(S_n, V d\mu\right)\;\;\;\;\;(\textrm{see} \eqref{n}).
\end{eqnarray}
Taking $ n = \left[\|V\|_{\mathcal{B}, S_n, \mu}\right] + 1 $ in \eqref{W2}, one gets similarly to  Lemma \ref{measlemma5}
\begin{equation}\label{radest2}
N_-\left(\mathcal{E}_{2,2V\mu, S_n}\right) \le  C_{23}\|V\|_{\mathcal{B}, S_n, \mu} + 2,\;\;\;\;\;\forall V \geq 0
\end{equation}
where $C_{23} := 2C_{22}$ .
Again, taking $n = 1$ in \eqref{W2}, we get
$$
2\int_{\overline{S_n}}V(x)|u(x)|^2\,d\mu(x) \le C_{23}\|V\|_{\mathcal{B}, S_n, \mu}\left(\mathcal{E}^R_{S_n}[u] - \lambda_1\int_{S_n}|u(x)|^2dx\right)\,,
$$for all $u \in W^1_2(S_n) \cap C(\overline{S_n}), \, u \perp u_1$ satisfying the $n_0$ orthogonality conditions.  If $\|V\|_{\mathcal{B}, S_n, \mu} \le \frac{1}{C_{23}}$, then
$$
N_-\left(\mathcal{E}_{2,2V\mu, S_n}\right) = 0\,.
$$ Otherwise, \eqref{radest2} implies
$$
N_-\left(\mathcal{E}_{2,2V\mu, S_n}\right)\le C_{24}\|V\|_{\mathcal{B}, S_n, \mu}\,,
$$
where $C_{24} := 3 C_{23}$.\\\\ Let $M_n = \|V\|_{\mathcal{B}, S_n, \mu}$. Then for any $c \le \frac{1}{C_{23}}$, the variational principle (see \eqref{varstrip}) implies
\begin{equation}\label{radest3}
N_-\left(\mathcal{E}_{2,2V}\right)\le C_{24}\underset{\{M_n >\; c,\;n\in\mathbb{Z}\}}\sum M_n,\;\;\;\;\forall V \geq 0\,.
\end{equation}
 Thus \eqref{radstrip}, \eqref{radest1} and \eqref{radest3} imply \eqref{rbtheqn}.
 \end{proof}
The presence of the terms $\sqrt{\mathcal{F}_n}$ and $M_n$ in the estimate indicate that different parts of the potential contribute differently to number the $N_-(\mathcal{E}^R_{\lambda_1, V, S})$. Since for the terms $\sqrt{\mathcal{F}_n}$, $V$ is integrated over long rectangles, then the long range effect of $V$ in the $x_1$-direction becomes similar to that of one-dimensional potential.
\begin{theorem}\label{sthm1}{\rm(cf. Theorem \ref{measthm3})
Let $V \ge 0$. If $N_-\left(\mathcal{E}^R_{\lambda_1,\gamma V\mu, S}\right) = O(\gamma)\mbox{ as } \gamma \longrightarrow +\infty$, then $\|\mathcal{F}_n\|_{1, w} < \infty$.
}
\end{theorem}
\begin{proof}
For $\sigma > 0$, consider the function
$$
w_n(x_1) := \left\{\begin{array}{cl}
  0 ,   & \  |x_1| \le 2^{n-1}\,\textrm{or}\, |x_1|\ge 2^{n+2} , \\ \\
  2^{(\sigma -1)(n + 1)}(x_1 - 2 ^{n-1}) ,  & \ 2^{n-1} <  |x_1| < 2^n  , \\ \\
 x_1^{\sigma} , & \ 2^n \le |x_1| \le  2^{n+1}, \\ \\
 2^{(\sigma -1)(n+1)}(2^{n+2} - x_1)  ,  & \ 2^{n+1} <  |x_1| < 2^{n+2}\,.
\end{array}\right.
$$ Let $u_n(x) = w_n(x_1)u_1(x_2)$.
Then by a computation similar to the one leading to \eqref{stripH1} we get
\begin{eqnarray*}
\mathcal{E}^R_S[u_n] &-& \lambda_1\int_{S}|u_n(x)|^2\,dx =\int_S\left(|\nabla u_n(x)|^2 - \lambda_1|u_n(x)|^2\right)dx\\ &-& \alpha \int_{\mathbb{R}}|u_n(x_1, 0)|^2dx_1 + \beta \int_{\mathbb{R}}|u_n(x_1, a)|^2dx_1 = \int_{\mathbb{R}}|w'_n(x_1)|^2dx_1\\&=& \int_{2^{n-1}}^{2^n}|w'_n(x_1)|^2dx_1 + \int_{2^{n}}^{2^{n+1}}|w'_n(x_1)|^2dx_1+\int_{2^{n+1}}^{2^{n+2}}|w'_n(x_1)|^2dx_1\\ &=& 2^{(2\sigma -1)n}\left( 2 + \frac{\sigma^22^{2\sigma -1}}{2\sigma - 1} + 2^{2\sigma -1}\right)\\ &=& C_{25}2^{(2\sigma -1)n}\,,
\end{eqnarray*}where
$$
C_{25} := \left( 2 + \frac{\sigma^22^{2\sigma -1}}{2\sigma - 1} + 2^{2\sigma -1}\right),\,\,\sigma \not= \frac{1}{2}\,.
$$
Now
\begin{eqnarray*}
\int_S V(x)|u_n(x)|^2\,d\mu(x) &\ge& \int_{2^n}^{2^{n+1}}x_1^{2\sigma}\int_0^a V(x)|u_1(x_2)|^2\,d\mu(x)\\ &=&\int_{2^n}^{2^{n+1}}x_1^{2\sigma -1}|x_1|\int_0^a V(x)|u_1(x_2)|^2\,d\mu(x)\\&\ge& 2^{(2\sigma -1)n}\mathcal{F}_n\\ &=& \frac{1}{C_{25}}\left(\mathcal{E}^R_S[v_n] - \lambda_1\int_S|v_n(x)|^2\,dx\right)\mathcal{F}_n\,.
\end{eqnarray*}
Hence $\mathcal{E}^R_{\lambda_1, V\mu, S}[u_n] < 0$ if $\mathcal{F}_n > C_{25}$. The fact that $u_n$ and $u_k$ have disjoint supports if $|m - k| \ge 3$ implies that
$$
N_-\left(\mathcal{E}^R_{\lambda_1, V\mu, S}\right) \ge \frac{1}{3}\textrm{card}\{n\in\mathbb{Z}\,:\, \mathcal{F}_n > C_{25}\}
$$ (see \cite[Theorem 9.1]{Eugene}). If $N_-\left(\mathcal{E}^R_{\lambda_1, \gamma V\mu, S}\right) \le C\gamma$, then
$$
\frac{1}{3}\textrm{card}\{n\in\mathbb{Z}\,:\, \gamma\mathcal{F}_n > C_{25}\} \le C\gamma\,,
$$which implies
$$
\textrm{card}\left\{n\in\mathbb{Z}\,:\, \mathcal{F}_n > \frac{C_{25}}{\gamma}\right\} \le 3C\gamma\,.
$$ With $s = \frac{C_{25}}{\gamma}$ we have
$$
\textrm{card}\{n\in\mathbb{Z}\,:\, \mathcal{F}_n > s\} \le C_{26}s^{-1},\,\,\,s > 0,
$$ where $C_{26}:= 3 C_{25}C$.
\end{proof}

Now, suppose that $\mu = |\cdot|$, the  Lebesgue measure . Then
\begin{eqnarray*}
 \mathcal{F}_n &=&   \int_{I_n}|x_1|\left(\int_0^aV(x)|u_1(x_2)|^2dx_2\right)dx_1,  \;\;\;\; n\neq 0, \\
 \mathcal{F}_0 &=&  \int_{I_0}\left(\int_0^aV(x)|u_1(x_2)|^2dx_2\right)dx_1.
 \end{eqnarray*} Let $J_n := (n, n + 1),\;\;I := (0, a)$ and
 $$
 \mathcal{D}_n := \|V\|_{L_1\left(J_n, L_{\mathcal{B}}(I)\right)}\,,
 $$ (see \eqref{L1LlogLnorm}). Then one has the following better estimate
 \begin{equation}\label{radest4}
N_-\left(\mathcal{E}^R_{\lambda_1,V, S}\right)\le 1 + 7.16 \underset{\{\mathcal{F}_n > \;0.046,\;n\in\mathbb{Z}\}}\sum \sqrt{\mathcal{F}_n} +  C_{27}\underset{\{\mathcal{D}_n >\; c,\;n\in\mathbb{Z}\}}\sum \mathcal{D}_n,\;\;\;\;\forall V \geq 0\,.
\end{equation}
Indeed, suppose that $ \parallel V\parallel_{\mathcal{B}, S_n, |\cdot|}  = 1$, similarly to \eqref{imp} one has
\begin{equation}\label{equivnad*}
\mathcal{D}_n = \int_{J_n}\|V\|_{\mathcal{B}, I}\,dx_1  \le 4\|V\|_{\mathcal{B}, S_n, |\cdot|} = 4M_n\,.
\end{equation}The scaling $V \longmapsto t V,\; t > 0$, allows one to extend the above inequality to an arbitrary $V \geq 0$.
 By the same procedure leading to estimate \eqref{Est5}  one has the following estimate
\begin{equation}\label{est14}
N_-(\mathcal{E}^R_{\lambda_1, V, S})\leq 1 + C_{28}\left( \parallel(\mathcal{F}_n)_{n\in\mathbb{Z}}\parallel_{1,w} + \|V_{*}\|_{L_1\left(\mathbb{R}, L_{\mathcal{B}}(I)\right)}\right), \;\;\;\ \forall V \geq 0.
\end{equation}
where $V_{*} := V(x) - G(x_1)$ and
$$
G(x_1) :=\int_0^aV(x)|u_1(x_2)|^2\,dx_2\,.
$$

The condition $\|\mathcal{F}_n\|_{{1, w}} < \infty$ is necessary and sufficient for the semi-classical behaviour of the estimate coming from the subspace $\mathcal{H}_1$ (see the above Theorem). Roughly speaking, the terms $\mathcal{F}_n$ are responsible for the negative eigenvalues of $H_{\lambda_1, V}$ in the $x_1$-direction. In addition, if $V_* \in L_1\left(\mathbb{R}, L_{\mathcal{B}}(I)\right)$, then one gets an analogue of Theorem 1.1 in \cite{LapSolo}, i.e.,
$$
N_-(\mathcal{E}^R_{\lambda_1, \gamma V, S}) = O(\gamma) \;\;\textrm{as}\;\;\gamma \longrightarrow\; +\infty
$$ if and only if $\mathcal{F}_n \in l_{1, w}$.
\begin{remark}
\rm {When $\alpha = 0 \,(\alpha \longrightarrow \pm\infty)$ we have Neumann- Robin (Dirichlet-Robin respectively) boundary conditions.\\
If $\alpha, \beta \longrightarrow \pm\infty$, we obtain Dirichlet boundary conditions. In this case, $\lambda_1 = \left(\frac{\pi}{a}\right)^2, \; \lambda_2 =\min\left\{ 4\left(\frac{\pi}{a}\right)^2, \left(\frac{\pi}{a}\right)^2 + \pi^2\right\}$ and $u_1(x_2) = \sqrt{\frac{2}{a}}\sin \frac{\pi}{a}x_2$. \\
If $\alpha \longrightarrow \infty, \; \beta = 0$, we obtain Dirichlet-Neumann boundary conditions. In this case, $\lambda_1 = \left(\frac{\pi}{2a}\right)^2, \; \lambda_2 = \min\left\{9\left(\frac{\pi}{2a}\right)^2, \left(\frac{\pi}{2a}\right)^2 + \pi^2\right\}$ and $u_1(x_2) = \sqrt{\frac{2}{a}}\sin \frac{\pi}{2a}x_2$.
}
\end{remark}

\chapter{Appendices}\markboth{Chapter \ref{appendix}.
Introduction}{}\label{Introduction}
\section{Besecovitch covering Lemma}\label{besico}
Let $A$ be a bounded set in $\mathbb{R}^n$ and $\Xi$ denote a covering of $\overline{A}$ by cubes $\Delta \subset\mathbb{R}^n$. Suppose $\Xi$ can be split into $r$ subsets $\Xi_1, ..., \Xi_r$ in such a way that for each $k = 1, ..., r$, the cubes $\Delta \in\Xi_k$ are pairwise disjoint. The smallest number $r$ for which such a split of $\Xi$ is possible is called the \textit{linkage} of $\Xi$ and we denote it by $\mbox{ link }(\Xi)$.
\begin{lemma}{\rm \cite[Theorem 1.1]{Guz}
 Let $\Delta_x\subset\mathbb{R}^n$ be a closed cube centred at $x$ for any $x\in \overline{A}$. Then a finite or countable subset $\Xi = \{\Delta_{x_j}\}$ can be chosen in  such a way that $\overline{A}\subset \underset{j}\cup\Delta_{x_j}$ and $\mbox{ link }(\Xi) \le N$, where $N$ is a number depending only on the dimension $n$. For $n =2$, the optimal bound for $\mbox{ link }(\Xi)$ is 19 (see, e.g., \cite[Theorem 2.7]{FM}).
}
\end{lemma}
\section{The classical Poincar\'e inequality}\label{poincare}
For reference, see, e.g.,\cite[1.1.11]{Maz} and the references therein.
Assume that $1 \le p < \infty$ and that $\Omega$ is a bounded connected domain in $\mathbb{R}^n$ with a Lipschitz boundary. Then there is a constant $C$, depending only on $\Omega$ and $p$ such that for all $u\in W^1_p(\Omega)$
$$
\|u - u_{\Omega}\|_{L^p(\Omega)} \le C\|\nabla u \|_{L^p(\Omega)}\,,
$$
where $u_{\Omega} := \frac{1}{|\Omega|}\int_{\Omega}u(x)\,dx$ is the average value of $u$ over $\Omega$.
\section{Proofs of \eqref{*}, \eqref{**} and \eqref{***}}\label{proof}
Let $\mathcal{B}(s) = (1 + s)\ln(1 + s) - s = \frac{1}{t}$, then $s = \mathcal{B}^{-1}\left(\frac{1}{t}\right)$. For small values of $s$ (large values of $t$), using
$$
\ln(1 + s) = s - \frac{s^2}{2} + \frac{s^3}{3} + O\left(s^4\right),
$$ we have
$$
(1 + s)\ln(1 + s) - s = \frac{s^2}{2} + O\left(s^3\right) = \frac{1}{t}.
$$ One can write this in the form
\begin{eqnarray*}
\frac{s^2}{2} + s^2g(s) &=& \frac{1}{t},\;\;\;\;g(0)= 0,\\
\frac{s^2}{2}\left(1 + 2g(s)\right) &=& \frac{1}{t},\\
s\left(1 + h(s)\right) &=& \sqrt{\frac{2}{t}},\;\;\;\;h(0) = 0,
\end{eqnarray*} where $g$ and $h$ are $C^{\infty}$ smooth functions in a neighbourhood of $0$.
Let $f(s) = s\left(1 + h(s)\right)$. Then $f(0)= 0, f'(0) = 1$ and $(f^{-1})'(0) = 1$, which means that both $f$ and $f^{-1}$ are invertible in a neighbourhood of $0$. So by Taylor series we have
$$
s = f^{-1}\left(\sqrt{\frac{2}{t}}\right) = \sqrt{\frac{2}{t}} + O\left(\frac{1}{t}\right).
$$
Thus
$$
\mathcal{B}^{-1}\left(\frac{1}{t}\right) = \sqrt{\frac{2}{t}}\left(1 + o(1)\right) \;\;as \;\; t\longrightarrow\infty
$$
and
$$
t\mathcal{B}^{-1}\left(\frac{1}{t}\right) = \sqrt{2t}\left( 1 + o(1)\right)\;\;as \;\; t\longrightarrow\infty.
$$
For large values of $s$ (small values of $t$), let $\rho = 1 + s$ and $r = \frac{1}{t}$, then
$$
\rho\ln\rho - \rho + 1 = r.
$$ Let $\rho = e^z$, then
\begin{equation}\label{asymp}
ze^z - r - e^z + 1 = 0.
\end{equation} This implies
\begin{eqnarray*}
(z -1)e^z &=& r -1, \\ (z -1 )e^{z -1} &=& \frac{r - 1}{e}.
\end{eqnarray*}
Let $w := z - 1 \,\,\,\,\,\, v:=  \frac{r - 1}{e}$. Then
\begin{equation}\label{asymp1}
we^w = v.
\end{equation}
The solution of \eqref{asymp1} is given by
$$
w = \ln v - \ln\ln v + \frac{\ln\ln v}{\ln v} + O\left(\left(\frac{\ln\ln v}{\ln v}\right)^2\right)
$$ (see (2.4.10) and the formula following (2.4.3) in \cite{DEB}). So
\begin{eqnarray*}
z &=& 1 + \ln\frac{r - 1}{e} - \ln\ln\frac{r - 1}{e} + \frac{\ln\ln \frac{r- 1}{e}}{\ln \frac{r-1}{e}}\\
&& + \, O\left(\left(\frac{\ln\ln \frac{r- 1}{e}}{\ln \frac{r-1}{e}}\right)^2\right).
\end{eqnarray*}
Since
\begin{eqnarray*}
&&\ln (r-1) = \ln r + O\left(\frac{1}{r}\right),\\&&\ln\left(\ln (r -1) - 1\right) = \ln\ln r + O\left(\frac{1}{\ln r}\right),
\end{eqnarray*}
we get
\begin{eqnarray*}
 z &=& \ln r - \ln\ln r + \frac{\ln\ln r}{\ln r} + O\left(\frac{1}{\ln r}\right)\\ &=&\ln\frac{1}{t} - \ln\ln\frac{1}{t} + \frac{\ln\ln\frac{1}{t}}{\ln\frac{1}{t}} + O\left(\frac{1}{\ln\frac{1}{t}}\right).
\end{eqnarray*}
This implies
$$
\rho = e^z =  \frac{1}{t\ln\frac{1}{t}}\left( 1 +  \frac{\ln\ln \frac{1}{t}}{\ln\frac{1}{t}} +  O\left(\frac{1}{\ln\frac{1}{t}}\right)\right).
$$
 Hence

$$
t\mathcal{B}^{-1}\left(\frac{1}{t}\right) = \frac{1}{\ln\frac{1}{t}}\left(1 +  \frac{\ln\ln \frac{1}{t}}{\ln\frac{1}{t}} +  O\left(\frac{1}{\ln\frac{1}{t}}\right)\right)
$$
implying
$$
t\mathcal{B}^{-1}\left(\frac{1}{t}\right) = \frac{1}{\ln\frac{1}{t}}\left( 1 + o(1)\right)\;\;as \;\; t\longrightarrow 0.
$$
Let
$$
\frac{1}{\ln\frac{1}{t}}\left( 1 + o(1)\right) = \tau,
$$ then
\begin{equation}\label{t1}
\ln\frac{1}{t} = \frac{1 + o(1)}{\tau}.
\end{equation}
From
$$
\frac{1}{\ln\frac{1}{t}}\left(1 +  \frac{\ln\ln \frac{1}{t}}{\ln\frac{1}{t}} +  O\left(\frac{1}{\ln\frac{1}{t}}\right)\right) = \tau,
$$
we get
\begin{equation}\label{t2}
\ln\frac{1}{t} = \frac{1 +  \frac{\ln\ln \frac{1}{t}}{\ln\frac{1}{t}} +  O\left(\frac{1}{\ln\frac{1}{t}}\right)}{\tau}.
\end{equation}
Now \eqref{t1} implies
\begin{eqnarray*}
\ln\frac{1}{t} &=& \frac{1 +  \frac{\ln\frac{1 + o(1)}{\tau}}{1 + o(1)}\tau + O\left( \frac{\tau}{1 + o(1)}\right)}{\tau} \\&=& \frac{1 +  (1 + o(1))\tau \ln \frac{1}{\tau} + O(\tau)}{\tau}\,.
\end{eqnarray*}
Substituting this into \eqref{t2}, one gets
\begin{eqnarray*}
\ln\frac{1}{t} &=& \frac{1 + \frac{\ln\frac{1 + (1 + o(1))\tau\ln\frac{1}{\tau} + O(\tau)}{\tau}}{1 + (1 + o(1))\tau\ln \frac{1}{\tau} + O(\tau)}\tau + O(\tau)}{\tau}\\ &=& \frac{1}{\tau} -  \ln\tau + O(1).
\end{eqnarray*}
Hence
\begin{equation}\label{large}
t = \tau e^{-\frac{1}{\tau}}e^{O(1)} =: \varphi(\tau).
\end{equation}
\section{An analogue of Lemma \ref{measlemma3}}\label{analogue}
Let $\Delta := I_1\times I_2$ be a rectangle of sides of lengths $R_1$ and $R_2$ respectively and $\Delta^*$ the rectangle with the same centre as $\Delta$ and with sides of lengths 3 times those of $\Delta$. Let
\begin{equation}\label{norm*}
\|V\|^{(*)}_{\Psi, \Delta, \mu} := \sup\left\{\left|\int_{\Delta}Vu\,d\mu\right| : \int_{\Delta}\Phi(|u|)\,d\mu \le \mu(\Delta^*)\right\}.
\end{equation}
Then we have the following Lemma.
\begin{lemma}\label{clemma1}
{\rm Let $\mu$ be a $\sigma$-finite positive Radon measure that is Ahlfors regular. Then for any $V\in L_{\Psi}(\Delta, \mu),\,V\ge 0$, there is a constant $d_1 > 0$ such that
\begin{equation}\label{ceqn1}
\int_{\Delta}V(x)|w(x)|^2 d\mu(x) \le d_1\|V\|^{(*)}_{\Psi, \Delta, \mu}\int_{\Delta}|\nabla w(x)|^2 dx
\end{equation} for all $w \in W^1_2(\Delta)\cap C(\overline{\Delta})$ with $w_{\Delta} = 0$.
}
\end{lemma}
\begin{proof}
If supp$\,\mu \cap \Delta = \emptyset$, the left-hand side of \eqref{ceqn1} is 0 and there is nothing to prove. Suppose that supp$\,\mu \cap \Delta \not= \emptyset$. Then $\forall x\in$ supp$\;\mu \cap \Delta$, there exists a rectangle $\Delta^o$ centred at $x$ with sides of lengths twice those of $\Delta$ such that $\Delta \subseteq \Delta^o\subseteq \Delta^*$. Let $w^*$ be the extension of $w$ outside $\Delta$. Then similarly to \eqref{ext1} there is a constant $d_2> 0$ such that
$$
\int_{\Delta^o}|\nabla w^*(x)|^2\,dx \le d_2\int_{\Delta}|\nabla w(x)|^2\,dx\,.
$$
Let $V_*$ be the extension by 0 of $V$ outside $\Delta$. Then Lemma \ref{lemma8} and Lemma \ref{measlemma3} imply
\begin{eqnarray*}
\int_{\Delta}V(x)|w(x)|^2d\mu(x) &=& \int_{\Delta^o}V_*(x)|w(x)|^2d\mu(x)\\ &\le& A_2 \frac{c_1}{c_0}2^{\alpha}\max\left\{\frac{R_1}{R_2}, \frac{R_1}{R_2}\right\}^{\alpha + 1}\|V_*\|^{(\textrm{av})}_{\Psi, \Delta^o, \mu}\int_{\Delta^o}|\nabla w^*(x)|^2dx\\&\le&  A_2 \frac{c_1}{c_0}2^{\alpha}\max\left\{\frac{R_1}{R_2}, \frac{R_1}{R_2}\right\}^{\alpha + 1}d_2\|V\|^{(*)}_{\Psi, \Delta, \mu}\int_{\Delta}|\nabla w(x)|^2dx \\
&=&d_1\|V\|^{(*)}_{\Psi, \Delta, \mu}\int_{\Delta}|\nabla w(x)|^2dx\,,
\end{eqnarray*}where
$$
d_1 :=  A_2d_2 \frac{c_1}{c_0}2^{\alpha}\max\left\{\frac{R_1}{R_2}, \frac{R_1}{R_2}\right\}^{\alpha + 1}\,.
$$
\end{proof}
One would ask whether norm \eqref{norm*} is superadditive, i.e. whether \eqref{bsr1} holds with $\|\cdot\|^{(*)}$ in place of $\|\cdot \|^{(\textrm{av})}$, which is an important property used in the proof of Lemma \ref{measlemma4}. It follows from Lemma \ref{a} that if pairwise disjoint subsets $\Omega_k$ of $\Omega$ satisfy the condition
\begin{equation}\label{newco}
\sum_k\mu(\Omega^*_k) \le \kappa\, \mu(\Omega^*)
\end{equation}
with some $\kappa > 0$, then
\begin{equation*}
\sum_k\|V\|^{(*)}_{\Psi, \Omega_k, \mu} \le \kappa\,\|V\|^{(*)}_{\Psi, \Omega, \mu}\,.
\end{equation*}

Let $\Delta_1, ..., \Delta_N \subseteq \mbox{ supp }\mu \cap \Delta$ be pairwise disjoint squares of sides of length $l_k,\, k = 1, ..., N$ centered in the support of $\mu$. Then $\Delta^*_k\,,\,k = 1, ..., N$ are also centred in the support of $\mu$. Let us show that \eqref{newco} holds in this case.
By \eqref{Ahlfors} one has
$$
\mu(\Delta^*_k) \le c_1 \left(\frac{3}{\sqrt{2}}\right)^{\alpha}l_k^{\alpha}
$$
and
$$
\mu(\Delta_k) \ge c_0 \left(\frac{1}{2}\right)^{\alpha}l_k^{\alpha}\,.
$$
This implies
$$
\mu(\Delta^*_k) \le \frac{c_1}{c_0}(3\sqrt{2})^{\alpha}\mu(\Delta_k).
$$
Hence
\begin{eqnarray*}
\sum_{k = 1}^N\mu(\Delta^*_k) &\le& \frac{c_1}{c_0}(3\sqrt{2})^{\alpha}\sum_{k = 1}^N\mu(\Delta_k)\\ &\le&  \frac{c_1}{c_0}(3\sqrt{2})^{\alpha}\mu(\Delta)\\ &=& \kappa \,\mu(\Delta) \le \kappa \,\mu(\Delta^*) < \infty,
\end{eqnarray*}where $\kappa := \frac{c_1}{c_0}(3\sqrt{2})^{\alpha}$ .\\\\
 Note that if $\Delta_k$'s are not centred in the support of $\mu$, then \eqref{newco} may fail. Indeed, consider the standard ternary Cantor set
$$
\mathcal{C} =  [0, 1]\setminus \bigcup_{n =1}^{\infty}\bigcup_{j = 1}^{2^{n - 1}}I_{n,j}\,,
$$
where $I_{n,j}$'s are the ``middle third'' intervals of the length $\frac{1}{3^n}$. Let $\mu$ be the Hausdorff measure of dimension $\alpha = \frac{\ln 2}{\ln 3}$ supported by $\mathcal{C}$. Let $\Delta_{n,j}$ be the closed square with the middle line $I_{n,j}$ and $\Delta$ be the closed square with the middle line $[0, 1]$. It is clear that $\Delta_{n,j}$'s are pairwise disjoint and
$$
\bigcup_{n =1}^{\infty}\bigcup_{j = 1}^{2^{n - 1}}\Delta_{n,j} \subset \Delta .
$$
On the other hand,
$$
\mathcal{C} \subset\bigcup_{j = 1}^{2^{n - 1}}\Delta^*_{n,j}\,, \,\,\,\,\,\forall n\in\mathbb{N}.
$$
Hence
$$
\sum_{n = 1}^{\infty}\sum_{j = 1}^{2^{n -1}}\mu(\Delta^*_{n,j}) \ge \sum_{n = 1}^{\infty} \mu(\mathcal{C}) = \infty,
$$
while
$$
\mu(\Delta^*) = \mu(\mathcal{C}) < \infty .
$$


\begin{thebibliography}{1}
\bibitem{Ad} R.A. Adams,
{\em Sobolev Spaces.} Academic Press, New York, 1975.
\bibitem{Aize} M. Aizenman and B. Simon, Brownian motion and Harnack's inequality for Schr\"odinger operators, {\it Comm. Pure Appl. Math.}, \textbf{35} (1982), 209---273.
\bibitem{BE} A. A. Balinsky and W. D. Evans, {\em Spectral Analysis of Relativistic Operators.} Imperial College Press, London, 2011.
\bibitem{BEL} A. A. Balinsky, W. D. Evans and R. T. Lewis, {\em The Analysis and Geometry of Hardy's Inequality.} Universitext, Springer, Cham, 2015.
\bibitem{Bal} A. A. Balinsky, W. D. Evans and R. T. Lewis, On the number of negative eigenvalues of Schr\"odinger operators with Ahnaranov-Bohm magnetic field. {\it Proceedings of the Royal Society A: Mathematical, Physical and Engineering Sciences}, \textbf{457}, 2014 (2001), 2481--2489.
\bibitem{BerShu} F.A. Berezin and M.A. Shubin,
{\em The Schr\"odinger Equation.} Kluwer, Dordrecht etc., 1991.
\bibitem{BirLap} M.Sh. Birman and A. Laptev, The negative discrete spectrum of a two-dimensional Schr\"odinger operator, {\it Commun. Pure Appl. Math} \textbf{49}, 9(1996), 967--1991.
\bibitem{BirSol} M.Sh. Birman and M.Z. Solomyak,
{\em Spectral Theory of Self-Adjoint Operators in Hilbert Space.} Kluwer, Dordrecht etc., 1987.
\bibitem{BS} M.Sh. Birman and  M.Z. Solomyak,
Estimates for the number of negative eigenvalues of the Schr\"odinger operator and
its generalizations.
In:  {\em Estimates and Asymptotics for Discrete Spectra of Integral and Differential Equations.}
Adv. Sov. Math. \textbf{7} (1991), 1--55.
\bibitem{DEB} N. G. De Bruijn, {\em Asymptotic Methods in Analysis.} North-Holland Publishing Company - Amsterdam,  1970.
\bibitem{Calo} F. Calogero, Upper and lower limits for the number of bound states in a given central potential, {\it Comm. Math. Phys.} \textbf{1}(1965), 80--88.
\bibitem{Chad} K. Chadan, N.N. Khuri, A. Martin and T.T. Wu, Bound states in one and two spatial dimensions, {\it J. Math. Phy.}, \textbf{44}, 2 (2003), 406--422.

\bibitem{Cia} A. Cianchi,
Moser--Trudinger trace inequalities.
{\it Adv. Math.} \textbf{217}, 5 (2008), 2005--2044.
\bibitem{Cohn} J.H.E. Cohn, On the number of negatives eigenvalues of a singular boundary value problem. {\it J. London Math. Soc.} \textbf{40}(1965), 523--525.
\bibitem{Dav} G. David and S. Semmens, {\em Fractured Fractals and Broken Dreams.} Clarendon Press, Oxford, 1997.
\bibitem{EBD} E. B. Davies,
{\em Spectral Theory and Differential Operators.}  Cambridge University Press, Cambridge, 1995.

\bibitem{Grig} A. Grigor'yan and N. Nadirashvili, Negative eigenvalues of two-dimensional
Schr\"odinger operators, {\it Arch. Rational Mech. Analy.} \textbf{217} (2015), 975--1028.
\bibitem{GNY} A. Grigor'yan, Yu. Netrusov  and S.-T. Yau, Eigenvalues of elliptic operators and geometric applications, {\it Surveys in Differential Geometry} \textbf{IX} (2004), 147--218.

\bibitem{Gri} P. Grisvard,
Commutativit\'e de deux foncteurs d'interpolation et applications.
{\it J. Math. Pures Appl.} \textbf{IX}, S\'er. 45 (1966), 143--290.
\bibitem{Guz} M. de Guzm\'an,
{\em Differentiation of Integrals in $\mathbb{R}^n$. }
Springer, Berlin--Heidelberg--New York, 1975.

\bibitem{HLP} G.H. Hardy, J.E. Littlewood, and G. P\'olya,
{\em Inequalities.}  Cambridge University Press, Cambridge, 1988.
\bibitem{Her} J. Herczynski, On Schr\"odinger operators with a distributional potential, {\it J. Operator Theory}, \textbf{21} (1989), 273--295.
\bibitem{Hilb} R. Courant and D. Hilbert, {\em Methods in Mathematical Physics.} Vol. 1, Interscience Publishers. Inc., New York, 1966.
\bibitem{HUT} J. E. Hutchinson, Fractals and self similarity, {\it Indiana University Mathematics Journal} \textbf{30} (1981), 713--747.
\bibitem{Kov} H. Kova\v{r}\'{\i}k, Eigenvalue bounds for two-deimensional magnetic Sch\"odinger operators, {\it J. Spectr. Theory}, \textbf{1} 4 (2011) 363--387.
\bibitem{KMW} N.N. Khuri, A. Martin and T.T. Wu, Bound states in $n$ dimensions (especially n = 1 and n = 2), {\it Few Body Syst.}, \textbf{31} (2002), 83--89.

\bibitem{KR} M.A. Krasnosel'skii and Ya.B. Rutickii,
{\em Convex Functions and Orlicz Spaces.} P. Noordhoff,
Groningen, 1961.
\bibitem{KK} D. Krej\v{c}i\v{r}\'{\i}k and J. K\v{r}\'{\i}\v{z}, On the spectrum of curved planar waveguides, {\it Publ. RIMS, Kyoto Univ.}, \textbf{41} (2005), 757--791.
\bibitem{AL} A. Laptev, The negative spectrum of a class of two-dimensional Schr\"odinger operators with spherically symmetric potentials, {\it Funct. Anal. and its Appl.}, \textbf{34}, 4(2000), 85--87.
\bibitem{LapNest} A. Laptev and Yu. Netrusov,
On the negative eigenvalues of a class of Schr\"odinger operators. In:
V. Buslaev (ed.) et al., {\em Differential Operators and Spectral Theory.
M. Sh. Birman's 70th anniversary collection.} Providence, RI: Transl., Ser. 2, Am. Math. Soc. \textbf{189(41)} (1999), 173--186.

\bibitem{LapSolo} A. Laptev and M. Solomyak,  On spectral estimates for two-dimensional Schr\"odinger operators,
{\it J. Spectr. Theory} \textbf{3}, 4 (2013), 505--515.
\bibitem{Lax} P. D. Lax, {\em Functional Analysis.} John Wiley and sons, Inc., New york, 2002.
\bibitem{Maz} V.G. Maz'ya,
{\em Sobolev Spaces. With Applications to Elliptic Partial Differential Equations.}
Springer, Berlin--Heidelberg, 2011.
\bibitem{MS} V.G. Maz'ya and T.O. Shaposhnikova,
{\em Theory of Sobolev Multipliers. With Applications to Differential and Integral Operators.}
Springer, Berlin--Heidelberg, 2009.
\bibitem{MR} M. Renardy and R. C. Rogers {\em An Introduction to Partial Differential Equations}. Springer-Verlag, Berlin, 1992.
\bibitem{FM} F. Morgan {\em Geometric Measure Theory}. 2nd ed., Academic Press, New York, 1995.
\bibitem{Peet} J. Peetre,
Espaces d'interpolation et th\'eor\`eme de Soboleff,
{\it Ann. Inst. Fourier} \textbf{16}, 1 (1966), 279--317.
\bibitem{RR} M.M. Rao and Z.D. Ren,
{\em Theory of Orlicz Spaces.}
Marcel Dekker, New York, 1991.


\bibitem{Reed} A. Reed and B. Simon,
\textit{Methods in Mathematical Physics, I. Functional analysis}, Academic Press, New York, 1972.
\bibitem{Roz}G. V. Rozemblum, The distribution of the discrete spectrum for singular differential operators, {\it Dokl. Akad. Nauk SSSR} \textbf{202} (1972), 1012--1015.
\bibitem{Rudin} W. Rudin, \textit{Real and Complex Analysis}, International student ed., McGraw-Hill Inc., London, 1970.
\bibitem{Eugene1}E. Shargorodsky,
An estimate for the Morse index of a Stokes wave, {\it Arch. Rational Mech. Anal.} \textbf{209}, 1 (2013), 41--59.
\bibitem{Eugene}E. Shargorodsky,
On negative eigenvalues of two-dimensional Schr\"odingers operators, Proceedings LMS, \textbf{108}, 2 (2014), 441--483.
\bibitem{Sol} M. Solomyak,
Piecewise-polynomial approximation of functions from $H\sp \ell((0,1)\sp d)$, $2\ell=d$, and applications to the spectral theory of the Schr\"odinger operator,
{\it Isr. J. Math.}  \textbf{86}, 1-3 (1994), 253--275.

\bibitem{Sol2} M. Solomyak,
On a class of spectral problems on the half-line and their applications to multi-dimensional problems,
{\it J. Spectr. Theory} \textbf{3}, 2 (2013), 215--235.
\bibitem{Stein} E. M. Stein, {\em Singular Integrals and Differentiability Properties of Functions.} Princeton University Press, New Jersey, 1970.
\bibitem{Herc} P. Stollmann and J. Voigt, Pertubation of Dirichlet forms by measures, {\it Potential Analysis} \textbf{5 }(1996), 109--138.
\bibitem{STR} R. S. Strichartz, Fractals in the large, {\it Can. J. Math} \textbf{50}, 3 (1996), 638--657 .
\bibitem{strok} D. W. Stroock, \textit{A Concise Introduction to the Theory of Integration.} Birkh\"auser, Berlin, 1994.
\bibitem{Teschl}G. Teschl, \textit{Mathematical Methods in Quantum Mechanics. With Applications to Schr\"odinger Operators.} American Mathematical Society, Providence, 2009.

\bibitem{We}J. Weidmann, \textit{Linear Operators in Hilbert Spaces.} Springer, Heidelberg, 1980.
\end{thebibliography}
\end{document}